\documentclass{article}

\usepackage[preprint]{neurips_2026}

\usepackage[utf8]{inputenc} % allow utf-8 input
\usepackage[T1]{fontenc}    % use 8-bit T1 fonts
\usepackage{hyperref}       % hyperlinks
\usepackage{url}            % simple URL typesetting
\usepackage{booktabs}       % professional-quality tables
\usepackage{amsfonts}       % blackboard math symbols
\usepackage{nicefrac}       % compact symbols for 1/2, etc.
\usepackage{microtype}      % microtypography
\usepackage{xcolor}         % colors
\usepackage{import}
\usepackage{enumitem}

\usepackage{tikz}
\usetikzlibrary{arrows.meta,positioning,calc}

\usepackage{multicol}
\usepackage{microtype}
\usepackage{graphicx}
\usepackage{booktabs} % for professional tables

\usepackage{wrapfig}

\RequirePackage{fancyhdr}
\RequirePackage{xcolor} % changed from color to xcolor (2021/11/24)
\RequirePackage{algorithm}
\RequirePackage{algorithmic}
\RequirePackage{natbib}
\RequirePackage{eso-pic} % used by \AddToShipoutPicture
\RequirePackage{forloop}
\RequirePackage{url}

\usepackage{amsmath}
\usepackage{amssymb}
\usepackage{mathtools}
\usepackage{amsthm}
\usepackage{bbm}

\usepackage[capitalize,noabbrev]{cleveref}

\usepackage{amsmath,amssymb}
\usepackage{booktabs}
\usepackage{multirow}
\usepackage{graphicx}
\graphicspath{{figures/}}
\usepackage{subcaption}
\usepackage{hyperref}
\usepackage{xcolor}
\usepackage{microtype}

\theoremstyle{plain}
\newtheorem{theorem}{Theorem}[section]
\newtheorem{proposition}[theorem]{Proposition}
\newtheorem{lemma}[theorem]{Lemma}
\newtheorem{corollary}[theorem]{Corollary}
\theoremstyle{definition}
\newtheorem{definition}[theorem]{Definition}
\newtheorem{assumption}[theorem]{Assumption}

\newtheoremstyle{boldremark}
  {3pt}{3pt}        % space above/below
  {}                % body font
  {}                % indent
  {\bfseries}       % head font (bold)
  {.}               % punctuation after head
  { }               % space after head
  {}                % head spec

\theoremstyle{boldremark}
\newtheorem{remark}{Remark}

\newcommand{\R}{\mathbb{R}}

\newcommand{\Pcal}{\mathcal{P}}
\newcommand{\cP}{\mathcal{P}}
\newcommand{\E}{\mathbb{E}}
\newcommand{\Prob}{\mathbb{P}}

\newcommand{\norm}[1]{\left\lVert #1 \right\rVert}

\newcommand{\dist}{\operatorname{dist}}
\newcommand{\dd}{\,\mathrm{d}}
\newcommand{\Id}{I_d}

\newcommand{\Law}{\operatorname{Law}}
\newcommand{\Tr}{\operatorname{Tr}}

\newcommand{\thetae}{\widetilde\theta}

\newcommand{\softJ}{J_{\beta,\tau}^L}
\newcommand{\qsoft}{q_{\beta,\tau}^L}
\newcommand{\Gtilde}{\widetilde G}

\newcommand{\abs}[1]{\left\lvert #1\right\rvert}
\newcommand{\br}[1]{\left(#1\right)}
\newcommand{\sqbr}[1]{\left[#1\right]}

\usepackage[textsize=tiny]{todonotes}

\title{
Convergence of Consensus-Based Particle Methods for Nonconvex Bi-Level Optimization
}

\author{%
  Yutong Chao \\
  Department of Computer Science\\
Technical University of Munich\\
\texttt{yutong.chao@tum.de}\\
\And
Xudong Sun \\
Munich Institute of Robotics and Machine Intelligence \\
Technical University of Munich \\
\texttt{xudong.sun@tum.de} \\
\And
Konstantin Riedl \\
Mathematical Institute\\
University of Oxford\\
\texttt{konstantin.riedl@maths.ox.ac.uk} \\
\And
Majid Khadiv \\
Munich Institute of Robotics and Machine Intelligence \\
Technical University of Munich \\
\texttt{majid.khadiv@tum.de} \\
\And
Jalal Etesami \\
Department of Computer Science\\
Technical University of Munich\\
\texttt{j.etesami@tum.de} \\
}

\begin{document}

\maketitle

\begin{abstract}
%In this paper, we study a consensus-based particle method for nonconvex bi-level optimization, where the goal is to minimize an upper-level objective over the set of global minimizers of a lower-level objective. The method is derivative-free and computes its consensus point through smooth quantile selection and a Gibbs-type Laplace approximation. We provide theoretical convergence guarantee for the corresponding mean-field dynamics and for its finite-particle numerical approximation. The proof shows that, under suitable smooth quantile localization, error-bound, and stability assumptions, the mean-field law converges exponentially fast in Wasserstein distance to the global solution of the bi-level problem. We further validate our result on several numerical experiments.
In this paper, we study a consensus-based optimization method for nonconvex bi-level optimization, where the objective is to minimize an upper-level function over the set of global minimizers of a lower-level problem. The proposed approach is derivative-free,  and constructs its consensus point via smooth quantile selection combined with a Gibbs-type Laplace approximation. We establish convergence guarantees for both the associated \textit{mean-field} dynamics and its \textit{finite-particle} approximation. 
In particular, under suitable assumptions on smooth quantile localization, error bounds, and stability, we show that the mean-field law reaches any arbitrary prescribed Wasserstein neighborhood of the target bi-level solution with an explicit exponential rate up to the hitting time. Numerical experiments on a two-dimensional constrained problem and neural network training further support the theoretical results.

\end{abstract}

\section{Introduction}\label{sec:intro}

Bi-level optimization provides a natural framework for hierarchical decision-making problems in which an upper-level objective is optimized subject to the solution of a lower-level optimization problem. It has become a central tool in modern
machine learning, with applications in 
meta-learning~\cite{franceschi2018bilevel,kim2025stochastic}, hyperparameter optimization \cite{mackay2019self}, neural architecture search~\cite{xue2021rethinking}, reinforcement
learning~\cite{gammelli2023graph,10.5555/3618408.3618835}, adversarial robustness \cite{zhang2022revisiting}, model pruning~\cite{zhang2022revisiting}, etc. 
In this paper, we study the bi-level optimization problem
\begin{align}\label{eq:bi-problem}
    \theta^\star \in \arg\min_{\theta\in\Theta} G(\theta),
    \qquad
    \Theta := \arg\min_{\theta\in\mathbb R^d} L(\theta),   
\end{align}
where \(L,G:\mathbb R^d\to\mathbb R\) are the nonconvex lower- and upper-level objectives, respectively. 
This formulation captures situations in which many candidates may solve the lower-level task equally well, while the upper-level objective selects among them according to an
additional criterion, such as robustness, sparsity or
generalization~\cite{petrulionyte2024functional,zhang2024introduction}.

Despite its simple form, this problem is challenging when the objectives are
nonconvex. The feasible set of the upper-level problem is the solution set
\(\Theta\) of the lower-level problem, which is typically defined only implicitly.
Consequently, it is generally difficult to project onto it, parametrize
it, or optimize over it directly~\cite{trillos2024cb}. This difficulty becomes
more pronounced when the lower-level objective \(L\) is nonconvex, since
\(\Theta\) may be nonconvex, disconnected, or have a complicated geometry. At the
same time, the upper-level objective \(G\) may also be nonconvex, as is common in
modern machine learning and engineering applications~\cite{fornasier2021consensus,bae2022constrained}. Therefore, the
problem considered in this paper is not merely a constrained optimization problem
with a nonconvex feasible set; rather, both the lower-level landscape defining
the feasible set and the upper-level landscape used for selection may be
nonconvex.

Most existing convergence analyses for bi-level optimization that are gradient-based impose structural conditions that simplify this geometry. A standard assumption is that the lower-level objective is strongly convex, which guarantees the uniqueness of the lower-level solution and enables implicit differentiation  \cite{shen2023online,liang2023lower,cao2024accelerated}. 
%{\color{blue} When the lower-level is convex, then $\Theta$ is a point instead of a set. Then the upper-level optimization is meaningless. Unless you change the setting into something like the ones for multi-level optimization paper with gradient-based solver. I think you mention the gradient-based methods here because next paragraph you start talking about derivative-free . }
Other approaches replace strong convexity with convexity, error-bound conditions, or Polyak--Lojasiewicz (PL)-type conditions~\cite{xiao2023alternating,huang2024optimal}. 
While these assumptions lead to tractable algorithms with convergence guarantees, they exclude many problems in which \(\Theta\) is a nonconvex set. 
In particular, such assumptions impose a regular global geometry that is not appropriate for many overparameterized or nonsmooth optimization landscapes \cite{du2019gradient,cooper2021global,islamov2024loss}. This motivates the development of methods that can handle nonconvex \(L\) and \(G\).
%, and nonconvex lower-level solution sets without relying on convexity or PL-type geometry.

Consensus-based optimization (CBO) offers a derivative-free alternative for global optimization in nonconvex landscapes~\cite{totzeck2021trends}. Instead of following
gradients, CBO evolves a system of interacting particles that explore the space and communicate through a consensus mechanism. Such methods have been studied at the level of mean-field dynamics~\cite{fornasier2022convergence} and finite-particle
approximations~\cite{beddrich2024constrained}. A key theoretical breakthrough was provided by \cite{fornasier2024consensus}, which shows that CBO can converge globally for a broad class of nonconvex and nonsmooth objective functions in the mean-field and discrete particle regime. 
%Their analysi reveals that the consensus mechanism, together with a Laplace-principle-based weighting, induces an effective global selection mechanism rather than merely a local descent dynamics. 
Their derivative-free nature makes them particularly appealing for bi-level problems in which both objectives may be nonconvex.
%and where the set \(\Theta\) is not explicitly available and has complex geometry.
% A recent line of work extends the CBO philosophy to nonconvex bi-level
% optimization~\cite{}. The key idea is to construct the consensus point in two
% stages: first, particles with favorable lower-level objective values are selected
% through a quantile rule; second, the selected particles are reweighted by a
% Gibbs-type measure associated with the upper-level objective. This mechanism is
% well aligned with the structure of the bi-level problem. The lower-level
% quantile selection localizes particles near the solution set \(\Theta\), while
% the upper-level Gibbs weighting favors candidates with smaller values of \(G\)
% among those lower-level candidates~\cite{}. However, the existing method uses a
% hard quantile selection rule. Although this rule is intrinsic and invariant under
% monotone transformations of the lower-level objective, it makes the consensus
% point non-smooth with respect to perturbations of the particle distribution.
% This non-smoothness is a major obstacle to establishing stability estimates for
% the interacting particle system. As a result, existing theory proves convergence
% of the associated mean-field law, but does not provide a high-probability
% finite-particle convergence guarantee for the algorithm implemented with
% finitely many particles~\cite{}.
Recently, CBO has been adopted to solve bi-level problems of the form~\eqref{eq:bi-problem} \cite{trillos2024cb}. However, the method proposed therein, called CB$^2$O, relies on a hard quantile-based consensus mapping that is discontinuous (Section \ref{sec:cb2o}). As a consequence, convergence guarantees beyond the mean-field approximation  are generally not available.

In this work, we instead introduce a soft-quantile consensus-based particle method (Section \ref{sec:cb2o}) that maintains strong performance (Section \ref{sec:experiment}). Given a particle distribution $\rho$, we define a soft lower-level quantile via a smooth selection function (Equation \eqref{eq:soft_quantile_def}). The consensus point is then obtained by applying a Gibbs-type (Laplace) weighting to this measure with respect to the upper-level objective. This construction preserves the two-stage structure of the bi-level problem: the soft-quantile selection identifies particles close to the lower-level solution set, while the Gibbs weighting promotes favorable upper-level candidates. At the same time, the smooth selection mechanism provides sufficient regularity to analyze both the mean-field dynamics and the finite-particle interacting system.
Our main contributions are as follows.

\begin{itemize}[leftmargin=*, align=left]
    \item We propose a consensus-based particle method, called SCB$^2$O, for nonconvex bi-level optimization problems. Unlike many existing convergence analyses for bi-level optimization, our framework does not require the lower-level solution set \(\Theta\) to be convex, nor does it require either the lower-level  or the upper-level objective to be convex or to satisfy a PL condition. Instead, the method relies on a soft-quantile mechanism to identify particles near the lower-level solution set and a weighting to select favorable upper-level candidates.

    \item We provide a quantitative convergence analysis for the finite-particle dynamics (Section \ref{sec:results}). In particular, we establish a localized high-probability finite-particle approximation bound for the particle estimator, obtained by combining moment-dependent consensus bounds, a cutoff-event propagation-of-chaos estimate, and finite-sample concentration. This extends existing mean-field convergence results for CBO-based bi-level optimization.

    %While prior work such as \emph{CB2O: Consensus-Based Bi-Level Optimization} establishes convergence of the mean-field law, its use of a hard-quantile mechanism prevents the derivation of finite-particle convergence guarantees for the interacting particle system.
\end{itemize}

\vspace{-.3cm}
\subsection{Related Work}
\vspace{-.1cm}

\textbf{Bi-level optimization.}
%It aims to minimize an upper-level objective over the solution set of a lower-level optimization problem, and has found broad applications in machine learning, including hyperparameter optimization \cite{mackay2019self}. 
A large body of work has studied convergence guarantees for bi-level optimization
under various curvature and smoothness assumptions. Most existing
analyses impose strong structural conditions on the lower-level problem to simplify the geometry of the lower-level solution set \cite{shen2023online,liang2023lower,cao2024accelerated}. 
%These assumptions substantially simplify the geometry of the lower-level solution set: strong convexity implies uniqueness of the lower-level minimizer, while PL-type conditions often enforce a highly regular solution landscape. 
Consequently, such settings do not fully capture bi-level problems whose lower-level solution set
$
\Theta 
$
is nonconvex, disconnected, or geometrically complex.

Several recent works have attempted to relax these requirements. One line of work weakens the global PL condition to a local PL condition near the lower-level solution~\cite{masiha2025superquantile}. Under suitable regularity assumptions, this allows \(\Theta\) to be characterized as a \(\mathcal C^2\) embedded submanifold of the ambient space \(\mathbb R^d\) without boundary. 
% Such results provide a more flexible geometric description of the lower-level solution set and motivate
% Langevin-type algorithms for bi-level optimization. 
However, the resulting convergence guarantees typically characterize convergence toward an invariant distribution and do not directly quantify the distance between the algorithmic output and the true bi-level solution.
Another line of work reformulates the bi-level problem as a single-level unconstrained problem by combining the lower-level and upper-level objectives through an adaptive penalty or weighting parameter \cite{cabot2005proximal, dutta2020algorithms}. This approach is algorithmically attractive because it avoids explicitly solving the lower-level problem at each iteration. Nevertheless, the theoretical guarantees for such reformulations usually rely on convexity or strong regularity assumptions to ensure that the penalized objective faithfully represents the original bi-level structure.

A related class of methods approximates the lower-level solution set through sublevel-set relaxations and then applies classical gradient-based methods to the resulting constrained or regularized problem \cite{beck2014first,jiang2023conditional}. These approaches can provide stability and convergence guarantees when the lower-level objective is convex or when the sublevel sets have favorable geometry. However, when \(\Theta\) is a general nonconvex set, sublevel-set approximations may be loose or geometrically misleading.
In contrast, we impose no convexity assumptions on the lower-level solution set $\Theta$ or on the objectives, nor do we require a PL condition, and we establish high-probability convergence guarantees for the finite-particle system.
%Instead, we develop a consensus-based particle dynamics that combines soft quantile selection for lower-level localization with a Gibbs-type upper-level selection mechanism, and we establish high-probability convergence guarantees for the resulting finite-particle system.

\textbf{Consensus-based optimization (CBO).} It is a multi-agent, derivative-free optimization framework originally introduced in \cite{pinnau2017consensus}. 
It has demonstrated strong empirical performance on challenging nonconvex and nonsmooth optimization problems. Due to its success in practice, it has attracted significant theoretical interest, and a growing body of work has established convergence guarantees for CBO both at the level of the mean-field regime \cite{borghi2023constrained,huang2024consensus} and in the finite-particle regime~\cite{fornasier2021consensus,fornasier2024consensus}.
These methods have also been applied to various practical problems, including phase retrieval \cite{carrillo2021consensus}, robust subspace detection \cite{carrillo2023consensus}, robotics \cite{sun2026consensus}, federated learning \cite{carrillo2024fedcbo}, and adversarial attacks \cite{roith2025consensus,garcia2025defending}.

The CBO framework has been extended in several directions beyond standard single-objective optimization. For instance, CBO-type methods have been developed for multi-objective optimization \cite{herty2025multiscale}, random mini-batch sampling-based
optimization \cite{carrillo2022consensus}, zero-sum game \cite{huang2026well} and multiplayer game \cite{chenchene2025consensus}. These extensions
illustrate the flexibility of consensus-based particle dynamics and their ability to incorporate different structural features of modern optimization problems. Among these developments, the work \cite{trillos2024cb} proposes a CBO-type algorithm for nonconvex bi-level optimization. 
Their method employs a consensus mapping based on a hard selection rule defined via a quantile of the lower-level objective values. The authors establish convergence guarantees for the corresponding mean-field dynamics. However, they also show that, due to discontinuity of their consensus mapping, standard Lipschitz-based propagation-of-chaos and numerical stability arguments do not directly apply to the hard-threshold consensus map.
Here, we address this limitation by introducing a novel soft consensus mapping and establishing high-probability convergence guarantees for the finite-particle system.

%In contrast, our method replaces the hard lower-level selection rule with a soft selection mechanism based on the lower-level objective value. This modification preserves a smoother dependence of the consensus point on the underlying particle distribution and enables a quantitative analysis of both the mean-field dynamics and the finite-particle system. In particular, we provide high-probability finite-particle convergence guarantees, which are not established in the existing CBO-based bi-level optimization framework.

\section{Problem Setting}\label{sec:problem}
\vspace{-.1cm}

We consider the bi-level optimization problem in \eqref{eq:bi-problem}, for which, we denote the optimum value of the lower-level by $L_{\min}:=\inf_{\theta\in\R^d}L(\theta)$ and introduce the following notations for a given $r>0$, $N_r(\Theta):=\big\{\theta\in\R^d:\dist(\theta,\Theta)\le r\big\}$ and $B_r(\theta^\star):=\big\{\theta\in\R^d:\norm{\theta-\theta^\star}_2\leq r\big\}$. 
We also denote by \(\mathcal{P}(\mathbb{R}^d)\) the set of probability measures on \(\mathbb{R}^d\), and by \(\mathcal{P}_k(\mathbb{R}^d)\) the subset consisting of those with finite \(k\)-th moments. We use $[N]:=\{1,\cdots,N\}$.

\subsection{Particle Systems and Diffusion Structure}
\vspace{-.1cm}
Suppose we are interested in minimizing the upper-level objective \(G(\theta)\) without considering any lower-level problem, i.e., 
$
\arg\min_{\theta \in \mathbb{R}^d} G(\theta).
$
In this case, the central idea of CBO is to introduce a population of agents (particles) that explore the space \(\mathbb{R}^d\) according to a prescribed stochastic dynamics. These agents are designed to interact in a way that gradually concentrates their distribution, ultimately leading them to reach consensus at the global minimizer  of \(G\).
More precisely, let $\{\Theta^1, \ldots, \Theta^N\}$ denote a finite collection of agents that explore the domain and collectively form a global consensus on the solution. 
Each agent $i \in [N]$ is modeled as a stochastic process, i.e., $\Theta^i = (\Theta^i_t)_{t \geq 0}$ with $\Theta^i_t \in \mathbb{R}^d$, whose time evolution is governed by the following SDE:
\begin{equation}\label{eq:dynamic_cbo}
    \Theta_0^i\sim \rho_0, \quad  \dd\Theta^{i}_{t}
    =-\lambda\cdot\big(\Theta^{i}_{t}-m(\rho^{N}_{t})\big)\dd t
    +\sigma D\big(\Theta^{i}_{t}-m(\rho^{N}_{t})\big)\dd W^i_t,
\end{equation}
where $\lambda, \sigma>0$ are user-specific parameters and $\{(W_t^i)_{t\geq0}\}_{i=1}^N$ are independent standard Brownian motions in $\mathbb R^d$.
The agents are initially distributed according to a common law $\rho_0$, and their empirical distribution at time $t$ is given by $\rho^{N}_{t}:= \tfrac{1}{N}\sum_i \delta_{\Theta_t^i}\in\mathcal{P}(\R^d)$. The mapping $m(\cdot)$ is a functional that computes a consensus point among the agents in $\mathbb R^d$ from their distribution.  
In this SDE, the first term is the \textit{drift} towards the instantaneous consensus $m(\rho^{N}_{t})$, which is computed as a weighted average of all agents and serves as an instantaneous proxy for the global minimizer $\theta^\star$,
%\begin{small}
\begin{equation}
    m(\rho^{N}_{t}):=\frac{\int_{\R^d}x\exp({-\alpha G(x)})\,\dd \rho^{N}_{t}(x)}{\int_{\R^d}\exp({-\alpha G(x)})\,\dd \rho_t^N(x)}=\frac{\sum\limits_{i=1}^{N}
\Theta_{t}^{i}\,\exp({-\alpha G(\Theta_{t}^{i}}))}{
\sum\limits_{j=1}^{N}
\exp({-\alpha G(\Theta_t^{j})})}.
\end{equation}    
%\end{small}
%where $\mathcal Z(\rho):=\int_{\R^d}\exp({-\alpha G(x)})\,\dd \rho(x)$, for a given distribution $\rho$.
\noindent Such a consensus is motivated by the well-known Laplace principle \cite{miller2006applied} and can be interpreted as an approximation of the discrete minimization problem
$
\arg\min_{i \in [N]} G(\Theta_t^i),
$
which becomes increasingly accurate as $\alpha \to \infty$, provided that a minimizer exists. 
% which states that, for any absolutely continuous probability distribution $\rho$ on $\mathbb R^d$, we have
% \begin{align*}
%     \lim_{\alpha\to \infty}\frac{1}{\alpha}\log\Big(\mathbb E_{x\sim\rho}\big[\exp\big(-\alpha G(x)\big)\big]\Big)=\lim_{\alpha\to \infty}\frac{1}{\alpha}\log\int \exp\big(-\alpha G(x)\big) d\rho(x) =-\inf_{x\in supp(\rho)}G(x).
% \end{align*}
The second term randomly diffuses agents according to a scaled Brownian motion and offers exploration over the domain. For a given $z \in \mathbb{R}^d$, two commonly studied diffusion types are the isotropic diffusion $D(z)=\|z\|_2 I_d$ \cite{pinnau2017consensus,carrillo2018analytical} and the anisotropic diffusion $D(z)=\mathrm{diag}(z)$ \cite{carrillo2021consensus}. In this paper, we adopt the isotropic diffusion, thus,
\begin{equation}
    D(z)D(z)^\top=\norm{z}_2^2\Id,
    \quad
    \Tr\big(D(z)D(z)^\top\big)=d\norm{z}_2^2,
    \quad
    \norm{D(z)-D(y)}_F\le \sqrt d\norm{z-y}_2.
    \label{eq:D-basic}
\end{equation}
Ideally, as a result of the above mechanism, the agents should eventually reach a near-optimal global consensus, i.e., $\rho_t^N$ is expected to converge over time to a Dirac mass concentrated at $\arg\min_{\theta} G(\theta)$.

\subsection{Soft Consensus-Based Bi-Level Optimization (SCB$^2$O)}\label{sec:cb2o}
\vspace{-.1cm}
In this section, we introduce a formulation of CBO for solving the bi-level optimization problem in \eqref{eq:bi-problem}. 
We build on the framework developed by \cite{trillos2024cb}, referred to as CB$^2$O, and introduce a new method, called SCB$^2$O, for which, in contrast to CB$^2$O, we are able to provide theoretical guarantees for the finite-particle system. 
The central idea is to consider an interacting system of particles, analogous to the original CBO framework \cite{carrillo2023consensus,pinnau2017consensus}, which explores the domain $\mathbb{R}^d$. However, unlike standard CBO, both CB$^2$O and our SCB$^2$O compute the consensus among the particles at time $t$ as follows. 

CB$^2$O computes the consensus point $m(\rho_t^N)$ in \eqref{eq:dynamic_cbo} at time $t$, first by evaluating the lower-level objective function $L$ for each of the $N$ particles, i.e., $\{L(\Theta_t^i)\}_{i=1}^N$, and then sorting the resulting values in non-decreasing order, i.e.,
\begin{small}
$
    L(\Theta_t^{\pi(1)})\leq L(\Theta_t^{\pi(2)})\leq \cdots\leq L(\Theta_t^{\pi(N)}),
$
\end{small}
where $\Theta_t^{\pi(i)}$ denotes the position of the particle with the $i$-th smallest lower-level objective value at time $t$.
Next, CB$^2$O selects the subset of particles from the sorted list with smallest values and computes the consensus using this subset. More precisely, for a parameter $\beta \in (0,1]$, let $\lceil \beta N \rceil$ denote the ceiling of $\beta N$, and define the $\beta$-quantile of the set $\{L(\Theta_t^i)\}_{i=1}^N$ as
$
q^{L}_\beta[\rho_t^N] = L\big(\Theta_t^{\pi(\lceil \beta N \rceil)}\big).
$
In words, $q^{L}_\beta[\rho_t^N]$ corresponds to the value of the lower-level function $L$ evaluated at the particle whose rank is such that exactly $\lceil \beta N \rceil$ particles have smaller  lower-level objective values. The consensus point is then computed by
\begin{small}
\begin{equation}
m_{\alpha,\beta}^{G,L}(\rho_t^N):=\frac{\sum\limits_{i=1}^{\lceil \beta N \rceil}
\Theta_{t}^{\pi(i)}\,\exp({-\alpha G(\Theta_{t}^{\pi(i)})})}{
\sum\limits_{j=1}^{\lceil \beta N \rceil}
\exp({-\alpha G(\Theta_t^{\pi(j)})})}=
\frac{\sum\limits_{i:\,L(\Theta_{t}^i)\le q_\beta^L[\rho_t^N]}
\Theta_{t}^i\,\exp({-\alpha G(\Theta_{t}^i)})}{
\sum\limits_{j:\,L(\Theta_t^j)\le q_\beta^L[\rho_t^N]}
\exp({-\alpha G(\Theta_t^j)})}.
\end{equation}    
\end{small}
\noindent It is noteworthy that the above definition of consensus is based on hard thresholding, i.e., it ignores $N - \lceil \beta N \rceil$ particles. This introduces a discontinuity in the consensus mapping and, consequently, in the dynamics. As a result, standard Lipschitz-based propagation-of-chaos and numerical stability arguments do not directly apply to the finite-particle system. To highlight this issue, as also discussed in Section 3.2 of \cite{trillos2024cb}, there exist settings in which two measures $\rho$ and $\tilde{\rho}$ are close in the Wasserstein sense, i.e., $W_2(\rho,\tilde{\rho}) = c$, but
$
\| m_{\alpha,\beta}^{G,L}(\tilde{\rho}) - m_{\alpha,\beta}^{G,L}(\rho) \|_2 = c_\alpha,
$
where $c_\alpha$ is a constant that depends on the hyperparameter $\alpha$ but is independent of $c$.

In this work, we introduce a soft consensus mapping to overcome this limitation, allowing us to derive convergence guarantees in both the mean-field and more importantly, the finite-particle settings to find the solution to \eqref{eq:bi-problem}.
Formally, let $\psi(\cdot)$ denote a selection function that satisfies Assumption \ref{ass:selector}. 
We define the soft $\beta$-quantile of  $\rho\in\mathcal{P}(\R^d)$, $q^L_{\beta,\tau}[\rho]$, as the unique solution (see Lemma \ref{lem:soft-quantile-existence}) of
\begin{equation}\label{eq:soft_quantile_def}
    \int_{\R^d} \psi\left(\frac{q-L(\theta)}{\tau}\right)\dd\rho(\theta)=\beta.
\end{equation}

\begin{assumption}\label{ass:selector}
The selector $\psi:\mathbb{R}\to(0,1)$ is of class $C^1$, strictly increasing with \(\psi'(z)>0\) for every \(z\in\mathbb R\), and Lipschitz continuous; that is, there exists a constant $L_\psi>0$ such that
$|\psi(u)-\psi(v)| \leq L_\psi |u-v|,$ for all $u,v \in \mathbb{R}.$
Moreover, $\lim_{u\to\infty}\psi(u)=1$, $\lim_{u\to -\infty}\psi(u)=0$, and there exist constants $c_-,C_- > 0$ such that, for all $z \geq 0$,
    \begin{equation}
        c_-e^{-z}\le \psi(-z)\le C_-e^{-z}.
        \label{eq:s_left_tail}
    \end{equation}
\end{assumption}

\begin{remark}
An example of $\psi(\cdot)$ is the sigmoid (softmax) function which continuously approximates a step function. 
Note that, when $\tau\to 0$, $\psi(\cdot/\tau)$ behaves like a step function, and the above equation becomes finding the $q$ such that $ \int_{L(\theta)\le q} 1 \dd\rho(\theta) =\beta$. In this case, $q^L_{\beta,\tau}[\rho_t^N]$ is equivalent to the CB$^2$O's $\beta$-quantile, $q^{L}_\beta[\rho_t^N]$.
%In the general case of $\tau > 0$, the integration can be written as $\int_{L(\theta) \ge q} \psi^{(-)} \rho(d\theta) + \int_{L(\theta)\le q} \psi^{(+)}  \rho(d\theta) =\beta$
% {\color{red} To match this with the CB2O, dont you need the measure $\rho_t^N$, i.e., 
% $$
% \int_{L(\theta)\le q} 1 \dd\rho_t^N(\theta)=\frac{1}{N}\sum_{i:L(\theta^i_t)\le q}1 =\beta
% $$
% }
% {\color{blue} No. We rather have $
%     \int_{\mathbb R^d}J^L_{\beta,\tau}[\rho](\dd \theta)=\beta.
% $}
\end{remark}

% \begin{remark}
% $\frac{1}{\tau}$ can be understood as a temperature like hyperparameter controlling extent of this approximation. The relationship between $\alpha_L=\frac{1}{\tau}$ and $\alpha$ (see below) decides which softmax term dominates the other. The multiplication of two softmax can be visually thought of as amplitude modulation.
% \end{remark}

\begin{lemma}[Existence and uniqueness of the soft quantile]\label{lem:soft-quantile-existence}
Under Assumption~\ref{ass:selector}, for every $\rho\in\Pcal(\R^d)$, Equation \eqref{eq:soft_quantile_def} has a unique solution.
\end{lemma}
All detailed proofs are presented in Appendix \ref{app:proofs}.
When no confusion arises, we simplify the notation by writing \( q_\rho := q^L_{\beta,\tau}[\rho] \). Let
\begin{equation}
    \eta_{\rho,\tau}(\theta):=\psi\Big(\frac{q_\rho-L(\theta)}{\tau}\Big),
    \qquad
    J^L_{\beta,\tau}[\rho](\!\dd \theta):=\eta_{\rho,\tau}(\theta)\dd\rho( \theta).
\end{equation}
Note that by the definition of the soft quantile, we have
$
    \int_{\mathbb R^d}J^L_{\beta,\tau}[\rho](\!\dd\theta)=\beta.
$
Given the above definitions, we define our \textit{soft consensus mapping} for a measure $\rho$ as 
\begin{equation}\label{eq:soft-consensus}
   m^{G,L}_{\alpha,\beta,\tau}(\rho)=\frac{A(\rho)}{B(\rho)}:=\frac{\int_{\R^d}\theta w_\alpha(\theta)\,\dd J^L_{\beta,\tau}[\rho](\theta)}{\int_{\R^d}w_\alpha(\theta)\,\dd J^L_{\beta,\tau}[\rho](\theta)},
\end{equation}
where $\alpha>0$, $\beta\in(0,1)$, and $\tau>0$.
% , and 
% \begin{equation*}
%     A(\rho):=\int_{\R^d}\theta w_\alpha(\theta)\,\dd J^L_{\beta,\tau}[\rho](\theta),
%     \qquad
%     B(\rho):=\int_{\R^d}w_\alpha(\theta)\,\dd J^L_{\beta,\tau}[\rho](\theta).
% \end{equation*}
The interacting particle system with the above soft consensus mapping is given by
\begin{equation} \label{eq:interacting_sde}
    \dd\Theta^{i,N}_{t,\tau}
    =-\lambda\cdot \big(\Theta^{i,N}_{t,\tau}-m^{G,L}_{\alpha,\beta,\tau}(\rho^{N}_{t,\tau})\big)\dd t
    +\sigma D\big(\Theta^{i,N}_{t,\tau}-m^{G,L}_{\alpha,\beta,\tau}(\rho^{N}_{t,\tau})\big)\dd W^i_t, \quad i\in[N],
\end{equation}
where
$
    \rho^{N}_{t,\tau}:=\frac1N\sum_{i=1}^N\delta_{\Theta^{i,N}_{t,\tau}}.
$
Here, $\{(\Theta^{i,N}_{t,\tau})_{t\geq0}\}_{i=1}^N$ denote the $N$ agents (particles) at different times for a fixed parameter $\tau$. 
Next, we study the setting where the number of agents tends to infinity, called the mean-field regime.

\textbf{Mean-field Approximation.}
A theoretical analysis of the SCB$^2$O dynamics can be performed either at the microscopic level of the interacting particle system~\eqref{eq:interacting_sde}, e.g., \cite{bellavia2025discrete}, or, following recent works~\cite{fornasier2022convergence,grassi2023mean,carrillo2021consensus}, at the macroscopic level via the mean-field limit describing the evolution of the agent density, or through a combination of both approaches. The mean-field system is given by
\begin{equation}   \label{eq:mf-sde}
    \dd\Theta_{t,\tau}
    \!=\!-\!\lambda\cdot\big(\Theta_{t,\tau}\!-\!m^{G,L}_{\alpha,\beta,\tau}(\rho_{t,\tau})\big)\dd t
    +\sigma D\big(\Theta_{t,\tau}\!-\!m^{G,L}_{\alpha,\beta,\tau}(\rho_{t,\tau})\big)\dd W_t,
    \quad
    \rho_{t,\tau}:=\Law(\Theta_{t,\tau}).
\end{equation}
The mean-field approach enables the use of more flexible analytical tools to characterize the average behavior of the agents through the evolution of the measure $(\rho_{t,\tau})_{t \ge 0}$ which is described by the Fokker--Planck equation, 
\begin{align}\notag
\partial_t\rho_{t,\tau}
&=
\lambda\,\mathrm{div}\!\left(
\big(\theta\!-\!m^{G,L}_{\alpha,\beta,\tau}(\rho_{t,\tau})\big)\rho_{t,\tau}
\right)\\ \label{eq:fokker}
&+
\frac{\sigma^2}{2}
\!\sum_{k=1}^d
\partial_{kk}^2
\left(
[D(\theta-m^{G,L}_{\alpha,\beta,\tau}(\rho_{t,\tau}))
D(\theta-m^{G,L}_{\alpha,\beta,\tau}(\rho_{t,\tau}))^\top]_{kk}
\rho_{t,\tau}
\right).
\end{align}
For every \(N\in\mathbb N\), we assume that one can realize i.i.d. realizations $\bar\Theta^1_{t,\tau},\ldots,\bar\Theta^N_{t,\tau}$ 
on the same probability $\Law(\bar\Theta^i_{t,\tau})=\rho_{t,\tau}$, for $t\in[0,T]$, driven by the same
initial data and Brownian motions, i.e.,
$
\bar\Theta^i_{0,\tau}\!=\!\Theta^{i,N}_{0,\tau},
$
for $i\!\in\! [N]$.
% Let $\bar\Theta^1_{t,\tau}, \dots, \bar\Theta^N_{t,\tau}$ denote $N$ i.i.d. copies of the mean-field process with common law $\rho_{t,\tau}$, coupled synchronously with the above interacting particle system through the same initial conditions and Brownian motions. 
We can then define the corresponding empirical measure using the realizations by
$
\bar\rho^N_{t,\tau}\!:=\!\frac1N\!\sum_{i=1}^N\delta_{\bar\Theta^i_{t,\tau}}.
$
However, a purely mean-field analysis may be less representative in finite-agent settings. In this work, we provide both the mean-field and finite-particle analysis. 
%where the corresponding theoretical guarantees may no longer apply.

\textbf{Discrete Dynamic.}
The discrete dynamic of the above SDE is described by the Euler--Maruyama \cite{higham2001algorithmic} scheme with $K\Delta t=T$ and it is given by
\begin{equation}
    \theta^i_{k+1,\tau}
    =\theta^i_{k,\tau}
    -\lambda\Delta t\big(\theta^i_{k,\tau}-m^{G,L}_{\alpha,\beta,\tau}(\rho^N_{k,\tau})\big)
    +\sigma\sqrt{\Delta t}\,D\big(\theta^i_{k,\tau}-m^{G,L}_{\alpha,\beta,\tau}(\rho^N_{k,\tau})\big)B^i_k,\quad i\in[N],
    \label{eq:em_scheme}
\end{equation}
where $(B_k^i)_{i=1}^N$ are i.i.d. standard normal random vectors, each distributed as $\mathcal{N}(0, \mathrm{Id})$. The corresponding empirical measure is defined by
$
    \rho^N_{k,\tau}:=\frac1N\sum_{i=1}^N\delta_{\theta^i_{k,\tau}}.
$

Although this dynamic is the one used in practice to solve the problem in \eqref{eq:bi-problem}, no theoretical convergence guarantees had been established prior to this work. In this work, under a set of mild assumptions (discussed below) that are frequently satisfied in real-world settings, we prove its convergence.

%\paragraph{Assumptions.}
%Here, we present the conditions required to establish theoretical convergence guarantees and discuss their significance.

\begin{assumption}\label{ass:primitive}We assume that the bi-level solution $
\theta^\star \in \arg\min_{\theta\in\Theta}G(\theta)$
is unique and the following conditions hold.
\begin{enumerate}[leftmargin=*, align=left]%[leftmargin=18pt]
\renewcommand{\labelenumi}{\Roman{enumi}.}
\item \textbf{Lower-level objective.}
Function $L$ is globally Lipschitz with constant $L_L$. The global minimum value is
$L_{\min}:=\inf_{x\in\R^d}L(x)$ and $\Theta$ is nonempty. 
% Fix once and for all a target point $\theta^\star\in\Theta$. 
We assume that $L$ is bounded from above: there exists $L_{\max}<\infty$ such that 
$
L_{\min}\le L(\theta)\le L_{\max}$ for all $\theta\in\R^d.
$
Moreover, there exist $\eta_L>0$, $\nu_L>0$, and $L_\infty>0$ such that
\begin{equation}
    \dist(\theta,\Theta)\le \eta_L^{-1}\big(L(\theta)-L_{\min}\big)^{\nu_L}
    \qquad\text{whenever }L(\theta)-L_{\min}\le L_\infty.
    \label{eq:L-growth}
\end{equation}

\item \textbf{Upper-level objective.}
Function $G$ is globally Lipschitz with constant $L_G$ and bounded: there exist constants $G_{\min},G_{\max}\in\R$ such that
$
G_{\min}\le G(\theta)\le G_{\max}$ for all $\theta\in\R^d.
$

\item \textbf{Upper-level local growth near lower-level minimizers.}
There exists $R_G>0$ such that, for every $r_G\in(0,R_G]$, there is
$\thetae_{r_G}\in  B_{r_G}(\theta^\star)$ satisfying
$
    G(\thetae_{r_G})=\min_{\theta\in N_{r_G}(\Theta)}G(\theta).
$
Moreover, there exist $\eta_G>0$, $\nu_G>0$, and $G_\infty>0$, independent of $r_G$, such that
\begin{equation}
    \norm{\theta-\thetae_{r_G}}_2
    \le \eta_G^{-1}\big(G(\theta)-G(\thetae_{r_G})\big)^{\nu_G}
    \label{eq:G-growth}
\end{equation}
whenever $\theta\in N_{r_G}(\Theta)$ and $G(\theta)-G(\thetae_{r_G})\le G_\infty$.

\item \textbf{Initialization.}
The initial particles
$\Theta_{0,\tau}^{1,N},\ldots,\Theta_{0,\tau}^{N,N}$ are i.i.d. with common law $\rho_0\in\mathcal{P}_4(\R^d)$. 
%The mean-field copies are synchronously coupled with the interacting particles by setting $\bar\Theta_{0,\tau}^i=\Theta_{0,\tau}^{i,N},$ for $ i\in[N]$ and by driving $\Theta_{t,\tau}^{i,N}$ and $\bar\Theta_{t,\tau}^i$ with the same Brownian motion $W^i$. 

% \item \textbf{Initialization and synchronous coupling.}
% For every $N\in\N$, the initial particles
% $\Theta_{0,\tau}^{1,N},\ldots,\Theta_{0,\tau}^{N,N}$ are i.i.d. with common law $\rho_0\in\mathcal{P}_4(\R^d)$. The mean-field copies are synchronously coupled with the interacting particles by setting
% $\bar\Theta_{0,\tau}^i=\Theta_{0,\tau}^{i,N},$ 
% for 
% $ i\in[N]$ and by driving $\Theta_{t,\tau}^{i,N}$ and $\bar\Theta_{t,\tau}^i$ with the same Brownian motion $W^i$. 

\item \textbf{Mean-field well-posedness.}
For every finite horizon \(T>0\) and every fixed parameter tuple
\((\alpha,\beta,\tau,\lambda,\sigma)\), the mean-field SDE
\eqref{eq:mf-sde} admits a weak solution \((\Theta_{t,\tau})_{t\in[0,T]}\)
with continuous paths and law
$
\rho_{t,\tau}:=\Law(\Theta_{t,\tau})
\in C([0,T],\mathcal P_4(\mathbb R^d)).
$
Moreover, \((\rho_{t,\tau})_{t\in[0,T]}\) satisfies the associated weak
Fokker--Planck formulation. 
\end{enumerate}
\end{assumption}

\begin{remark}
     Assumption \ref{ass:primitive} I. is called \textbf{Hölderian error bound} and it imposes that the objective is locally $h_L$-Hölder continuous around $\Theta$. It ensures that the lower-level values on suitable sublevel sets provide quantitative information about the distance of the corresponding points to the set of global minimizers. This condition appears in \cite{jiang2023conditional,cao2024accelerated,chen2024penalty}
\end{remark}

\begin{remark}[Only the derived estimates are essential]
The assumptions $G \le G_{\max}$ and $L \le L_{\max}$ serve as sufficient, but not necessary, conditions for the convergence proof. In the analysis, the condition $G \le G_{\max}$ is utilized exclusively through the consensus-moment estimate (see \eqref{eq:consensus-moment}), and $L \le L_{\max}$ only via the inverse-stability estimate (see \eqref{eq:q-inverse-stability}). Consequently, the results extend to any set of assumptions that yield these two estimates. 
%For instance, in our experiment, we consider such an objective function that is not bounded.
\end{remark}

\begin{remark}[Boundedness by objective-value clipping]
The boundedness assumptions on the lower- and upper-level objectives are
sufficient conditions for the proof. In practice, one often has a priori admissible objective levels obtained from a feasible or reference solution.
One may therefore replace the original objectives by their capped versions, $G_{\rm cap}(x):=\min\{G(x),G_{\max}\},
L_{\rm cap}(x):=\min\{L(x),L_{\max}\}$.
Equivalently, whenever an evaluated objective value exceeds the prescribed upper level, it can be clipped to that level. 
\end{remark}

\begin{remark}
Assumption \ref{ass:primitive} III. is called inverse continuity condition. It states that, inside \(N_{r_G}(\Theta)\), points with
upper-level objective value close to the local minimum value
\(G(\tilde\theta_{r_G})\) must also be close in distance to
\(\tilde\theta_{r_G}\). Equivalently, the upper-level objective gap
\(G(\theta)-G(\tilde\theta_{r_G})\) quantitatively controls the distance from \(\theta\) to the selected local minimizer \(\tilde\theta_{r_G}\). This prevents the upper-level objective from having flat low-value regions, far from
\(\tilde\theta_{r_G}\), inside the relevant lower-level neighborhood. This type of condition appears in \cite{fornasier2021consensus,fornasier2024consensus,riedl2024leveraging}.
\end{remark}

% Regarding the well-posedness of the mean-field assumption, it is noteworthy that \cite{trillos2024cb} proves the existence of a stochastic process whose law evolves according to the mean-field dynamics and satisfies the associated Fokker--Planck equation in the limit $\tau \to 0$. 

Assumption \ref{ass:primitive} V is used as a standing well-posedness condition for the mean-field dynamics. We do not prove mean-field well-posedness in full generality. Instead, this assumption allows us to focus on the convergence mechanism induced by the soft-quantile consensus map. Related well-posedness results have been established for standard CBO and CB\(^2\)O under comparable regularity assumptions \cite{fornasier2024consensus,trillos2024cb}.

\section{Main Theoretical Results}\label{sec:results}
%\vspace{-.1cm}

We now state the global mean-field convergence result for SCB$^2$O dynamics. 

\begin{theorem}[Convergence of the direct soft-quantile mean-field dynamics]\label{thm:direct-mf-convergence}
Under Assumptions~\ref{ass:primitive} and \ref{ass:selector} suppose that \(\theta^\star\in\operatorname{supp}(\rho_0)\). Fix any 
$\varepsilon\in\big(0,\frac12 W_2^2(\rho_0,\delta_{\theta^\star})\big),$ 
and $\vartheta\in(0,1)$, 
choose \(\lambda,\sigma>0\) with 
$
2\lambda>d\sigma^2,
$
and define the time horizon
\begin{equation*}
    T^\star
    :=
    \frac{1}{(1-\vartheta)(2\lambda-d\sigma^2)}
    \log\left(
        \frac{W_2^2(\rho_0,\delta_{\theta^\star})/2}
        {\varepsilon}
    \right).
\end{equation*}
Then there exist \(\tau_0>0\) and mappings
$
\tau\mapsto \beta_\tau\in(0,1),
$
and 
$
\tau\mapsto \alpha_\tau>0,
$
defined for \(\tau\in(0,\tau_0]\), such that the following holds: 
For every \(\tau\in(0,\tau_0]\), let
$
    \rho=(\rho_{t,\tau})_{t\in[0,T^\star]}
    \in C([0,T^\star],\mathcal P_4(\mathbb R^d)),
$
be a weak solution (see Definition \ref{def:weak-solution}) of the mean-field Fokker--Planck equation in~\eqref{eq:fokker}, with initial law \(\rho_0\), and with consensus point computed using the parameter tuple
$
    (\alpha_\tau,\beta_\tau,\tau,\lambda,\sigma).
$
Assume also that the mapping  $t\mapsto m^{G,L}_{\alpha_\tau,\beta_\tau,\tau}(\rho_{t,\tau})$ is continuous on $[0,T^\star]$.
Then there exists a time $T\in\left[\frac{1-\vartheta}{1+\vartheta/2}T^\star,T^\star\right]$ such that
\begin{equation*}
     \frac12
    W_2^2(\rho_{T,\tau},\delta_{\theta^\star})
    =
    \varepsilon.
\end{equation*}
Furthermore, for  \(t\in[0,T]\), $W_2^2(\rho_{t,\tau},\delta_{\theta^\star})$ decays at least exponentially fast, i.e.,
\begin{equation*}
    W_2^2(\rho_{t,\tau},\delta_{\theta^\star})
    \le
    W_2^2(\rho_0,\delta_{\theta^\star})
    \exp\left(
        -(1-\vartheta)(2\lambda-d\sigma^2)t
    \right).
\end{equation*}
\end{theorem}

Proof is in Appendix \ref{pp:thm:direct-mf-convergence}. For a discussion of the choice of the central hyperparameters $\alpha,\beta,$ etc, admissible choices with arbitrarily small \(\beta\) and arbitrarily large \(\alpha\) are allowed. We refer to the detailed discussions and selection rule in the \cref{rem:auxiliary-mass-alpha-beta} and \cref{rem:selection-rule-alpha-beta}. 
The assumption $\theta^\star \in \operatorname{supp}(\rho_0)$ on the initial configuration $\rho_0$ is not a real restriction, since it holds for $\rho_t$ for any $t>0$ due to the diffusive nature of the dynamics; see also \cite{fornasier2024consensus,fornasier2025regularity}.
The dimensional dependence in the condition $2\lambda > d\sigma^2$ on the drift and noise parameters can be removed by introducing anisotropic noise \cite{carrillo2021consensus}, which leads to the weaker requirement $2\lambda > \sigma^2$. Such conditions may be further relaxed by truncating the noise; see \cite{fornasier2025consensus}.
The convergence rate is at least $(1-\vartheta)(2\lambda - d\sigma^2)$. Hence, by taking $\alpha \to \infty$ (which allows $\vartheta \to 0$), the rate can be made arbitrarily close to $2\lambda - d\sigma^2$.

% This result requires that $\rho \in C([0,T^\ast], \mathcal{P}_4(\mathbb{R}^d))$ is a weak solution to \eqref{eq:fokker} with initial condition $\rho_0$, and that the mapping $t \mapsto m^{G,L}_{\alpha,\beta,\tau}(\rho_t)$ is continuous. 
% This is also not restrictive, for example
% Theorem~2.4 in \cite{trillos2024cb} shows that solutions to the mean-field Fokker--Planck equation with the required regularity, for which the map $t \mapsto m^{G,L}_{\alpha,\beta,\tau}(\rho_t)$ is continuous, do indeed exist under for CB$^2$O which is the limiting case of our SCB$^2$O. Note that these assumptions are therefore sufficient, but not necessarily necessary. 

In order to establish the convergence result for the finite-particle dynamics, we first show that the particles generated by the mean-field dynamics and those produced by \eqref{eq:interacting_sde} are close on average.

\begin{proposition}[Sharp mean-field approximation on the cutoff event]\label{prop:mfa-cutoff}
Under Assumption~\ref{ass:primitive}, if
$
\Prob(\Omega_M)\ge\frac12,
$ 
then there exists a constant $C_{\mathrm{MFA}}
=C_{\mathrm{MFA}}
\bigl(M,T,d,\lambda,\sigma,\alpha,\beta,\tau,\rho_0\bigr)>0$ such that
\begin{equation*}
    \max_{1\le i\le N}\sup_{t\in[0,T]}
    \E\sqbr{\norm{\Theta_{t,\tau}^{i,N}-\bar\Theta_{t,\tau}^i}_2^2\mid\Omega_M}
    \le\frac{C_{\mathrm{MFA}}}{N}.
\end{equation*}
\end{proposition}

Proof is in Appendix \ref{pp:prop:mfa-cutoff}. Our analysis for the finite-particle regime relies on a potential function that measures how far the particle-induced distribution and the consensus point are from the global optimum, i.e., $V_\tau(t):=\frac12\int_{\R^d}\norm{\theta-\theta^\star}_2^2\rho_{t,\tau}(\!\dd\theta)$. For details, see Appendix \ref{pp:thm:numerical-convergence}.

%and $M_\tau(t):=\|m^{G,L}_{\alpha,\beta,\tau}(\rho_{t,\tau})-\theta^\star\|_2$.

% \begin{equation}
%     V_\tau(t):=\frac12\int_{\R^d}\norm{\theta-\theta^\star}_2^2\rho_{t,\tau}(\dd\theta),
%     \qquad
%     M_\tau(t):=\norm{m^{G,L}_{\alpha,\beta,\tau}(\rho_{t,\tau})-\theta^\star}_2.
%     \label{eq:V-M-def}
% \end{equation}

\begin{theorem}[Numerical convergence in probability]
\label{thm:numerical-convergence}
Under Assumptions~\ref{ass:selector} and \ref{ass:primitive}  let \(T=T_\varepsilon\) be the
hitting time supplied by Theorem~\ref{thm:direct-mf-convergence}, and let
$
    K\Delta t=T,
$
for 
$
    0<\Delta t\le1.
$
Define \(\Omega_M^\theta\) and \(\Omega_M^{\mathrm{num}}\) as in
\eqref{eq:Omega-theta-def} and~\eqref{eq:Omega-num-def}, respectively, and
$
C_{\mathrm{bd}}^{\mathrm{num}}(T)
:=
C_{\mathrm{bd}}(T)+C_{\mathrm{bd}}^\theta(T),
$
and let \(\delta_M\in(0,1/2]\). 
Set
$
M:=\tfrac{C_{\mathrm{bd}}^{\mathrm{num}}(T)}{\delta_M}.
$
If
$
V_\tau(T)=\varepsilon,
$
then there exists a finite constant
$
C_{\mathrm{NA}}\!:=\!C_{\mathrm{NA}}
\br{M,N,d,\alpha,\beta,\tau,\lambda,\sigma,T,\rho_0}\!>\!0,
$
 such that, for every
\(\varepsilon_{\mathrm{tot}}>0\),
\begin{equation*}
\mathbb P\Big(
\|\widehat\theta_{K,\tau}^{N}-\theta^\star\|_2^2
\le \varepsilon_{\mathrm{tot}}
\Big)\ge 1-\delta_M -\frac{3C_{\mathrm{NA}}\Delta t+6C_{\mathrm{MFA}}N^{-1}
+12\varepsilon}{\varepsilon_{\mathrm{tot}}}.
\end{equation*}
\end{theorem}

\section{Experiments}\label{sec:experiment}
%\vspace{-.1cm}
In this section, we evaluate the performance of the SCB$^2$O and compare it with relevant benchmarks. 
%{\color{red}The assumptions in Assumption \ref{ass:primitive} are sufficient for the theory; the experiments below illustrate the algorithmic behavior on the original objectives. }
In particular, we design two settings, a 2D constrained problem and a MNIST neural-network problem.
% to answer the following questions:
% \begin{itemize}[leftmargin=*, align=left]
%     \item Does SCB$^2$O approximate the behavior of CB$^2$O as the lower-to-upper
%     selection sharpness parameter $\xi:=\frac{1}{\tau \alpha}$ increases?
%     \item How does the soft lower-to-upper coupling in SCB$^2$O affect
%     feasibility, consensus formation, and optimisation quality on bi-level optimization problems?
%     \item Does the same soft-selection mechanism remain stable when the
%     optimisation variable is a neural-network parameter vector rather than a
%     two-dimensional decision variable?
% \end{itemize}
Here, we compare CB$^2$O vs SCB$^2$O and defer comparisons with other benchmarks such as VPBGDD of \cite{shen2023penalty} to Appendix \ref{app:comparison}. 
Recall that CB$^2$O uses a hard-selection rule, while SCB$^2$O uses a soft  rule parametrized by $\tau$. Note that small $\tau$, equivalently, large value of  $\xi:=\frac{1}{\tau \alpha}$ make the soft selection increasingly sharp and therefore closer to the hard decision rule of CB$^2$O, while smaller values allow a broader set of particles to influence the consensus. We use the same particle count, time horizon, and shared hyperparameters for both methods, so that differences are attributable to the selection mechanism and to the value of $\xi$. 
All reported curves aggregate five independent seeds. 

\textbf{2D Constrained Problem:}\label{sec:exp-2d}
The first setting isolates constrained optimization behavior in two dimensions, where feasibility, consensus distance, and particle collapse can be measured directly.  We optimize the shifted \emph{Ackley} upper objective given by
\begin{equation}
  G(\theta) = -A\exp\!\Bigl(-a\sqrt{\tfrac{b^2}{d}}\|\theta-\hat\theta\|_2\Bigr)
         -\exp\!\Bigl(\tfrac{1}{d}\textstyle\sum_j \cos(2\pi b(\theta_j-\hat\theta_j))\Bigr)
         + A + \exp(1),
\end{equation}

with $A=20$, $a=0.2$, $b=3,d=2$, $\hat\theta=(0.5,\tfrac{1}{3})$ (the shifts),
subject to two lower-level functions, where
$r(\theta)=\theta_1^2+\theta_2^2$ and $\phi(\theta)=tan^{-1}(\theta_2/\theta_1)$.
%$\phi(\theta)=\operatorname{atan2}(\theta_2,\theta_1)$, where $\operatorname{atan2}$ is an efficient implementation of $tan^{-1}(y/x)$.
\begin{table}[h]
\vspace{-.2cm}
\centering
%\caption{2D problems.}
%\label{tab:2d-problems}
\begin{tabular}{lllcc}
\toprule
Problem & Constraint $L(x)= 0$ & Shape & $\theta^\star$ & $G(\theta^\star)$ \\
\midrule
Circular & $\theta_1^2+\theta_2^2-1$ & Unit circle & $(0.782,\,0.624)$ & $4.003$ \\
Star & $r(\theta)-\bigl(1+0.5\sin(5\phi(\theta))\bigr)^2$ & Star-shaped & $(0.473,\,0.464)$ & $2.777$ \\
\bottomrule
\end{tabular}
\vspace{-.2cm}
\end{table}

For each problem and seed, both algorithms are run for 600 steps with
$N=100$ particles.  The shared parameters are
$\alpha=30$, $\beta=0.05$, $\lambda=1.0$, and $\sigma=1.0$.  For SCB$^2$O, we sweep $\xi\in\{10,100,1000,10000\}$; 
Table~\ref{tab:2d-settings} in Appendix \ref{app:experiment} summaries the setting.

\textit{Evaluation.}
Let $c_\star$ denotes the final consensus point produced by the algorithms. The following quantities are recorded at each step as performance metrics: $L(c_\star)$: constraint violation at consensus point (lower objective). $G(c_\star)$: upper objective value at consensus. 
%$G(c_\star)-G^*$: upper gap (floored at $10^{-12}$). 
$\|c_\star-\theta^\star\|_2$: consensus distance to the global optimum, and $\sigma(x):=\frac{1}{d}\sum_j\mathrm{std}_i(x_{ij})$: mean particle spread.
% \begin{itemize}[leftmargin=*, align=left]
%   \item $L(c_\star)$: constraint violation at consensus point (lower objective).
%   \item $G(c_\star)$: upper objective value at consensus.
%   \item $G(c_\star)-G^*$: upper gap (floored at $10^{-12}$).
%   \item $\|c_\star-x^*\|_2$: consensus distance to known optimum.
%   \item $\sigma(x)=\frac{1}{d}\sum_j\mathrm{std}_i(x_{ij})$: mean particle spread.
% \end{itemize}

%% ============================================================
%\subsection{Results: 2D Benchmark}
%% ============================================================
\begin{figure}[h]
  \centering
  \includegraphics[width=0.24\textwidth, trim=0 0 0 40, clip]{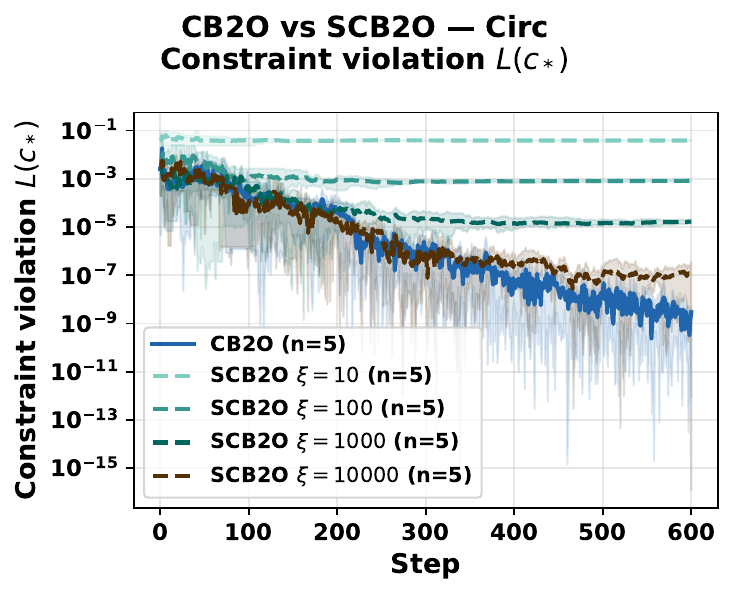}%
  \
  \includegraphics[width=0.24\linewidth, trim=0 0 0 40, clip]{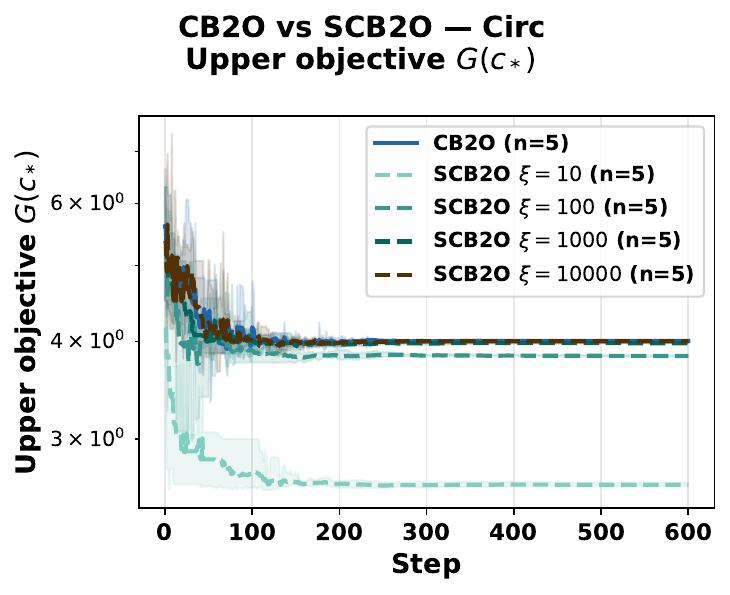}\
  \includegraphics[width=0.24\linewidth, trim=0 0 0 40, clip]{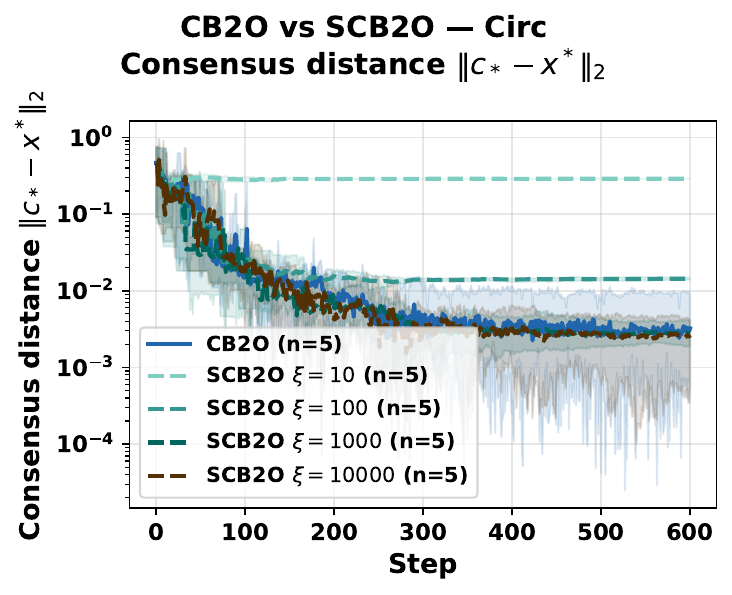}\
  \includegraphics[width=0.24\linewidth, trim=0 0 0 40, clip]{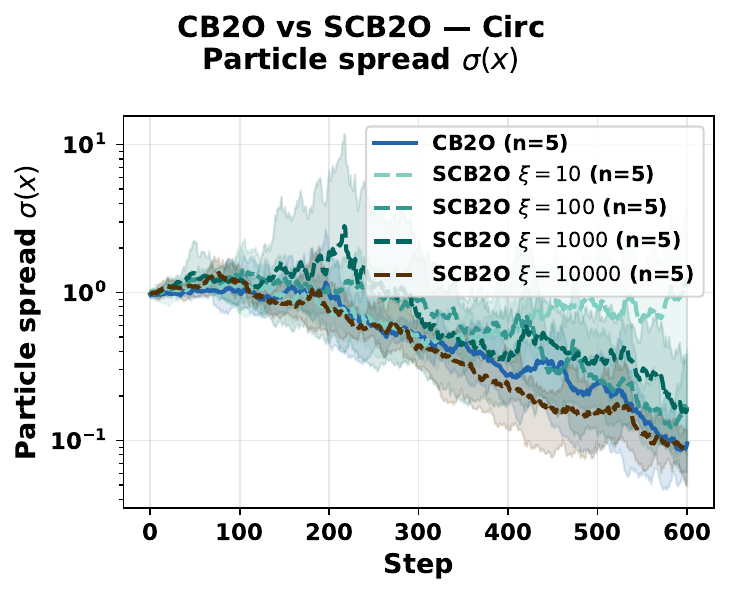}
%  \centering
  \\
  \includegraphics[width=0.24\linewidth, trim=0 0 0 40, clip]{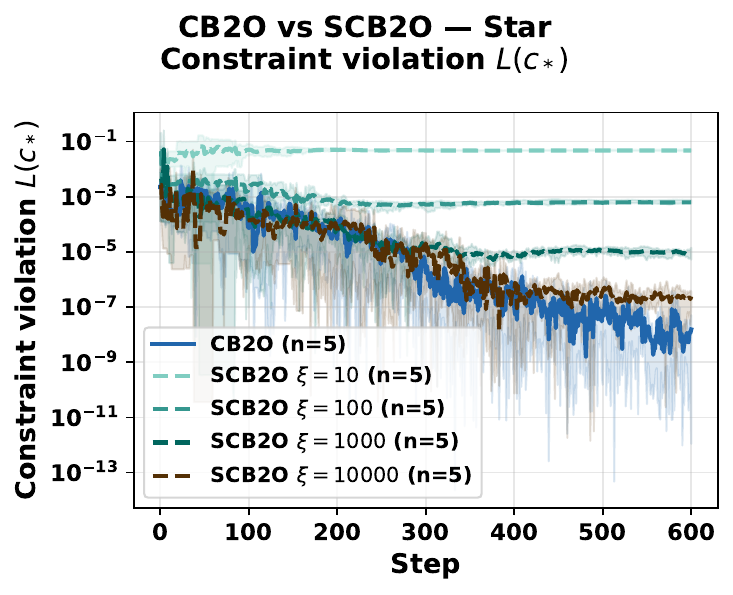}%
  \includegraphics[width=0.24\linewidth, trim=0 0 0 40, clip]{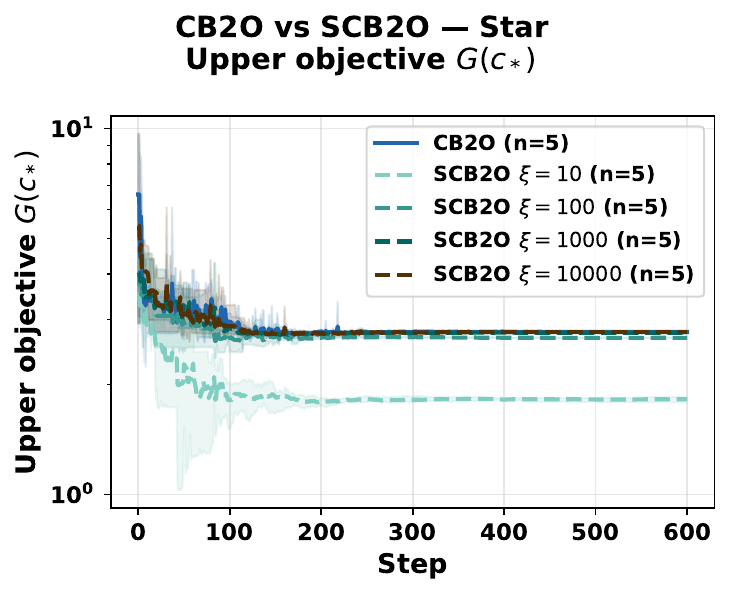}\
  %\[4pt]
  %\includegraphics[width=0.48\linewidth]{figures/benchmark2d_Star_upper_gap.pdf}%
  \includegraphics[width=0.24\linewidth, trim=0 0 0 40, clip]{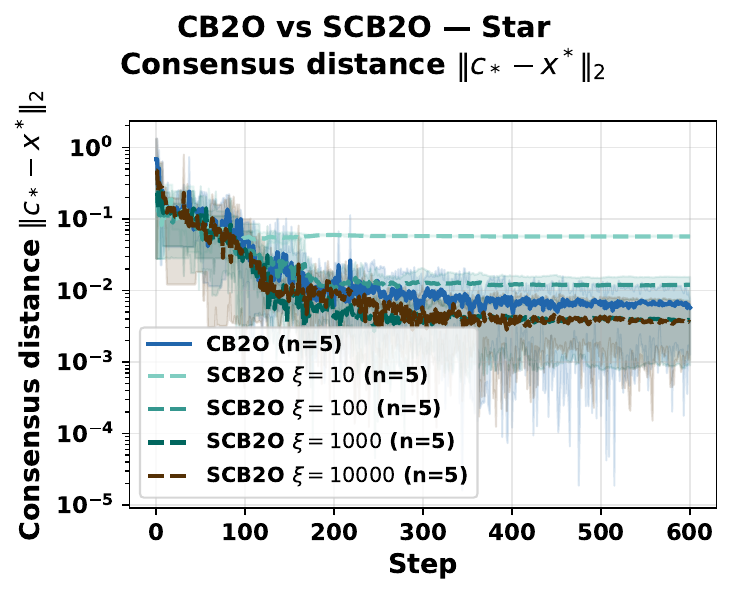}\
  \includegraphics[width=0.24\linewidth, trim=0 0 0 40, clip]{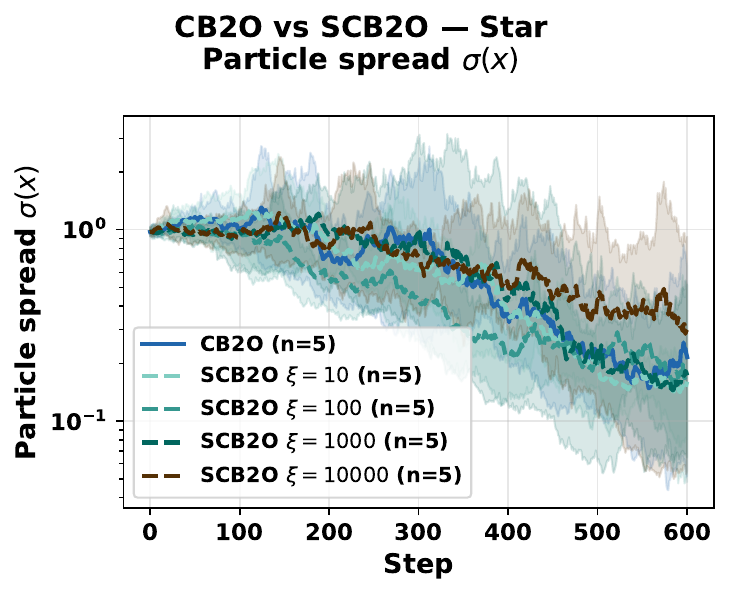}
  \caption{CB$^2$O vs SCB$^2$O on the circle constraint (first row) and on the star-shaped constraint (second row). All metrics use log scale. Solid line depicts the mean over 5 seeds; shaded band is min--max range.
From left to right: $L(c_\star)$, $G(c_\star)$, $\|c_\star-\theta^\star\|_2$, and $\sigma(x)$.           
}\label{fig:2d-star}
\vspace{-.4cm}
\end{figure}

% \caption{CB$^2$O vs SCB$^2$O on the circle constraint. All metrics use log scale. Solid line = mean over 5 seeds; shaded band = min--max range.}
%   \label{fig:2d-circ}

\textit{Results.}
Figure \ref{fig:2d-star} depicts the resulting plots on the circle and star constraints. See also Appendix \ref{app:experiment}. Moreover, Table~\ref{tab:2d-results} reports the metrics after 600 iterations for a quantitative comparison. These results suggest: 
\textit{Constraint satisfaction improves with larger $\xi$}, that is the lower-level objective $L(c_\star)$ decreases monotonically as $\xi$ increases.  %At $\xi=10$, the consensus violates the constraint by $\sim4\times10^{-2}$, roughly three orders of magnitude worse than CB$^2$O.  
%At $\xi=10000$ the violation ($\sim10^{-7}$) approaches CB$^2$O ($\sim10^{-9}$).
% \textit{SCB$^2$O with $\xi=10$--$100$ achieves lower apparent upper-level objective.}
% The upper gap is clamped at $10^{-12}$ for $\xi\le1000$, meaning the consensus point reaches an upper objective value \emph{below} the analytically computed constrained optimum $G^*$.  This is not a contradiction: when
% $L(c_\star)>0$ the consensus lies \emph{outside} the feasible set, where $G$ can take values lower than the constrained minimum.  The result is therefore misleading---the particle swarm has not truly converged.
Particle spread $\sigma(x)$ indicates consensus quality and as $\xi$ increases the particles collapse quicker. 
%For $\xi=10$, $\sigma(x)$ is an order of magnitude larger than CB$^2$O, confirming that particles have not collapsed to a common feasible point. For $\xi=10000$, $\sigma(x)$ is comparable to CB$^2$O, indicating genuine consensus.
Overall, as $\xi$ increases, SCB$^2$O matches CB$^2$O in constraint satisfaction and consensus distance while using a differentiable soft selection. For smaller $\xi$, it relaxes constraint adherence in exchange for faster initial descent of the upper-level objective, a useful property when the feasible set is difficult to reach.

We emphasize that the goal of introducing SCB$^2$O is not to design a superior particle-based optimizer, but to provide a proxy that provably converges to the global optimum while maintaining performance comparable to the state-of-the-art CB$^2$O. Notably, although CB$^2$O performs well in practice, its discontinuity prevents the direct use of standard Lipschitz-based propagation-of-chaos and numerical stability arguments.

\textbf{MNIST Neural-Network Problem:}\label{sec:exp-mnist}
The second setting tests whether the performance remains stable in a high-dimensional learning problem.  
We train a CNN with bias parameters on MNIST using particle-based bi-level optimization.  The training set is restricted to 10,000 samples, mini-batches contain 60 examples, and each particle update uses $M=100$ mini-batch evaluations.  Each run lasts $T=100$ epochs.  
%The model and dataset settings are identical for CB$^2$O and SCB$^2$O.
The shared parameters for CB$^2$O and SCB$^2$O are $\alpha=50$, $\lambda_1=1.0$, $\sigma_1=0.632$, $\gamma=0.1$, anisotropic noise, and $p=0.5$ for the latent $\ell_p$-norm.  
%The minimum $\beta$ is $0.02$, and the decay factor is $1.0$, so the configured $\beta$ value remains constant over the run.  
We sweep $(\beta,N)\in$ $(0.04,50)$, $(0.06,50)$, $(0.08,50)$, and $(0.10,50)$.  Thus all MNIST experiments use $N=50$ particles and differ only in $\beta$, algorithm, seed, and, $\xi$ for SCB$^2$O. Moreover,  SCB$^2$O is evaluated at $\xi\in\{1,5,10,100,1000\}$. Table \ref{tab:mnist-settings} summarizes the setting.

%\textbf{Evaluation.}
%For each epoch, we record the training loss, test accuracy, and latent $\ell_p$-norm.  The per-$\beta$ figures show the full training trajectory over 100 epochs, while the varying-$\beta$ figures compare final-epoch statistics for each value of $\xi$.  Curves are averaged over five trials.
%, and shaded regions show the seed-wise range.

%% ============================================================
\textit{Results.}
%% ============================================================
%Training and test curves for the CNN on MNIST are shown in
Figure~\ref{fig:mnist-b004} shows the results for $\beta=0.04$, with all curves averaged over five trials. The plots report training loss, test accuracy, and the latent $\ell_p$-norm across different random seeds. Additional results are provided in Appendix~\ref{app:experiment}. 
Overall, SCB$^2$O remains competitive with CB$^2$O across all values of $\beta$ when $\xi \ge 100$. Larger values of $\beta$ correspond to including a larger fraction of particles in the computation of the consensus, or stronger noise injection, which can facilitate escaping local optima in early stages but may slow convergence in later epochs. Both methods exhibit similar sensitivity to $\beta$ at a fixed particle size $N=50$. 
For a comparison of final-epoch performance across different $\beta$ values, see Figure~\ref{fig:mnist-varying-beta} in Appendix~\ref{app:experiment}.

\begin{figure}[h]
  \centering
  \includegraphics[width=\linewidth, trim=0 0 0 30, clip]{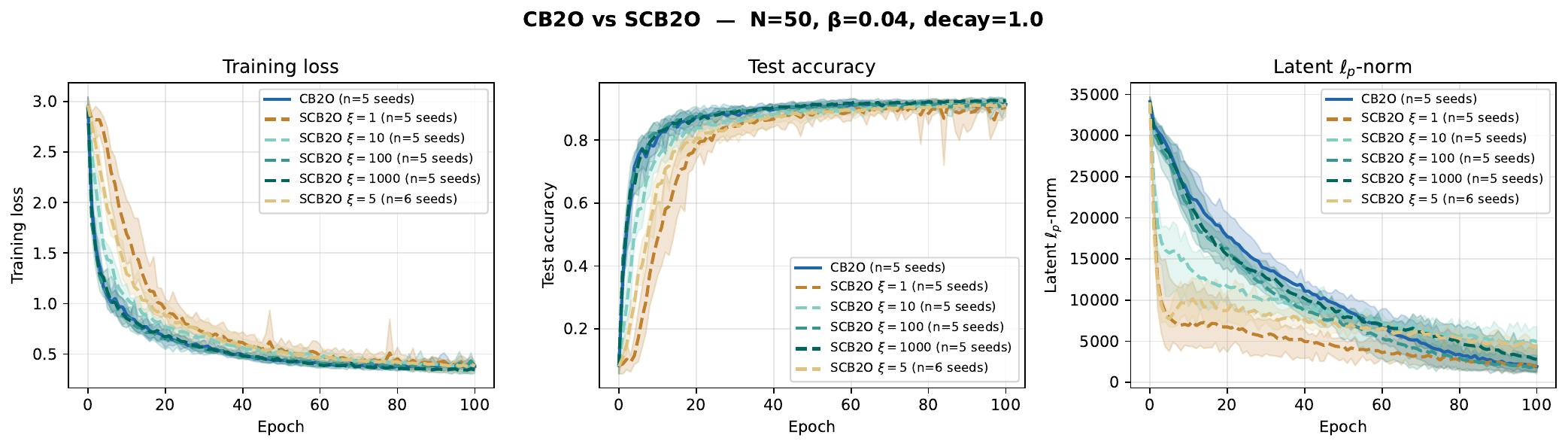}
  \caption{MNIST training curves: $N=50$, $\beta=0.04$.}
  \label{fig:mnist-b004}
\end{figure}
\vspace{-.4cm}

%\textit{Effect of $\xi$ on MNIST.}
Small values of $\xi$ (e.g., $1$ and $5$) allow poorly performing particles to influence the consensus, which degrades performance. In contrast, $\xi \ge 10$ stabilizes training and yields results close to CB$^2$O, thereby confirming the theoretical connection between the two algorithms.

%After the numerical instability fix (log-sum-exp stabilisation of the soft-selection weights), SCB$^2$O achieves test accuracy and training loss comparable to CB$^2$O for $\xi\ge100$.
%% ============================================================
%\textit{Summary.}
%% ============================================================
%SCB$^2$O provides a smooth, differentiable alternative to CB$^2$O's hard quantile selection.  The key findings are:

%\textbf{Equivalence at large $\xi$}: SCB$^2$O with $\xi\ge1000$       matches CB$^2$O on both 2D and MNIST benchmarks.

%  \textbf{Soft selection trades constraint adherence for exploration}:       smaller $\xi$ allows constraint-violating particles to contribute,         which can accelerate upper-objective descent but at the cost of         feasibility.

  %\textbf{Particle spread as a convergence diagnostic}: $\sigma(x)$ is a reliable indicator of whether the swarm has genuinely converged to a consensus, complementing the consensus-distance metric.

  %\textbf{Numerical stability}: the log-sum-exp stabilisation of soft-selection weights is critical; without it, SCB$^2$O degenerates to picking a single particle.

%\subimport{./}{exp_gd}

\vspace{-.2cm}
\section{Conclusion} 
\vspace{-.1cm}
We studied a bi-level optimization problem with nonconvex lower- and upper-level objectives and proposed a derivative-free particle-based method, SCB$^2$O, based on a soft consensus mapping rule.
The prior method CB$^2$O employs a hard consensus rule and demonstrates strong empirical performance in finite-particle regimes. However, due to the discontinuity of its mapping, standard Lipschitz-based finite-particle and numerical stability arguments do not directly apply.
Our approach addresses this gap by introducing a smooth consensus rule parameterized by $\tau$, which recovers the hard mapping as $\tau$ tends to zero. We established a localized high-probability finite-particle approximation bound for SCB$^2$O and showed empirically that it performs comparably to CB$^2$O for sufficiently small, but non-vanishing, values of $\tau$.

% As such a soft mapping incorporates larger portion of particles (depending on $\tau$) to form a consensus and approximate a solution, its performance degrades compare to the hard mapping one which only focuses on a portion of good behaving particles. 

\clearpage

\bibliographystyle{plainnat}
\bibliography{ref}

% \clearpage
% \subimport{./}{appendix1}

\clearpage
\appendix

\newcommand{\set}[1]{\left\{#1\right\}}
\newcommand{\one}{\mathbf 1}

\newcommand{\inner}[2]{\left\langle #1,#2\right\rangle}

\begin{center}
    {\Large\textbf{Appendix}}
\end{center}

The appendix is organized as follows:

\begin{itemize}[leftmargin=*, align=left]
    \item \textbf{Additional Terminologies: \ref{app:termin}}
    \item \textbf{Supporting Results: \ref{app:proofs}} 
    \item \textbf{Proof of Theorem \ref{thm:direct-mf-convergence}: \ref{pp:thm:direct-mf-convergence}} 
    \item \textbf{Proof of Proposition \ref{prop:mfa-cutoff}: \ref{pp:prop:mfa-cutoff}} 
    \item \textbf{Proof of Theorem~\ref{thm:numerical-convergence}: \ref{pp:thm:numerical-convergence}}
    \item \textbf{Extension to parameterized lower-level solution sets: \ref{app:extension}}
    \item \textbf{Limitations: \ref{app:limits}}
    \item \textbf{Extended Experimental Details: \ref{app:experiment}}
\end{itemize}

\section{Additional Terminologies}\label{app:termin}
\begin{definition}\label{def:mean-field}
    Let $F,G : \mathcal{P}(\mathbb{R}^d)\times \mathbb{R}^d \to \mathbb{R}^d$ be two functions, and consider for \(i=1,\dots,N\) the system of SDEs, expressed in Itô form, as
$$
dV_t^i=F\bigl(\hat{\rho}_t^N,V_t^i\bigr)\,dt+G\bigl(\hat{\rho}_t^N,V_t^i\bigr)\,dB_t^i,
\quad
\hat{\rho}_t^N=\frac{1}{N}\sum_{i=1}^N \delta_{V_t^i},
\quad V_0^i \sim \rho_0.
$$
We say that this SDE system converges in mean-field law to
$\bar v \in \mathbb{R}^d$
if all solutions of
$$
dV_t=F(\rho_t,V_t)\,dt + G(\rho_t,V_t)\,dB_t,
\quad
\rho_t = \operatorname{Law}(V_t),
\quad
V_0 \sim \rho_0,
$$
satisfy
\(\lim_{t\to\infty} W_p\bigl(\rho_t,\delta_{\bar v}\bigr)=0\), where  \(W_p\) denotes the $p$-Wasserstein distance, for some \(p\geq 1\).
\end{definition}

\begin{definition}\label{def:weak-solution}
Let $\rho_0 \in \mathcal{P}(\mathbb{R}^d)$ and $T>0$. We say that
$\rho \in C([0,T], \mathcal{P}(\mathbb{R}^d))$ satisfies the Fokker--Planck equation \eqref{eq:fokker} with initial condition $\rho_0$ in the weak sense on the time interval $[0,T]$ if, for all $\phi \in C_c^\infty(\mathbb{R}^d)$ and all $t \in (0,T)$,
\begin{align*}
\frac{d}{dt}\int_{\mathbb{R}^d} \phi(\theta)\, d\rho_t(\theta)
&=
\lambda \sum_{k=1}^d \int_{\mathbb{R}^d} \big(\theta_k - m^{G,L}_{\alpha,\beta,\tau}(\rho_t)_k\big)\, \partial_k \phi(\theta)\, d\rho_t(\theta) \\
&\quad + \frac{\sigma^2}{2} \sum_{k=1}^d \int_{\mathbb{R}^d}
\left(D\big(\theta - m^{G,L}_{\alpha,\beta,\tau}(\rho_t)\big)^2\right)_{kk}\,
\partial_{kk}^2 \phi(\theta)\, d\rho_t(\theta),
\end{align*}
and
$
\lim_{t \to 0} \rho_t = \rho_0
$
in the weak sense of probability measures.
\end{definition}

% \begin{definition}[Admissible algorithmic parameters]
% \label{def:admissible-algorithmic-parameters}
% A parameter tuple
% $
% (\alpha,\beta,\tau,\lambda,\sigma)
% $
% is called algorithmically admissible if
% \begin{equation}
%     \alpha>0,
%     \qquad
%     \beta\in(0,1),
%     \qquad
%     \tau>0,
%     \qquad
%     \lambda>0,
%     \qquad
%     \sigma>0.
%     \label{eq:algorithmic-parameter-domain}
% \end{equation}
% {\color{red} I do not understand this :
% Whenever a result below involves
% \[
%     q_{\beta,\tau}^L[\rho],
%     \qquad
%     J_{\beta,\tau}^L[\rho],
%     \qquad
%     A(\rho),
%     \qquad
%     B(\rho),
%     \qquad
%     m(\rho),
% \]
% and does not explicitly choose the parameters, the statement is understood for
% an arbitrary but fixed algorithmically admissible tuple
% $
%     (\alpha,\beta,\tau,\lambda,\sigma).
% $
% }
% \end{definition}

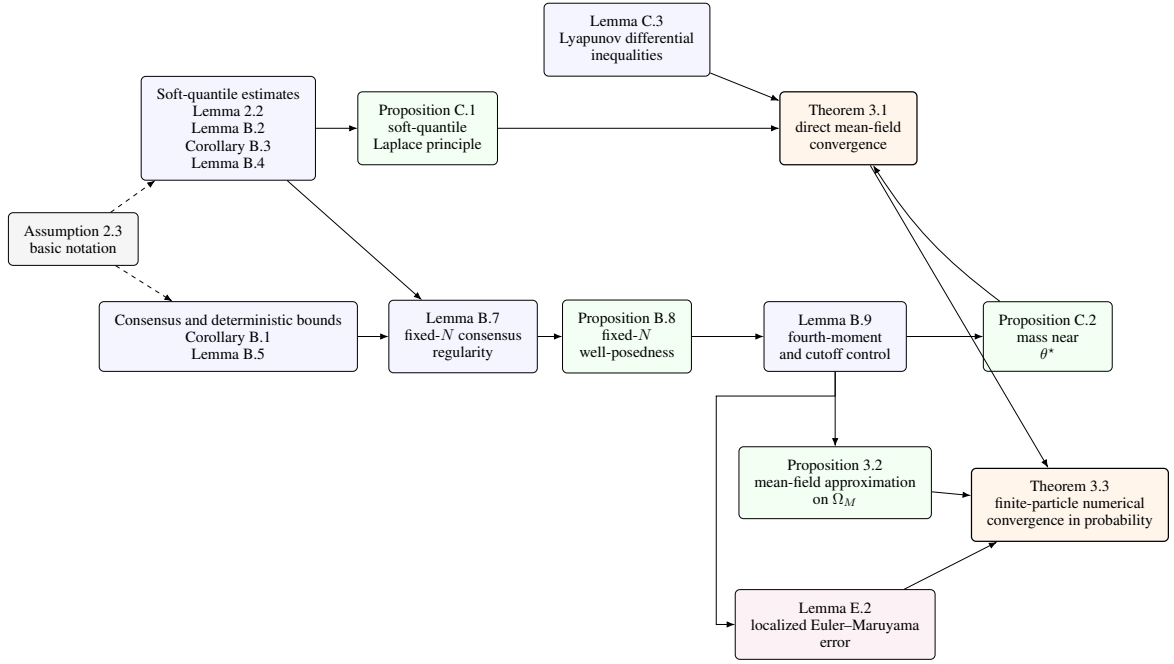
\begin{figure}[h]
%\centering
\hspace{-.5cm}
\resizebox{1.1\textwidth}{!}{%
\begin{tikzpicture}[
    >=Latex,
    every node/.style={font=\normalsize},
    dep/.style={->, line width=0.5pt},
    weakdep/.style={->, dashed, line width=0.47pt},
    box/.style={
        draw,
        rounded corners=3pt,
        align=center,
        inner xsep=10pt,
        inner ysep=8pt,
        minimum height=9mm,
        fill=white
    },
    assbox/.style={box, fill=gray!8},
    lembox/.style={box, fill=blue!4},
    propbox/.style={box, fill=green!5},
    thmbox/.style={box, fill=orange!8, thick},
    auxbox/.style={box, fill=purple!5}
]

% --------------------------------------------------------------------
% Left anchor
% --------------------------------------------------------------------
\node[assbox] (ass) at (0,0)
{Assumption~\ref{ass:primitive}\\basic notation};

% --------------------------------------------------------------------
% Left column: deterministic ingredients
% --------------------------------------------------------------------
\node[lembox] (soft) at (3.5,2.5)
{Soft-quantile estimates\\
Lemma~\ref{lem:soft-quantile-existence}\\
Lemma~\ref{lem:quantile-localization}\\
Corollary~\ref{cor:quantile-strong-monotonicity}\\
Lemma~\ref{lem:quantile-inverse-stability}};

\node[lembox] (ab) at (3.5,-2.2)
{Consensus and deterministic bounds\\
Corollary~\ref{cor:consensus-moment-stability}\\
Lemma~\ref{lem:deterministic-AB-bounds}};

\draw[weakdep] (ass) -- (soft);
\draw[weakdep] (ass) -- (ab);

% --------------------------------------------------------------------
% Mean-field line
% --------------------------------------------------------------------
\node[propbox] (laplace) at (8.0,2.5)
{Proposition~\ref{prop:soft-laplace}\\
soft-quantile\\
Laplace principle};

\node[lembox] (lyap) at (12.5,4.5)
{Lemma~\ref{lem:Lyapunov-DI}\\
Lyapunov differential\\
inequalities};

\node[thmbox] (mfconv) at (17.5,2.5)
{Theorem~\ref{thm:direct-mf-convergence}\\
direct mean-field\\
convergence};

\draw[dep] (soft) -- (laplace);
\draw[dep] (laplace) -- (mfconv);
\draw[dep] (lyap) -- (mfconv);

% --------------------------------------------------------------------
% Fixed-N / cutoff line
% --------------------------------------------------------------------
\node[lembox] (fixedreg) at (8.8,-2.2)
{Lemma~\ref{lem:fixed-N-consensus-regularity}\\
fixed-\(N\) consensus\\
regularity};

\node[propbox] (fixedwp) at (12.5,-2.2)
{Proposition~\ref{prop:fixed-N-well-posedness}\\
fixed-\(N\)\\
well-posedness};

\node[lembox] (fourth) at (17.2,-2.2)
{Lemma~\ref{lem:fourth-moment-bound}\\
fourth-moment\\
and cutoff control};

\node[propbox] (mass) at (22.0,-2.2)
{Proposition~\ref{prop:mass-lower-bound}\\
mass near\\
\(\theta^\star\)};

\draw[dep] (soft) -- (fixedreg);
\draw[dep] (ab) -- (fixedreg);
\draw[dep] (fixedreg) -- (fixedwp);
\draw[dep] (fixedwp) -- (fourth);
\draw[dep] (fourth) -- (mass);
\draw[dep] (mass) to[bend left=10] (mfconv);

% --------------------------------------------------------------------
% Lower-right: MFA / numerical line
% --------------------------------------------------------------------
\node[propbox] (mfa) at (17.2,-5.5)
{Proposition~\ref{prop:mfa-cutoff}\\
mean-field approximation\\
on \(\Omega_M\)};

\node[auxbox] (emloc) at (17.2,-8.7)
{Lemma~\ref{lem:localized-EM-approximation}\\
localized Euler--Maruyama\\
error};

\node[thmbox] (numconv) at (22.5,-6.0)
{Theorem~\ref{thm:numerical-convergence}\\
finite-particle numerical\\
convergence in probability};

\draw[dep] (fourth) -- (mfa);

% routed from the left to avoid crossing the mfa box
\draw[dep]
    (fourth.south) -- ++(0,-0.55)
    -| ([xshift=-12pt]emloc.west)
    -- (emloc.west);

\draw[dep] (mfconv) -- (numconv);
\draw[dep] (mfa) -- (numconv);
\draw[dep] (emloc) -- (numconv);

\end{tikzpicture}%
}
\caption{
Logical dependency structure of the proof. For readability, only the main
proof dependencies are displayed; several routine technical lemmas are omitted.
}
\label{fig:proof-dependency}
\end{figure}

\section{Supporting Results}\label{app:proofs}

For notational simplicity, we use the following throughout the remainder of this appendix,
$$
m(\rho)=m^{G,L}_{\alpha,\beta,\tau}(\rho).
$$
Figure~\ref{fig:proof-dependency} depicts the logical structure of the technical proof section and helps navigate the dependency order of the results.

\subsection{Existence and uniqueness of the soft quantile }
We  show that the soft quantile is well-defined by presenting a proof for Lemma \ref{lem:soft-quantile-existence}. 

\begin{proof}
For fixed $\rho$, define
\begin{equation*}
    H_\rho(q):=\int_{\R^d}\psi\left(\frac{q-L(x)}{\tau}\right)\rho(\dd x).
\end{equation*}
Since $0<\psi<1$ and $\psi$ is continuous, dominated convergence implies that $H_\rho$ is continuous. The limits of $\psi$ imply
\begin{equation*}
    \lim_{q\to-\infty}H_\rho(q)=0,
    \qquad
    \lim_{q\to+\infty}H_\rho(q)=1.
\end{equation*}
Because $\beta\in(0,1)$, the intermediate value theorem gives existence. If $q_2>q_1$, then strict monotonicity of $\psi$ gives
\begin{equation}
    \psi\left(\frac{q_2-L(x)}{\tau}\right)
    >\psi\left(\frac{q_1-L(x)}{\tau}\right)
    \qquad\forall x\in\R^d.
\end{equation}
Integrating yields $H_\rho(q_2)>H_\rho(q_1)$, hence the solution is unique.
\end{proof}

 \subsection{Consensus-moment stability from the upper bound on $G$}

\begin{corollary}
\label{cor:consensus-moment-stability}
Under Assumption~\ref{ass:primitive}, for every $\rho\in\cP_4(\R^d)$,
\begin{equation}
    \norm{m(\rho)}_2^4
    \le C_m^4\int_{\R^d}\norm{x}_2^4\rho(\dd x),
    \label{eq:consensus-moment}
\end{equation}
where
\begin{equation}
    C_m:=\beta^{-1/4}e^{\alpha(G_{\max}-G_{\min})}.
    \label{eq:Cm-def}
\end{equation}
For compatibility with the finite-horizon notation used below, set $C_m(T):=C_m$ for every $T>0$.
\end{corollary}
\begin{proof}
Let $J:=\softJ[\rho]$. Since $G(x)\le G_{\max}$ and $J(\R^d)=\beta$,
\begin{equation}
    B(\rho):=\int_{\R^d}e^{-\alpha G(x)}J(\dd x)
    \ge e^{-\alpha G_{\max}}J(\R^d)
    =\beta e^{-\alpha G_{\max}}.
    \label{eq:B-lower-basic}
\end{equation}
Since $G(x)\ge G_{\min}$,
\begin{align}\notag
    \norm{A(\rho)}_2
    &=\norm{\int_{\R^d}xe^{-\alpha G(x)}J(\dd x)}_2                                      \le \int_{\R^d}\norm{x}_2e^{-\alpha G(x)}J(\dd x)                                      \\
    &\le e^{-\alpha G_{\min}}\int_{\R^d}\norm{x}_2J(\dd x).
    \label{eq:A-upper-basic}
\end{align}
Holder's inequality with exponents $4$ and $4/3$ gives
\begin{align}\notag
    \int_{\R^d}\norm{x}_2J(\dd x)
    &\le \br{\int_{\R^d}\norm{x}_2^4J(\dd x)}^{1/4}J(\R^d)^{3/4}                         \\
    &=\beta^{3/4}\br{\int_{\R^d}\norm{x}_2^4J(\dd x)}^{1/4}.
    \label{eq:J-holder}
\end{align}
Because $J(\dd x)=\eta_{\rho,\tau}(x)\rho(\dd x)$ and $0<\eta_{\rho,\tau}(x)<1$, one has $J\le\rho$ as finite measures. Hence
\begin{equation}
    \int_{\R^d}\norm{x}_2^4J(\dd x)
    \le \int_{\R^d}\norm{x}_2^4\rho(\dd x).
    \label{eq:J-less-rho}
\end{equation}
Combining~\eqref{eq:B-lower-basic}, \eqref{eq:A-upper-basic}, \eqref{eq:J-holder}, and~\eqref{eq:J-less-rho},
\begin{align*}
    \norm{m(\rho)}_2
    &=\frac{\norm{A(\rho)}_2}{B(\rho)}                                                       \le \frac{e^{-\alpha G_{\min}}\beta^{3/4}}{\beta e^{-\alpha G_{\max}}}
    \br{\int_{\R^d}\norm{x}_2^4\rho(\dd x)}^{1/4}                                            \\
    &=\beta^{-1/4}e^{\alpha(G_{\max}-G_{\min})}
    \br{\int_{\R^d}\norm{x}_2^4\rho(\dd x)}^{1/4}.
\end{align*}
Taking fourth powers proves~\eqref{eq:consensus-moment}.
\end{proof}

\begin{remark}
The proofs below use the upper bound $G\le G_{\max}$ only through Corollary~\ref{cor:consensus-moment-stability}. Thus the same fourth-moment and convergence arguments remain valid if boundedness of $G$ is replaced by any other assumption implying~\eqref{eq:consensus-moment} with a finite constant $C_m$ fixed before the fourth-moment estimate is proved.
\end{remark}

\subsection{Quantile inverse stability from the upper bound on $L$}

For $\nu\in\cP(\R^d)$ define
\begin{equation}
    F_\nu(q):=\int_{\R^d}\psi\br{\frac{q-L(x)}{\tau}}\nu(\dd x)-\beta.
    \label{eq:Fnu-def}
\end{equation}
Since $\psi\in C^1$ and $\psi'$ is bounded by the Lipschitz constant $L_\psi$, dominated convergence gives
\begin{equation}
    \dot F_\nu(q):=\frac{\dd}{\dd q}F_\nu(q)
    =\frac1\tau\int_{\R^d}\psi'\br{\frac{q-L(x)}{\tau}}\nu(\dd x).
    \label{eq:Fnu-deriv}
\end{equation}

\begin{lemma}
\label{lem:quantile-localization}
Let $z_\beta:=\psi^{-1}(\beta)$. Under Assumption~\ref{ass:primitive}, for every probability measure $\nu\in\cP(\R^d)$,
\begin{equation}
    L_{\min}+\tau z_\beta\le q_\nu\le L_{\max}+\tau z_\beta.
    \label{eq:q-localization}
\end{equation}
\end{lemma}
\begin{proof}
The number $z_\beta$ is well-defined because $\psi$ is continuous, strictly increasing, and maps $\R$ onto $(0,1)$. The soft quantile $q_\nu$ satisfies
\begin{equation}
    \int_{\R^d}\psi\br{\frac{q_\nu-L(x)}{\tau}}\nu(\dd x)=\beta.
    \label{eq:q-nu-def}
\end{equation}
Because $L(x)\ge L_{\min}$, one has $q_\nu-L(x)\le q_\nu-L_{\min}$ for every $x$. Since $\psi$ is increasing,
\begin{equation*}
    \psi\br{\frac{q_\nu-L(x)}{\tau}}
    \le \psi\br{\frac{q_\nu-L_{\min}}{\tau}}.
\end{equation*}
Integrating and using~\eqref{eq:q-nu-def} gives
\begin{equation*}
    \beta\le \psi\br{\frac{q_\nu-L_{\min}}{\tau}}.
\end{equation*}
Applying the increasing inverse $\psi^{-1}$ yields
\begin{equation*}
    z_\beta\le\frac{q_\nu-L_{\min}}{\tau},
\end{equation*}
which proves the lower bound in~\eqref{eq:q-localization}. Similarly, because $L(x)\le L_{\max}$,
\begin{equation*}
    q_\nu-L(x)\ge q_\nu-L_{\max}.
\end{equation*}
Thus
\begin{equation*}
    \psi\br{\frac{q_\nu-L(x)}{\tau}}
    \ge \psi\br{\frac{q_\nu-L_{\max}}{\tau}}.
\end{equation*}
Integrating gives
\begin{equation*}
    \beta\ge \psi\br{\frac{q_\nu-L_{\max}}{\tau}}.
\end{equation*}
Applying $\psi^{-1}$ gives
\begin{equation*}
    z_\beta\ge\frac{q_\nu-L_{\max}}{\tau},
\end{equation*}
which proves $q_\nu\le L_{\max}+\tau z_\beta$.
\end{proof}

\begin{corollary}
\label{cor:quantile-strong-monotonicity}
Let
\begin{equation}
    c_\psi:=\inf_{z\in\left[z_\beta-\frac{L_{\max}-L_{\min}}{\tau},\ z_\beta+\frac{L_{\max}-L_{\min}}{\tau}\right]}\psi'(z),
    \qquad
    \kappa:=\frac{c_\psi}{\tau}.
    \label{eq:kappa-def}
\end{equation}
Then $c_\psi>0$, $\kappa>0$, and, for every $\mu,\nu\in\cP(\R^d)$ and every $q$ lying between $q_\nu$ and $q_\mu$,
\begin{equation}
    \br{F_\nu(q)-F_\nu(q_\nu)}\br{q-q_\nu}
    \ge \kappa\abs{q-q_\nu}^2.
    \label{eq:F-strong-monotone}
\end{equation}
\end{corollary}
\begin{proof}
The interval in~\eqref{eq:kappa-def} is compact. Since $\psi'(z)>0$ for every $z\in\R$ and $\psi'$ is continuous, its infimum over this compact interval is strictly positive. Hence $c_s>0$ and $\kappa>0$.

Let $q$ lie between $q_\nu$ and $q_\mu$. By Lemma~\ref{lem:quantile-localization}, both $q_\nu$ and $q_\mu$ belong to $[L_{\min}+\tau z_\beta,L_{\max}+\tau z_\beta]$. Therefore $q$ also belongs to this interval. Since $L_{\min}\le L(x)\le L_{\max}$, for every $x\in\R^d$,
\begin{equation*}
    z_\beta-\frac{L_{\max}-L_{\min}}{\tau}
    \le \frac{q-L(x)}{\tau}
    \le z_\beta+\frac{L_{\max}-L_{\min}}{\tau}.
\end{equation*}
By the definition of $c_\psi$,
\begin{equation*}
    \psi'\br{\frac{q-L(x)}{\tau}}\ge c_\psi,
    \qquad x\in\R^d.
\end{equation*}
Using~\eqref{eq:Fnu-deriv},
\begin{equation*}
    \dot F_\nu(q)
    =\frac1\tau\int_{\R^d}\psi'\br{\frac{q-L(x)}{\tau}}\nu(\dd x)
    \ge\frac{c_\psi}{\tau}=\kappa.
\end{equation*}
If $q\ge q_\nu$, then
\begin{equation*}
    F_\nu(q)-F_\nu(q_\nu)
    =\int_{q_\nu}^q\dot F_\nu(r)\dd r
    \ge \kappa(q-q_\nu).
\end{equation*}
Multiplying by $q-q_\nu\ge0$ gives~\eqref{eq:F-strong-monotone}. If $q\le q_\nu$, then
\begin{equation*}
    F_\nu(q_\nu)-F_\nu(q)
    =\int_q^{q_\nu}\dot F_\nu(r)\dd r
    \ge\kappa(q_\nu-q).
\end{equation*}
Equivalently,
\begin{equation*}
    F_\nu(q)-F_\nu(q_\nu)
    \le -\kappa(q_\nu-q).
\end{equation*}
Because $q-q_\nu\le0$, multiplying both sides by $q-q_\nu$ reverses the inequality and gives
\begin{equation*}
    \br{F_\nu(q)-F_\nu(q_\nu)}\br{q-q_\nu}
    \ge \kappa(q_\nu-q)^2
    =\kappa\abs{q-q_\nu}^2.
\end{equation*}
Thus~\eqref{eq:F-strong-monotone} holds in both cases.
\end{proof}

\begin{lemma}
\label{lem:quantile-inverse-stability}
Under Assumption~\ref{ass:primitive}, for every $\mu,\nu\in\cP(\R^d)$,
\begin{equation}
    \abs{q_\mu-q_\nu}
    \le \kappa^{-1}\abs{F_\mu(q_\nu)-F_\nu(q_\nu)}.
    \label{eq:q-inverse-stability}
\end{equation}
\end{lemma}
\begin{proof}
Because $q_\mu$ and $q_\nu$ solve their soft-quantile equations,
\begin{equation*}
    F_\mu(q_\mu)=0,
    \qquad
    F_\nu(q_\nu)=0.
\end{equation*}
Therefore, we obtain
\begin{equation}
    \abs{F_\mu(q_\nu)-F_\nu(q_\nu)}=\abs{F_\mu(q_\nu)}.
    \label{eq:F-reduction}
\end{equation}
If $q_\nu\le q_\mu$, then Corollary~\ref{cor:quantile-strong-monotonicity}, applied to the measure $\mu$ with $q=q_\nu$, gives
\begin{equation*}
    \br{F_\mu(q_\nu)-F_\mu(q_\mu)}\br{q_\nu-q_\mu}
    \ge \kappa(q_\mu-q_\nu)^2.
\end{equation*}
Since $F_\mu(q_\mu)=0$ and $q_\nu-q_\mu\le0$, this is equivalent to
\begin{equation*}
    -F_\mu(q_\nu)(q_\mu-q_\nu)
    \ge \kappa(q_\mu-q_\nu)^2.
\end{equation*}
If $q_\mu=q_\nu$, the desired estimate is trivial. If $q_\mu>q_\nu$, dividing by $q_\mu-q_\nu>0$ gives
\begin{equation*}
    -F_\mu(q_\nu)\ge \kappa(q_\mu-q_\nu).
\end{equation*}
Hence
$
    \abs{F_\mu(q_\nu)}\ge\kappa\abs{q_\mu-q_\nu}.
$ 
If $q_\nu>q_\mu$, the same corollary applied to the measure $\mu$ with $q=q_\nu$ gives
\begin{equation*}
    \br{F_\mu(q_\nu)-F_\mu(q_\mu)}\br{q_\nu-q_\mu}
    \ge \kappa(q_\nu-q_\mu)^2.
\end{equation*}
Since $F_\mu(q_\mu)=0$ and $q_\nu-q_\mu>0$, division by $q_\nu-q_\mu$ gives
\begin{equation*}
    F_\mu(q_\nu)\ge\kappa(q_\nu-q_\mu).
\end{equation*}
Again, $\abs{F_\mu(q_\nu)}\ge\kappa\abs{q_\mu-q_\nu}.$ 
Combining both cases with~\eqref{eq:F-reduction} proves~\eqref{eq:q-inverse-stability}.
\end{proof}

\begin{remark}
\label{rem:minimal-derived-inputs}
The later coupling proof uses the boundedness of $L$ only through the inverse-stability estimate~\eqref{eq:q-inverse-stability}. Consequently, the same mean-field approximation and numerical convergence results remain valid if the assumption $L\le L_{\max}$ is replaced by any other condition implying~\eqref{eq:q-inverse-stability} with a fixed positive constant $\kappa$. Likewise, the fourth-moment argument uses the upper bound on $G$ only through Corollary~\ref{cor:consensus-moment-stability}. Therefore the truly essential derived inputs for the rest of the proof are~\eqref{eq:consensus-moment} and~\eqref{eq:q-inverse-stability}. The boundedness assumptions on $G$ and $L$ are a simple sufficient route to obtain these estimates without circular reasoning.
\end{remark}

We define the uniform denominator constant and the moment-dependent numerator bound by
\begin{equation}
    b_G:=\beta e^{-\alpha G_{\max}},
    \qquad
    a(R_4):=\beta^{1/2}e^{-\alpha G_{\min}}R_4^{1/4}.
    \label{eq:bg-aR-def}
\end{equation}

\begin{lemma}
\label{lem:deterministic-AB-bounds}
For every $\rho\in\cP(\R^d)$,
\begin{equation}
    B(\rho)\ge b_G,
    \label{eq:B-lower-bG}
\end{equation}
where $B(\rho)$ is given in \eqref{eq:B-lower-basic}. 
In addition, if $\rho\in\cP_4(\R^d)$ satisfies
\begin{equation}
    \int_{\R^d}\norm{x}_2^4\rho(\dd x)\le R_4,
    \label{eq:R4-moment}
\end{equation}
then
\begin{equation}
    \norm{A(\rho)}_2\le a(R_4).
    \label{eq:A-upper-aR}
\end{equation}
Consequently,
\begin{equation}
    \norm{m(\rho)}_2=\frac{\norm{A(\rho)}_2}{B(\rho)}
    \le b_G^{-1}a(R_4)
    \label{eq:m-upper-aR}
\end{equation}
whenever~\eqref{eq:R4-moment} holds.
\end{lemma}
\begin{proof}
Let $J:=\softJ[\rho]$. Since $G(x)\le G_{\max}$,
\begin{align*}
B(\rho)
&=\int_{\R^d}e^{-\alpha G(x)}J(\dd x)                                            \ge e^{-\alpha G_{\max}}\int_{\R^d}J(\dd x)                =e^{-\alpha G_{\max}}J(\R^d)                                 =\beta e^{-\alpha G_{\max}}=b_G.
\end{align*}
This proves~\eqref{eq:B-lower-bG}.
Assume now~\eqref{eq:R4-moment}. Since $G(x)\ge G_{\min}$,
\begin{align}\notag
    \norm{A(\rho)}_2
    &=\norm{\int_{\R^d}xe^{-\alpha G(x)}J(\dd x)}_2                                  \le\int_{\R^d}\norm{x}_2e^{-\alpha G(x)}J(\dd x)                                 \\
    &\le e^{-\alpha G_{\min}}\int_{\R^d}\norm{x}_2J(\dd x).
    \label{eq:A-bound-lemma3}
\end{align}
By Cauchy--Schwarz with respect to the finite measure $J$,
\begin{align}\notag
    \int_{\R^d}\norm{x}_2J(\dd x)
    &\le\br{\int_{\R^d}\norm{x}_2^2J(\dd x)}^{1/2}J(\R^d)^{1/2}                       \\
    &=\beta^{1/2}\br{\int_{\R^d}\norm{x}_2^2J(\dd x)}^{1/2}.
    \label{eq:CS-J-second}
\end{align}
Because $0<\eta_{\rho,\tau}<1$, one has $J\le\rho$, hence
\begin{equation*}
    \int_{\R^d}\norm{x}_2^2J(\dd x)
    \le \int_{\R^d}\norm{x}_2^2\rho(\dd x).
\end{equation*}
Holder's inequality gives
\begin{equation*}
    \int_{\R^d}\norm{x}_2^2\rho(\dd x)
    \le \br{\int_{\R^d}\norm{x}_2^4\rho(\dd x)}^{1/2}
    \le R_4^{1/2}.
\end{equation*}
Combining the last three displays yields
\begin{equation*}
    \norm{A(\rho)}_2\le\beta^{1/2}e^{-\alpha G_{\min}}R_4^{1/4}=a(R_4),
\end{equation*}
which proves~\eqref{eq:A-upper-aR}. Finally,~\eqref{eq:m-upper-aR} follows from $m=A/B$,~\eqref{eq:B-lower-bG}, and~\eqref{eq:A-upper-aR}.
\end{proof}

\subsection{Cutoff event and stopped-process inequality}
For $M>0$, set
\begin{equation}
X_t^N:=\frac1N\sum_{i=1}^N\max\set{\norm{\Theta_{t,\tau}^{i,N}}_2^4,\norm{\bar\Theta_{t,\tau}^{i}}_2^4},
\quad
M_T:=\sup_{t\in[0,T]}X_t^N,
\quad
\Omega_M:=\set{M_T\le M}.
\label{eq:MT-Omega-def}
\end{equation}
Define the cutoff time
\begin{equation}
    \tau_M:=\inf\set{t\in[0,T]:X_t^N>M},
    \label{eq:tauM-def}
\end{equation}
with the convention $\tau_M=+\infty$ if the set is empty. Define the cutoff indicator
\begin{equation}
    I_M(t):=\one_{\{t<\tau_M\}}.
    \label{eq:IM-def}
\end{equation}
Since the particle paths and the synchronously coupled mean-field paths are continuous, $X_t^N$ is continuous. Therefore, on $\{I_M(t)=1\}$, one has
\begin{equation}
    X_t^N\le M.
    \label{eq:cutoff-pointwise-bound}
\end{equation}
Moreover, for every $t\in[0,T]$,
\begin{equation}
    \one_{\Omega_M}\le I_M(t).
    \label{eq:Omega-below-I}
\end{equation}
Indeed, if $\Omega_M$ holds, then $X_s^N\le M$ for every $s\in[0,T]$, so the set in~\eqref{eq:tauM-def} is empty and $\tau_M=+\infty$.

\begin{lemma}
\label{lem:stopped-integral}
Let \(\tau\) be a stopping time and define
$I_\tau(t):=\mathbf 1_{\{t<\tau\}},
$ for $t\in[0,T].$
Let \(Y=(Y_t)_{t\in[0,T]}\) be a nonnegative continuous adapted process with
\(Y_0=0\). Suppose that there exist a progressively measurable process
\(f=(f_t)_{t\in[0,T]}\) satisfying
\[
    \int_0^T \mathbb E\bigl[|f_t|I_\tau(t)\bigr]\,dt<\infty,
\]
and a true martingale \((\mathcal N_{t\wedge\tau})_{t\in[0,T]}\) with
\(\mathbb E[\mathcal N_{t\wedge\tau}]=0\), such that, for every \(t\in[0,T]\),
\[
Y_{t\wedge\tau}
\le
\int_0^t I_\tau(s) f_s\,ds
+
\mathcal N_{t\wedge\tau}
\qquad\text{a.s.}
\label{eq:stopped-pathwise-ineq}
\]
Then, for every \(t\in[0,T]\),
\[
\mathbb E\bigl[Y_t I_\tau(t)\bigr]
\le
\int_0^t
\mathbb E\bigl[f_s I_\tau(s)\bigr]\,ds .
\label{eq:stopped-expectation-ineq}
\]
\end{lemma}

\begin{proof}
Since \(Y_t\ge 0\), we have the pathwise inequality:
$
Y_t I_\tau(t)\le Y_{t\wedge\tau}.
$
Indeed, if \(t<\tau\), then \(t\wedge\tau=t\), so
$
    Y_t I_\tau(t)=Y_t=Y_{t\wedge\tau}.
$
If \(t\ge\tau\), then \(I_\tau(t)=0\), and hence
$
    Y_t I_\tau(t)=0\le Y_{t\wedge\tau},
$
because \(Y\) is nonnegative. 
Taking expectations in \eqref{eq:stopped-pathwise-ineq} gives
\[
    \mathbb E[Y_{t\wedge\tau}]
    \le
    \mathbb E\left[
        \int_0^t I_\tau(s)f_s\,ds
    \right]
    +
    \mathbb E[\mathcal N_{t\wedge\tau}].
\]
By the true martingale assumption,
$
    \mathbb E[\mathcal N_{t\wedge\tau}]=0.
$
Moreover, the integrability assumption and Fubini's theorem imply
\[
    \mathbb E\left[
        \int_0^t I_\tau(s)f_s\,ds
    \right]
    =
    \int_0^t
    \mathbb E\bigl[f_s I_\tau(s)\bigr]\,ds.
\]
Therefore
\[
    \mathbb E[Y_{t\wedge\tau}]
    \le
    \int_0^t
    \mathbb E\bigl[f_s I_\tau(s)\bigr]\,ds.
\]
Combining this estimate with
$
    \mathbb E[Y_t I_\tau(t)]\le \mathbb E[Y_{t\wedge\tau}]
$ 
proves \eqref{eq:stopped-expectation-ineq}.
\end{proof}

\begin{lemma}
\label{lem:fixed-N-consensus-regularity}
Assume Assumption~\ref{ass:primitive}. Fix \(N\in\mathbb N\). For
$
    \mathbf x=(x^1,\ldots,x^N)\in\mathbb R^{dN},
$
define
\[
    \|\mathbf x\|_{N,2}
    :=
    \left(\sum_{i=1}^N\|x^i\|_2^2\right)^{1/2},
    \qquad
    \mu_{\mathbf x}^N
    :=
    \frac1N\sum_{i=1}^N\delta_{x^i}.
\]
Set
$
    Q_N(\mathbf x):=\qsoft[\mu_{\mathbf x}^N],
$ 
and 
$
    \mathfrak m_N(\mathbf x):=m(\mu_{\mathbf x}^N).
$
Then the following two properties hold. First, for every \(R>0\), there exists a finite constant
\[
    L_{\mathfrak m,R,N}
    =
    L_{\mathfrak m,R,N}(\alpha,\beta,\tau,L_L,L_G,L_\psi,\kappa,G_{\min},G_{\max},N)
    >0
\]
such that
\[
    \|\mathfrak m_N(\mathbf x)-\mathfrak m_N(\mathbf y)\|_2
    \le
    L_{\mathfrak m,R,N}\|\mathbf x-\mathbf y\|_{N,2},
    \qquad
    \mathbf x,\mathbf y\in B_R^N,
    \label{eq:fixed-N-consensus-local-Lipschitz}
\]
where
$
    B_R^N:=\{\mathbf x\in(\mathbb R^d)^N:\|\mathbf x\|_{N,2}\le R\}.
$

Second, for every \(\mathbf x\in\mathbb R^{dN}\),
\[
    \|\mathfrak m_N(\mathbf x)\|_2
    \le
    C_m N^{-1/4}\|\mathbf x\|_{N,2}.
    \label{eq:fixed-N-consensus-linear-growth}
\]
\end{lemma}

\begin{proof}
We first prove the local Lipschitz property. By
Lemma~\ref{lem:quantile-inverse-stability}, for any
\(\mathbf x,\mathbf y\in(\mathbb R^d)^N\),
\begin{align}\notag
    |Q_N(\mathbf x)-Q_N(\mathbf y)|
    &\le
    \kappa^{-1}
    \left|
    F_{\mu_{\mathbf x}^N}(Q_N(\mathbf y))
    -
    F_{\mu_{\mathbf y}^N}(Q_N(\mathbf y))
    \right|                                                     \\ \notag
    &=
    \kappa^{-1}
    \left|
    \frac1N\sum_{i=1}^N
    \left[
    \psi\left(\frac{Q_N(\mathbf y)-L(x^i)}{\tau}\right)
    -
    \psi\left(\frac{Q_N(\mathbf y)-L(y^i)}{\tau}\right)
    \right]
    \right|                                                     \\ \notag
    &\le
    \frac{L_\psi L_L}{\kappa\tau N}
    \sum_{i=1}^N\|x^i-y^i\|_2                                  \\
    &\le
    \frac{L_\psi L_L}{\kappa\tau\sqrt N}
    \|\mathbf x-\mathbf y\|_{N,2}.
    \label{eq:fixed-N-QN-Lipschitz}
\end{align}
Define
\(
    C_{Q,N}:=\frac{L_\psi L_L}{\kappa\tau\sqrt N}.
\)
For \(i\in [N]\), set
\[
    a_i(\mathbf x)
    :=
    e^{-\alpha G(x^i)}
    \psi\left(\frac{Q_N(\mathbf x)-L(x^i)}{\tau}\right).
\]
Since \(G\ge G_{\min}\), \(G\) is globally Lipschitz, and
\(r\mapsto e^{-\alpha r}\) has derivative bounded by
\(\alpha e^{-\alpha G_{\min}}\) on \([G_{\min},\infty)\), we have
\[
    \left|
    e^{-\alpha G(x^i)}
    -
    e^{-\alpha G(y^i)}
    \right|
    \le
    \alpha e^{-\alpha G_{\min}}L_G\|x^i-y^i\|_2 .
\]
Using \(0<s<1\), the Lipschitz continuity of \(\psi\), the Lipschitz continuity
of \(L\), and~\eqref{eq:fixed-N-QN-Lipschitz}, we get
\[
    |a_i(\mathbf x)-a_i(\mathbf y)|
    \le
    C_{a,1}\|x^i-y^i\|_2
    +
    C_{a,2}\|\mathbf x-\mathbf y\|_{N,2},
    \label{eq:fixed-N-ai-Lipschitz}
\]
where
\[
    C_{a,1}
    :=
    \alpha e^{-\alpha G_{\min}}L_G
    +
    e^{-\alpha G_{\min}}\frac{L_\psi L_L}{\tau},
    \qquad
    C_{a,2}
    :=
    e^{-\alpha G_{\min}}\frac{L_\psi C_{Q,N}}{\tau}.
\]
Consequently,
\[
    \frac1N\sum_{i=1}^N
    |a_i(\mathbf x)-a_i(\mathbf y)|
    \le
    \left(\frac{C_{a,1}}{\sqrt N}+C_{a,2}\right)
    \|\mathbf x-\mathbf y\|_{N,2}.
    \label{eq:fixed-N-average-ai-Lipschitz}
\]
Let
$
    L_{B,N}:=\frac{C_{a,1}}{\sqrt N}+C_{a,2}.
$
Then, we can write
\[
    A_N(\mathbf x)
    :=
    A(\mu_{\mathbf x}^N)
    =
    \frac1N\sum_{i=1}^N x^i a_i(\mathbf x),
    \qquad
    B_N(\mathbf x)
    :=
    B(\mu_{\mathbf x}^N)
    =
    \frac1N\sum_{i=1}^N a_i(\mathbf x).
\]
By~\eqref{eq:fixed-N-average-ai-Lipschitz}, we have
\[
    |B_N(\mathbf x)-B_N(\mathbf y)|
    \le
    L_{B,N}\|\mathbf x-\mathbf y\|_{N,2}.
    \label{eq:fixed-N-BN-Lipschitz}
\]
Moreover, by Lemma~\ref{lem:deterministic-AB-bounds}, we have
$
    B_N(\mathbf x)\ge b_G:=\beta e^{-\alpha G_{\max}}>0.
$

Let \(\mathbf x,\mathbf y\in B_R^N\). Since \(a_i(\mathbf x)\le e^{-\alpha G_{\min}}\), we have
\begin{align*}
    \|A_N(\mathbf x)-A_N(\mathbf y)\|_2
    &\le
    \frac1N\sum_{i=1}^N
    \|x^i-y^i\|_2 a_i(\mathbf x)
    +
    \frac1N\sum_{i=1}^N
    \|y^i\|_2 |a_i(\mathbf x)-a_i(\mathbf y)|                 \\
    &\le
    \frac{e^{-\alpha G_{\min}}}{\sqrt N}
    \|\mathbf x-\mathbf y\|_{N,2}
    +
    R L_{B,N}\|\mathbf x-\mathbf y\|_{N,2}.
\end{align*}
Hence, 
$
    \|A_N(\mathbf x)-A_N(\mathbf y)\|_2
    \le
    L_{A,R,N}\|\mathbf x-\mathbf y\|_{N,2},
$
where
$
    L_{A,R,N}
    :=
    \frac{e^{-\alpha G_{\min}}}{\sqrt N}
    +
    R L_{B,N}.
$
Also, for \(\mathbf y\in B_R^N\),
\[
    \|A_N(\mathbf y)\|_2
    \le
    \frac{e^{-\alpha G_{\min}}}{N}
    \sum_{i=1}^N\|y^i\|_2
    \le
    \frac{e^{-\alpha G_{\min}}R}{\sqrt N}.
    \label{eq:fixed-N-AN-bound}
\]

Since
$
    \mathfrak m_N(\mathbf x)=\frac{A_N(\mathbf x)}{B_N(\mathbf x)},
$
we obtain, for \(\mathbf x,\mathbf y\in B_R^N\),
\begin{align*}
    \|\mathfrak m_N(\mathbf x)-\mathfrak m_N(\mathbf y)\|_2
    &\le
    \frac{\|A_N(\mathbf x)-A_N(\mathbf y)\|_2}{B_N(\mathbf x)}
    +
    \|A_N(\mathbf y)\|_2
    \left|
    \frac1{B_N(\mathbf x)}-\frac1{B_N(\mathbf y)}
    \right|                                                     \\
    &\le
    b_G^{-1}L_{A,R,N}\|\mathbf x-\mathbf y\|_{N,2}
    +
    \frac{e^{-\alpha G_{\min}}R}{\sqrt N}
    b_G^{-2}L_{B,N}\|\mathbf x-\mathbf y\|_{N,2}.
\end{align*}
Thus~\eqref{eq:fixed-N-consensus-local-Lipschitz} holds with
\[
    L_{\mathfrak m,R,N}
    :=
    b_G^{-1}L_{A,R,N}
    +
    \frac{e^{-\alpha G_{\min}}R}{\sqrt N}b_G^{-2}L_{B,N}.
\]

It remains to prove the growth estimate. Applying
Corollary~\ref{cor:consensus-moment-stability} to
\(\rho=\mu_{\mathbf x}^N\), we get
\[
    \|\mathfrak m_N(\mathbf x)\|_2^4
    \le
    C_m^4\int_{\mathbb R^d}\|\theta\|_2^4\,\mu_{\mathbf x}^N(d\theta)
    =
    C_m^4\frac1N\sum_{i=1}^N\|x^i\|_2^4.
\]
Since
\[
    \sum_{i=1}^N\|x^i\|_2^4
    \le
    \left(\sum_{i=1}^N\|x^i\|_2^2\right)^2
    =
    \|\mathbf x\|_{N,2}^4,
\]
we obtain
$
    \|\mathfrak m_N(\mathbf x)\|_2^4
    \le
    C_m^4N^{-1}\|\mathbf x\|_{N,2}^4.
$
Taking fourth roots gives~\eqref{eq:fixed-N-consensus-linear-growth}.
\end{proof}

\begin{proposition}
\label{prop:fixed-N-well-posedness}
Under Assumption~\ref{ass:primitive}, fix \(N\in\mathbb N\). Suppose that
$
    \mathbf\Theta_0^N
    :=
    (\Theta_{0,\tau}^{1,N},\ldots,\Theta_{0,\tau}^{N,N})
$
is an \((\mathbb R^d)^N\)-valued random variable with
$
    \mathbb E\|\mathbf\Theta_0^N\|_{N,2}^2<\infty.
$
Then the interacting particle system~\eqref{eq:interacting_sde} admits a unique
nonexplosive strong solution
\[
    \mathbf\Theta_t^N
    =
    (\Theta_{t,\tau}^{1,N},\ldots,\Theta_{t,\tau}^{N,N})
\]
with continuous paths on every finite time interval.
\end{proposition}

\begin{proof}
Write the particle system in vector form. For
\(\mathbf x=(x^1,\ldots,x^N)\in(\mathbb R^d)^N\), define
\[
    b_N^i(\mathbf x)
    :=
    -\lambda\left(x^i-\mathfrak m_N(\mathbf x)\right),
    \qquad i\in[N],
\]
and
\[
    \Sigma_N^i(\mathbf x)
    :=
    \sigma D\left(x^i-\mathfrak m_N(\mathbf x)\right),
    \qquad i\in[N].
\]
Let
\[
    b_N(\mathbf x)
    :=
    (b_N^1(\mathbf x),\ldots,b_N^N(\mathbf x))
    \in(\mathbb R^d)^N,
\]
and let \(\Sigma_N(\mathbf x)\) be the block-diagonal \(dN\times dN\) matrix
with diagonal blocks \(\Sigma_N^1(\mathbf x),\ldots,\Sigma_N^N(\mathbf x)\).
Then~\eqref{eq:interacting_sde} is equivalent to
\[
    d\mathbf\Theta_t^N
    =
    b_N(\mathbf\Theta_t^N)\,dt
    +
    \Sigma_N(\mathbf\Theta_t^N)\,d\mathbf W_t^N,
    \label{eq:fixed-N-vector-SDE}
\]
where
$
    \mathbf W_t^N:=(W_t^1,\ldots,W_t^N)
$
is a \(dN\)-dimensional Brownian motion. We first prove local Lipschitz continuity of \(b_N\) and \(\Sigma_N\). Fix
\(R>0\) and let \(\mathbf x,\mathbf y\in B_R^N\). By
Lemma~\ref{lem:fixed-N-consensus-regularity},
\[
    \|\mathfrak m_N(\mathbf x)-\mathfrak m_N(\mathbf y)\|_2
    \le
    L_{\mathfrak m,R,N}\|\mathbf x-\mathbf y\|_{N,2}.
\]
Therefore,
\begin{align*}
    \|b_N(\mathbf x)-b_N(\mathbf y)\|_{N,2}^2
    &=
    \lambda^2
    \sum_{i=1}^N
    \left\|
    x^i-y^i
    -
    \left(\mathfrak m_N(\mathbf x)-\mathfrak m_N(\mathbf y)\right)
    \right\|_2^2                                      \\
    &\le
    2\lambda^2\|\mathbf x-\mathbf y\|_{N,2}^2
    +
    2\lambda^2N
    \|\mathfrak m_N(\mathbf x)-\mathfrak m_N(\mathbf y)\|_2^2 \\
    &\le
    2\lambda^2
    \left(1+N L_{\mathfrak m,R,N}^2\right)
    \|\mathbf x-\mathbf y\|_{N,2}^2.
    \label{eq:fixed-N-drift-local-Lipschitz-new}
\end{align*}
Using \(\|D(u)-D(v)\|_F\le\sqrt d\,\|u-v\|_2\), we similarly obtain
\begin{align*}
    \|\Sigma_N(\mathbf x)-\Sigma_N(\mathbf y)\|_F^2
    &=
    \sum_{i=1}^N
    \left\|
    \Sigma_N^i(\mathbf x)-\Sigma_N^i(\mathbf y)
    \right\|_F^2                                      \\
    &\le
    \sigma^2 d
    \sum_{i=1}^N
    \left\|
    x^i-y^i
    -
    \left(\mathfrak m_N(\mathbf x)-\mathfrak m_N(\mathbf y)\right)
    \right\|_2^2                                      \\
    &\le
    2\sigma^2d
    \left(1+N L_{\mathfrak m,R,N}^2\right)
    \|\mathbf x-\mathbf y\|_{N,2}^2.
    %\label{eq:fixed-N-diffusion-local-Lipschitz-new}
\end{align*}
Hence \(b_N\) and \(\Sigma_N\) are locally Lipschitz on
\((\mathbb R^d)^N\).

We next prove linear growth. By
Lemma~\ref{lem:fixed-N-consensus-regularity},
$
    \|\mathfrak m_N(\mathbf x)\|_2
    \le
    C_mN^{-1/4}\|\mathbf x\|_{N,2}.
$
Thus
\begin{align}\notag
    \sum_{i=1}^N
    \|x^i-\mathfrak m_N(\mathbf x)\|_2^2
    &\le
    2\sum_{i=1}^N\|x^i\|_2^2
    +
2N\|\mathfrak m_N(\mathbf x)\|_2^2  \\ 
    &\le
    2\|\mathbf x\|_{N,2}^2
    +
    2N C_m^2N^{-1/2}\|\mathbf x\|_{N,2}^2                    \\
    &=
    2\left(1+C_m^2N^{1/2}\right)
    \|\mathbf x\|_{N,2}^2.
    \label{eq:fixed-N-x-minus-m-linear-growth}
\end{align}
Since \(D(z)=\|z\|_2 I_d\), we have
$
    \|\Sigma_N(\mathbf x)\|_F^2
    =
    \sigma^2d
    \sum_{i=1}^N
    \|x^i-\mathfrak m_N(\mathbf x)\|_2^2,
$
and
$
    \|b_N(\mathbf x)\|_{N,2}^2
    =
    \lambda^2
    \sum_{i=1}^N
    \|x^i-\mathfrak m_N(\mathbf x)\|_2^2.
$
Combining these identities with~\eqref{eq:fixed-N-x-minus-m-linear-growth}
gives
\[
    \|b_N(\mathbf x)\|_{N,2}^2
    +
    \|\Sigma_N(\mathbf x)\|_F^2
    \le
    2(\lambda^2+\sigma^2d)
    \left(1+C_m^2N^{1/2}\right)
    \|\mathbf x\|_{N,2}^2.
    \label{eq:fixed-N-coefficients-linear-growth}
\]
Therefore the coefficients are locally Lipschitz and satisfy a linear-growth
bound.

By the standard finite-dimensional strong existence and uniqueness theorem
for SDEs with locally Lipschitz coefficients and linear growth, the vector
SDE~\eqref{eq:fixed-N-vector-SDE} has a unique nonexplosive strong solution on
every finite time interval. Since solutions of finite-dimensional Itô SDEs
with continuous coefficients have continuous paths, the particle paths
\[
    t\mapsto \Theta_{t,\tau}^{i,N},
    \qquad i\in[N],
\]
are continuous. This proves the claim.
\end{proof}

\subsection{Fourth-moment boundedness with explicit constants}
Let $B_{4,d}$ be any fixed BDG constant such that, for every continuous $d$-dimensional stochastic integral $\mathcal M_t=\int_0^tH_s\dd W_s$,
\begin{equation}
    \E\sup_{t\le T}\norm{\mathcal M_t}_2^4
    \le B_{4,d}\E\br{\int_0^T\norm{H_s}_F^2\dd s}^2.
    \label{eq:BDG-constant}
\end{equation}
In the application below, the martingale is
\[
    \mathcal M_t
    =
    \int_0^t
    D(\Theta_{s,\tau}^{i,N}-m_s^N)\,dW_s^i,
\]
so the BDG integrand is
$
    H_s=D(\Theta_{s,\tau}^{i,N}-m_s^N).
$
Since \(D(z)=\|z\|_2I_d\), we have
$
    \|H_s\|_F^2=d\|\Theta_{s,\tau}^{i,N}-m_s^N\|_2^2.
$
Moreover, let
\begin{equation}
    K_{\mathrm{bd}}(T)
    :=216\br{1+C_m(T)^4}\br{\lambda^4T^3+\sigma^4B_{4,d}d^2T},
    \label{eq:Kbd-def}
\end{equation}
\begin{equation}
    C_{\mathrm{bd}}(T):=54\mu_4\exp\br{K_{\mathrm{bd}}(T)T}.
    \label{eq:Cbd-def}
\end{equation}

\begin{lemma}
\label{lem:fourth-moment-bound}
Under Assumption~\ref{ass:primitive}, for every $T>0$,
\begin{equation}
    \E[M_T]\le C_{\mathrm{bd}}(T),
    \qquad
    \Prob(\Omega_M^c)\le\frac{C_{\mathrm{bd}}(T)}{M}.
    \label{eq:moment-prob-bound}
\end{equation}
\end{lemma}

\begin{proof}
By Proposition~\ref{prop:fixed-N-well-posedness}, the interacting particle
system~\eqref{eq:interacting_sde} has a unique nonexplosive strong solution with
continuous paths. Hence the integral representation below is valid and Itô
estimates may be applied to each particle.

Set $m_t^N:=m(\rho_{t,\tau}^{N,\mathrm{int}})$. For fixed $i$,
\begin{equation*}
    \Theta_{t,\tau}^{i,N}
    =\Theta_{0,\tau}^{i,N}
    -\lambda\int_0^t\br{\Theta_{s,\tau}^{i,N}-m_s^N}\dd s
    +\sigma\int_0^tD\br{\Theta_{s,\tau}^{i,N}-m_s^N}\dd W_s^i.
\end{equation*}
Using $(a+b+c)^4\le27(a^4+b^4+c^4)$ gives
\begin{align}\notag
    \E\sup_{t\le T}\norm{\Theta_{t,\tau}^{i,N}}_2^4
    &\le 27\E\norm{\Theta_{0,\tau}^{i,N}}_2^4                                      \\ \notag
    &\quad+27\lambda^4\E\sup_{t\le T}\norm{\int_0^t\br{\Theta_{s,\tau}^{i,N}-m_s^N}\dd s}_2^4 \\
    &\quad+27\sigma^4\E\sup_{t\le T}\norm{\int_0^tD\br{\Theta_{s,\tau}^{i,N}-m_s^N}\dd W_s^i}_2^4.
    \label{eq:particle-fourth-start}
\end{align}
Jensen's inequality in time gives
\begin{equation}
    \sup_{t\le T}\norm{\int_0^t\br{\Theta_{s,\tau}^{i,N}-m_s^N}\dd s}_2^4
    \le T^3\int_0^T\norm{\Theta_{s,\tau}^{i,N}-m_s^N}_2^4\dd s.
    \label{eq:jensen-time}
\end{equation}
By $(a+b)^4\le8(a^4+b^4)$ and Corollary~\ref{cor:consensus-moment-stability},
\begin{equation}
    \norm{\Theta_{s,\tau}^{i,N}-m_s^N}_2^4
    \le 8\norm{\Theta_{s,\tau}^{i,N}}_2^4
    +8C_m(T)^4\frac1N\sum_{j=1}^N\norm{\Theta_{s,\tau}^{j,N}}_2^4.
    \label{eq:particle-minus-m}
\end{equation}
For the stochastic integral, equations~\eqref{eq:BDG-constant} and~\eqref{eq:D-basic}, followed by Cauchy's inequality in time, imply
\begin{align}\notag
    \E\sup_{t\le T}\norm{\int_0^tD\br{\Theta_{s,\tau}^{i,N}-m_s^N}\dd W_s^i}_2^4
    &\le B_{4,d}\E\br{\int_0^T\norm{D\br{\Theta_{s,\tau}^{i,N}-m_s^N}}_F^2\dd s}^2 \\ \notag
    &=B_{4,d}\E\br{\int_0^Td\norm{\Theta_{s,\tau}^{i,N}-m_s^N}_2^2\dd s}^2             \\
    &\le B_{4,d}d^2T\int_0^T\E\norm{\Theta_{s,\tau}^{i,N}-m_s^N}_2^4\dd s.
    \label{eq:stochastic-fourth}
\end{align}
Let
$
    U_N(t):=\frac1N\sum_{i=1}^N\E\sup_{u\le t}\norm{\Theta_{u,\tau}^{i,N}}_2^4.
$ 
Averaging~\eqref{eq:particle-fourth-start} over $i$, using~\eqref{eq:jensen-time}, \eqref{eq:particle-minus-m}, \eqref{eq:stochastic-fourth}, and
\begin{equation*}
    \E\norm{\Theta_{s,\tau}^{i,N}}_2^4
    \le \E\sup_{u\le s}\norm{\Theta_{u,\tau}^{i,N}}_2^4,
\end{equation*}
one obtains
\begin{equation*}
    U_N(T)\le27\mu_4+K_{\mathrm{bd}}(T)\int_0^TU_N(s)\dd s.
\end{equation*}
Gronwall's inequality gives
\begin{equation}
    U_N(T)\le27\mu_4\exp\br{K_{\mathrm{bd}}(T)T}.
    \label{eq:UN-bound}
\end{equation}
For the i.i.d. mean-field copies, the same proof applies. Indeed, Corollary~\ref{cor:consensus-moment-stability} gives
\begin{equation*}
    \norm{m(\rho_{s,\tau})}_2^4
    \le C_m(T)^4\E\norm{\bar\Theta_{s,\tau}^1}_2^4,
\end{equation*}
and hence
\begin{equation}
    \bar U_N(T):=\frac1N\sum_{i=1}^N\E\sup_{u\le T}\norm{\bar\Theta_{u,\tau}^{i}}_2^4
    \le27\mu_4\exp\br{K_{\mathrm{bd}}(T)T}.
    \label{eq:barUN-bound}
\end{equation}
Since $\max\{a,b\}\le a+b$, equations~\eqref{eq:UN-bound} and~\eqref{eq:barUN-bound} imply
\begin{equation*}
    \E[M_T]
    \le54\mu_4\exp\br{K_{\mathrm{bd}}(T)T}
    =C_{\mathrm{bd}}(T).
\end{equation*}
Finally, Markov's inequality gives
\begin{equation*}
    \Prob(\Omega_M^c)=\Prob(M_T>M)
    \le\frac{\E[M_T]}{M}
    \le\frac{C_{\mathrm{bd}}(T)}{M}.
\end{equation*}
\end{proof}

\section{Proof of Theorem \ref{thm:direct-mf-convergence}}\label{pp:thm:direct-mf-convergence}

\subsection{Moment-based soft Laplace principle}

\begin{proposition}
\label{prop:soft-laplace}
Under Assumption~\ref{ass:primitive}, let $\varrho\in\cP_2(\R^d)$, $q=q_{\beta,\tau}^L[\varrho]$, and $J=\softJ[\varrho]$. Furthermore, assume that, for some $r_G\in(0,R_G]$, $r\in(0,r_G]$, $u>0$, and $\delta_{\mathrm{lev}}>0$,
\begin{equation}
    q+\delta_{\mathrm{lev}}
    \le L_{\min}+\min\set{L_\infty,(\eta_Lr_G)^{1/\nu_L}},
    \label{eq:laplace-level-assump}
\end{equation}
\begin{equation}
    u+\Gtilde_r\le G_\infty,
    \label{eq:laplace-mass-assump}
\end{equation}
where
$
    \thetae_G:=\thetae_{r_G},
$
$
    \Gtilde_r:=\sup_{\theta\in B_r(\theta^\star)}\br{G(\theta)-G(\thetae_G)},
$ $
    \Delta_r:=\sup_{\theta\in B_r(\theta^\star)}G(\theta)-G_{\min}.
    %\label{eq:laplace-local-constants}
$
Assume additionally that
\begin{equation}
    \varrho(B_r(\theta^\star))>0.
    \label{eq:soft-laplace-positive-local-mass}
\end{equation}
Set
$
    c_{\mathrm{in}}(q,r):=\psi\br{\frac{q-L_{\min}-L_Lr}{\tau}},
$ and $
    c_{\mathrm{out}}(\delta_{\mathrm{lev}}):=\psi\br{-\frac{\delta_{\mathrm{lev}}}{\tau}}.
    %\label{eq:cin-cout-def}
$
Then,
\begin{align}
    \norm{m(\varrho)-\theta^\star}_2
    &\le r_G+\frac{(u+\Gtilde_r)^{\nu_G}}{\eta_G}                                      \nonumber\\
    &\quad+\frac{e^{-\alpha u}}{c_{\mathrm{in}}(q,r)\varrho(B_r(\theta^\star))}
    \int_{\R^d}\norm{\theta-\thetae_G}_2J(\dd\theta)                                     \nonumber\\
    &\quad+\frac{e^{\alpha\Delta_r}c_{\mathrm{out}}(\delta_{\mathrm{lev}})}{c_{\mathrm{in}}(q,r)\varrho(B_r(\theta^\star))}
    \int_{\R^d}\norm{\theta-\thetae_G}_2\varrho(\dd\theta).
    \label{eq:soft-laplace-raw}
\end{align}
Let
$
    V_\varrho:=\frac12\int_{\R^d}\norm{\theta-\theta^\star}_2^2\varrho(\dd\theta),
    %\label{eq:V-varrho}
$
then
\begin{align}
    \norm{m(\varrho)-\theta^\star}_2
    &\le r_G+\frac{(u+\Gtilde_r)^{\nu_G}}{\eta_G}                                      \nonumber\\
    &\quad+\frac{e^{-\alpha u}}{c_{\mathrm{in}}(q,r)\varrho(B_r(\theta^\star))}
    \br{\sqrt{2V_\varrho}+\beta r_G}                                                     \nonumber\\
    &\quad+\frac{e^{\alpha\Delta_r}c_{\mathrm{out}}(\delta_{\mathrm{lev}})}{c_{\mathrm{in}}(q,r)\varrho(B_r(\theta^\star))}
    \br{\sqrt{2V_\varrho}+r_G}.
    \label{eq:soft-laplace-V}
\end{align}
\end{proposition}
\begin{proof}
Write $m=m(\varrho)$, $\eta=\eta_{\varrho,\tau}$, and $\thetae=\thetae_G$.

First, if $\theta\in B_r(\theta^\star)$, then global Lipschitz continuity of $L$ and $L(\theta^\star)=L_{\min}$ give
\begin{equation*}
    L(\theta)-L_{\min}\le L_L\norm{\theta-\theta^\star}_2\le L_Lr.
\end{equation*}
Thus,
\begin{equation*}
    \eta(\theta)=\psi\br{\frac{q-L(\theta)}{\tau}}
    \ge \psi\br{\frac{q-L_{\min}-L_Lr}{\tau}}
    =c_{\mathrm{in}}(q,r),
\end{equation*}
and therefore
\begin{equation}
    J(B_r(\theta^\star))
    =\int_{B_r(\theta^\star)}\eta(\theta)\varrho(\dd\theta)
    \ge c_{\mathrm{in}}(q,r)\varrho(B_r(\theta^\star)).
    \label{eq:J-ball-lower}
\end{equation}

By~\eqref{eq:soft-laplace-positive-local-mass} and
\(c_{\mathrm{in}}(q,r)>0\), the lower bound in
\eqref{eq:J-ball-lower} is strictly positive. Hence the denominators appearing
below are well defined.

Second, take $\theta\notin N_{r_G}(\Theta)$. Then $\dist(\theta,\Theta)>r_G$. We claim that
\begin{equation}
    L(\theta)-L_{\min}>\min\set{L_\infty,(\eta_Lr_G)^{1/\nu_L}}.
    \label{eq:outside-level-claim}
\end{equation}
Indeed, if $L(\theta)-L_{\min}>L_\infty$, the claim is immediate. If $L(\theta)-L_{\min}\le L_\infty$, then~\eqref{eq:L-growth} applies. If in addition $L(\theta)-L_{\min}\le(\eta_Lr_G)^{1/\nu_L}$, then
\begin{equation*}
    \dist(\theta,\Theta)
    \le\eta_L^{-1}\br{L(\theta)-L_{\min}}^{\nu_L}
    \le\eta_L^{-1}\eta_Lr_G=r_G,
\end{equation*}
contradicting $\theta\notin N_{r_G}(\Theta)$. Hence~\eqref{eq:outside-level-claim} holds. Combining~\eqref{eq:outside-level-claim} with~\eqref{eq:laplace-level-assump} gives
$
q-L(\theta)\le -\delta_{\mathrm{lev}},
$ and 
$ \theta\in N_{r_G}(\Theta)^c.
$ 
Since $\psi$ is increasing,
\begin{equation}
    \eta(\theta)\le \psi\br{-\frac{\delta_{\mathrm{lev}}}{\tau}}
    =c_{\mathrm{out}}(\delta_{\mathrm{lev}}),
    \qquad \theta\in N_{r_G}(\Theta)^c.
    \label{eq:eta-outside-upper}
\end{equation}

Let
\begin{equation*}
    Z:=\int_{\R^d}e^{-\alpha G(\theta)}J(\dd\theta).
\end{equation*}
By~\eqref{eq:J-ball-lower}, we get
\begin{equation}
    Z
    \ge\int_{B_r(\theta^\star)}e^{-\alpha G(\theta)}J(\dd\theta)
    \ge e^{-\alpha\sup_{B_r(\theta^\star)}G}c_{\mathrm{in}}(q,r)\varrho(B_r(\theta^\star)).
    \label{eq:Z-lower-ball}
\end{equation}
Since $\thetae\in B_{r_G}(\theta^\star)$, one has $\|\thetae-\theta^\star\|_2<r_G$, and hence
\begin{equation}
    \norm{m-\theta^\star}_2
    \le r_G+\norm{m-\thetae}_2.
    \label{eq:m-theta-star-split}
\end{equation}
Define
$
    \tilde r:=\eta_G^{-1}(u+\Gtilde_r)^{\nu_G}.
$
Then
\begin{equation}
    \norm{m-\thetae}_2\le T_1+T_2+T_3,
    \label{eq:T123-split}
\end{equation}
where
\begin{align*}
    T_1&:=Z^{-1}\int_{B_{\tilde r}(\thetae)}\norm{\theta-\thetae}_2e^{-\alpha G(\theta)}J(\dd\theta),\nonumber\\
    T_2&:=Z^{-1}\int_{N_{r_G}(\Theta)\setminus B_{\tilde r}(\thetae)}\norm{\theta-\thetae}_2e^{-\alpha G(\theta)}J(\dd\theta),\nonumber\\
    T_3&:=Z^{-1}\int_{N_{r_G}(\Theta)^c}\norm{\theta-\thetae}_2e^{-\alpha G(\theta)}J(\dd\theta).
\end{align*}
The term $T_1$ satisfies $T_1\le\tilde r$ because $Z$ is the full normalizing denominator and $\norm{\theta-\thetae}_2\le\tilde r$ on $B_{\tilde r}(\thetae)$. 
For $T_2$, take $\theta\in N_{r_G}(\Theta)\setminus B_{\tilde r}(\thetae)$. If $G(\theta)-G(\thetae)\le u+\Gtilde_r$, then by~\eqref{eq:laplace-mass-assump} we have $G(\theta)-G(\thetae)\le G_\infty$, so~\eqref{eq:G-growth} gives
\begin{equation*}
    \norm{\theta-\thetae}_2
    \le\eta_G^{-1}\br{G(\theta)-G(\thetae)}^{\nu_G}
    \le\eta_G^{-1}(u+\Gtilde_r)^{\nu_G}
    =\tilde r,
\end{equation*}
contradicting $\theta\notin B_{\tilde r}(\thetae)$. Hence
\begin{equation}
    G(\theta)-G(\thetae)>u+\Gtilde_r.
    \label{eq:G-large-T2}
\end{equation}
For every $\xi\in B_r(\theta^\star)$, by the definition of $\Gtilde_r$, we get
$
    G(\xi)-G(\thetae)\le\Gtilde_r.
$
Combining this with~\eqref{eq:G-large-T2} gives $G(\theta)-G(\xi)>u$, hence
\begin{equation*}
    e^{-\alpha G(\theta)}\le e^{-\alpha u}e^{-\alpha G(\xi)},
    \qquad \xi\in B_r(\theta^\star).
\end{equation*}
Integrating in $\xi$ over $B_r(\theta^\star)$ with respect to $J$ gives
\begin{equation*}
    e^{-\alpha G(\theta)}J(B_r(\theta^\star))
    \le e^{-\alpha u}\int_{B_r(\theta^\star)}e^{-\alpha G(\xi)}J(\dd\xi)
    \le e^{-\alpha u}Z.
\end{equation*}
Using~\eqref{eq:J-ball-lower}, we obtain
\begin{equation*}
    e^{-\alpha G(\theta)}
    \le \frac{e^{-\alpha u}Z}{c_{\mathrm{in}}(q,r)\varrho(B_r(\theta^\star))}.
\end{equation*}
Consequently,
\begin{equation}
    T_2
    \le \frac{e^{-\alpha u}}{c_{\mathrm{in}}(q,r)\varrho(B_r(\theta^\star))}
    \int_{\R^d}\norm{\theta-\thetae}_2J(\dd\theta).
    \label{eq:T2-bound}
\end{equation}

For $T_3$, use~\eqref{eq:eta-outside-upper}, $G\ge G_{\min}$, and~\eqref{eq:Z-lower-ball}:
\begin{align}\notag
    T_3
    &\le Z^{-1}e^{-\alpha G_{\min}}c_{\mathrm{out}}(\delta_{\mathrm{lev}})
    \int_{\R^d}\norm{\theta-\thetae}_2\varrho(\dd\theta)                              \\ \notag
    &\le \frac{e^{\alpha(\sup_{B_r(\theta^\star)}G-G_{\min})}c_{\mathrm{out}}(\delta_{\mathrm{lev}})}{c_{\mathrm{in}}(q,r)\varrho(B_r(\theta^\star))}
    \int_{\R^d}\norm{\theta-\thetae}_2\varrho(\dd\theta)                              \\
    &=\frac{e^{\alpha\Delta_r}c_{\mathrm{out}}(\delta_{\mathrm{lev}})}{c_{\mathrm{in}}(q,r)\varrho(B_r(\theta^\star))}
    \int_{\R^d}\norm{\theta-\thetae}_2\varrho(\dd\theta).
    \label{eq:T3-bound}
\end{align}
Combining~\eqref{eq:m-theta-star-split}, \eqref{eq:T123-split}, $T_1\le\tilde r$,~\eqref{eq:T2-bound}, and~\eqref{eq:T3-bound} proves~\eqref{eq:soft-laplace-raw}.

It remains to prove~\eqref{eq:soft-laplace-V}. Since $0<\eta<1$, $J\le\varrho$, and $J(\R^d)=\beta$,
\begin{align}\notag
    \int_{\R^d}\norm{\theta-\thetae_G}_2J(\dd\theta)
    &\le \int_{\R^d}\norm{\theta-\theta^\star}_2J(\dd\theta)
    +\norm{\theta^\star-\thetae_G}_2J(\R^d)                                         \\ \notag
    &\le \int_{\R^d}\norm{\theta-\theta^\star}_2\varrho(\dd\theta)+\beta r_G          \\
    &\le \sqrt{2V_\varrho}+\beta r_G.
    \label{eq:J-distance-bound}
\end{align}
The last step uses Cauchy--Schwarz under the probability measure $\varrho$. Similarly,
\begin{equation}
    \int_{\R^d}\norm{\theta-\thetae_G}_2\varrho(\dd\theta)
    \le \sqrt{2V_\varrho}+r_G.
    \label{eq:rho-distance-bound}
\end{equation}
Substituting~\eqref{eq:J-distance-bound} and~\eqref{eq:rho-distance-bound} into~\eqref{eq:soft-laplace-raw} proves~\eqref{eq:soft-laplace-V}.
\end{proof}

\subsection{Lower bound on the mass near the target}

\begin{proposition}
\label{prop:mass-lower-bound}
Let $T>0$ and $r>0$. Assume that $\rho_{t,\tau}$ solves the Fokker--Planck equation associated with~\eqref{eq:mf-sde} and that
\begin{equation}
    \sup_{t\in[0,T]}\int_{\R^d}\norm{x}_2^4\rho_{t,\tau}(\dd x)\le C_{\mathrm{mf},4}.
    \label{eq:Cmf4-assump}
\end{equation}
Define
$
    B_0:=C_m(T)C_{\mathrm{mf},4}^{1/4}+\norm{\theta^\star}_2,
   % \label{eq:B0-def}
$
and let
\begin{equation*}
    \varphi_r(\theta):=
    \begin{cases}
    \displaystyle \exp\br{1-\frac{r^2}{r^2-\norm{\theta-\theta^\star}_2^2}},
    &\norm{\theta-\theta^\star}_2<r,\\[0.7em]
    0,&\norm{\theta-\theta^\star}_2\ge r.
    \end{cases}
  %  \label{eq:bump-def}
\end{equation*}
Choose $c\in(1/2,1)$ satisfying
\begin{equation}
    d(1-c)^2\le c(2c-1).
    \label{eq:c-condition}
\end{equation}
Define
\begin{equation*}
    K_1:=\frac{2\lambda(cr+B_0\sqrt c)}{(1-c)^2r}
    +\frac{2\sigma^2(cr^2+B_0^2)(2c+d)}{(1-c)^4r^2},
    \qquad
    K_2:=\frac{\lambda^2}{\sigma^2c(2c-1)},
    %\label{eq:K1K2-def}
\end{equation*}
and
$
    p:=2\max\set{K_1,K_2}.
    %\label{eq:p-def}
$
Then, for every $t\in[0,T]$,
\begin{equation}
    \rho_{t,\tau}(B_r(\theta^\star))
    \ge \br{\int_{\R^d}\varphi_r(\theta)\rho_0(\dd\theta)}e^{-pt}.
    \label{eq:mass-lower-bound}
\end{equation}
\end{proposition}

\begin{proof}
Since $0\le\varphi_r\le1$ and $\mathrm{supp}(\varphi_r)\subset B_r(\theta^\star)$,
\begin{equation}
    \rho_{t,\tau}(B_r(\theta^\star))
    \ge \int_{\R^d}\varphi_r(\theta)\rho_{t,\tau}(\dd\theta).
    \label{eq:mass-bump}
\end{equation}
By Corollary~\ref{cor:consensus-moment-stability} and~\eqref{eq:Cmf4-assump},
\begin{equation*}
    \norm{m(\rho_{t,\tau})-\theta^\star}_2
    \le \norm{m(\rho_{t,\tau})}_2+\norm{\theta^\star}_2
    \le C_m(T)C_{\mathrm{mf},4}^{1/4}+\norm{\theta^\star}_2
    =B_0.
   % \label{eq:mt-theta-bound}
\end{equation*}
Set
$
    y:=\theta-\theta^\star, 
    s_\theta:=\norm{y}_2^2, 
    h:=r^2-s_\theta, 
    m_t:=m(\rho_{t,\tau}),
$ and $
    w_t:=m_t-\theta^\star.
$
Inside $B_r(\theta^\star)$, direct differentiation gives
\begin{equation}
    \nabla\varphi_r(\theta)
    =-\frac{2r^2y}{h^2}\varphi_r(\theta),
    \label{eq:bump-gradient}
\end{equation}
\begin{equation}
    \Delta\varphi_r(\theta)
    =\frac{2r^2}{h^4}\br{2(2s_\theta-r^2)s_\theta-dh^2}\varphi_r(\theta).
    \label{eq:bump-laplacian}
\end{equation}
Outside $B_r(\theta^\star)$, the function and its first two derivatives vanish. The weak formulation gives
\begin{equation}
    \frac{\dd}{\dd t}\int\varphi_r\dd\rho_{t,\tau}
    =\int_{\R^d}\br{T_1(\theta,t)+T_2(\theta,t)}\rho_{t,\tau}(\dd\theta),
    \label{eq:bump-weak}
\end{equation}
where, it uses ~\eqref{eq:D-basic}, and 
$
    T_1:=-\lambda\inner{\theta-m_t}{\nabla\varphi_r(\theta)},
 $ and $
    T_2:=\frac{\sigma^2}{2}\norm{\theta-m_t}_2^2\Delta\varphi_r(\theta).
$
We prove pointwise that
\begin{equation}
    T_1(\theta,t)+T_2(\theta,t)\ge -p\varphi_r(\theta).
    \label{eq:pointwise-bump}
\end{equation}
The claim is trivial outside $B_r(\theta^\star)$. For the inside the ball, we consider two cases. 
Write $z:=\theta-m_t=y-w_t$. 

\emph{Case 1: $s_\theta\le cr^2$.}
Then $h\ge(1-c)r^2$ and $\norm{y}_2\le\sqrt c\,r$. By~\eqref{eq:bump-gradient},
\begin{align}\notag
    T_1
    &=\frac{2\lambda r^2}{h^2}\inner{z}{y}\varphi_r                                      \\ \notag
    &\ge -\frac{2\lambda r^2}{h^2}\br{\norm{y}_2^2+\norm{w_t}_2\norm{y}_2}\varphi_r          \\
    &\ge -\frac{2\lambda(cr+B_0\sqrt c)}{(1-c)^2r}\varphi_r.
    \label{eq:T1-case1}
\end{align}
Also,
\begin{equation}
    \norm{z}_2^2\le2\br{\norm{y}_2^2+\norm{w_t}_2^2}
    \le2(cr^2+B_0^2).
    \label{eq:z-bound-case1}
\end{equation}
Since $s_\theta\le cr^2$ and $h\le r^2$,
\begin{equation}
    2(2s_\theta-r^2)s_\theta-dh^2
    \ge -(2c+d)r^4.
    \label{eq:laplace-lower-case1}
\end{equation}
Using~\eqref{eq:bump-laplacian}, \eqref{eq:z-bound-case1}, \eqref{eq:laplace-lower-case1}, and $h\ge(1-c)r^2$ gives
\begin{equation}
    T_2\ge -\frac{2\sigma^2(cr^2+B_0^2)(2c+d)}{(1-c)^4r^2}\varphi_r.
    \label{eq:T2-case1}
\end{equation}
Equations~\eqref{eq:T1-case1} and~\eqref{eq:T2-case1} imply $T_1+T_2\ge-K_1\varphi_r\ge-p\varphi_r$.

\emph{Case 2: $cr^2<s_\theta<r^2$.}
Since $s_\theta>cr^2$ and $h^2\le(1-c)^2r^4$,
\begin{align} \notag
    2(2s_\theta-r^2)s_\theta-dh^2
    &\ge 2(2c-1)cr^4-d(1-c)^2r^4                                      \\
    &\ge c(2c-1)r^4,
    \label{eq:case2-positive-laplace}
\end{align}
where the last step uses~\eqref{eq:c-condition}. Hence
\begin{equation}
    T_2\ge \sigma^2\frac{r^2}{h^4}c(2c-1)r^4\norm{z}_2^2\varphi_r.
    \label{eq:T2-case2}
\end{equation}
Moreover,
\begin{equation}
    T_1
    =\frac{2\lambda r^2}{h^2}\inner{z}{y}\varphi_r
    \ge -\frac{2\lambda r^2}{h^2}\norm{z}_2\norm{y}_2\varphi_r.
    \label{eq:T1-case2}
\end{equation}
For $X:=\norm{z}_2$, the expression $aX^2-bX$ is bounded below by $-b^2/(4a)$ with
$
    a:=\sigma^2\frac{r^2c(2c-1)r^4}{h^4},
 $ and $
    b:=\frac{2\lambda r^2\norm{y}_2}{h^2}.
$
Combining~\eqref{eq:T2-case2} and~\eqref{eq:T1-case2} yields
\begin{equation*}
    T_1+T_2
    \ge -\frac{\lambda^2\norm{y}_2^2}{\sigma^2c(2c-1)r^2}\varphi_r
    \ge -\frac{\lambda^2}{\sigma^2c(2c-1)}\varphi_r
    =-K_2\varphi_r
    \ge-p\varphi_r.
\end{equation*}
Thus~\eqref{eq:pointwise-bump} holds in all cases. Substituting it into~\eqref{eq:bump-weak} gives
\begin{equation*}
    \frac{\dd}{\dd t}\int\varphi_r\dd\rho_{t,\tau}
    \ge -p\int\varphi_r\dd\rho_{t,\tau}.
\end{equation*}
Gronwall's inequality gives
\begin{equation*}
    \int\varphi_r\dd\rho_{t,\tau}
    \ge \br{\int\varphi_r\dd\rho_0}e^{-pt}.
\end{equation*}
Combining this with~\eqref{eq:mass-bump} proves~\eqref{eq:mass-lower-bound}.
\end{proof}

\subsection{Direct mean-field convergence}
Recall 
 \begin{equation*}
     V_\tau(t):=\frac12\int_{\R^d}\norm{\theta-\theta^\star}_2^2\rho_{t,\tau}(\dd\theta),
     \qquad
     M_\tau(t):=\norm{m(\rho_{t,\tau})-\theta^\star}_2.
 \end{equation*}

\begin{lemma}
\label{lem:Lyapunov-DI}
For almost every $t$,
\begin{equation}
    -aV_\tau(t)-b\sqrt{V_\tau(t)}M_\tau(t)
    \le \dot V_\tau(t)
    \le -aV_\tau(t)+b\sqrt{V_\tau(t)}M_\tau(t)+c_dM_\tau(t)^2,
    \label{eq:Lyap-DI}
\end{equation}
where
$
    a:=2\lambda-d\sigma^2,
$$
    b:=\sqrt2(\lambda+d\sigma^2),
$$
    c_d:=\frac{d\sigma^2}{2}.
 %   \label{eq:abc-def}
$
\end{lemma}
\begin{proof}
Use the test function $\varphi(\theta)=\frac12\norm{\theta-\theta^\star}_2^2$. Since $\nabla\varphi=\theta-\theta^\star$ and $\nabla^2\varphi=\Id$, the weak formulation and~\eqref{eq:D-basic} give
\begin{equation}
    \dot V_\tau(t)
    =-\lambda\int\inner{\theta-\theta^\star}{\theta-m_t}\rho_{t,\tau}(\dd\theta)
    +\frac{d\sigma^2}{2}\int\norm{\theta-m_t}_2^2\rho_{t,\tau}(\dd\theta),
    \label{eq:Vdot-start}
\end{equation}
where $m_t=m(\rho_{t,\tau})$. Let $\bar\theta_t:=\int\theta\rho_{t,\tau}(\dd\theta)$. Expanding around $\theta^\star$ gives
\begin{equation}
    -\lambda\int\inner{\theta-\theta^\star}{\theta-m_t}\rho_{t,\tau}(\dd\theta)
    =-2\lambda V_\tau(t)+\lambda\inner{\bar\theta_t-\theta^\star}{m_t-\theta^\star}.
    \label{eq:drift-expand}
\end{equation}
Similarly,
\begin{equation}
    \frac{d\sigma^2}{2}\int\norm{\theta-m_t}_2^2\rho_{t,\tau}(\dd\theta)
    =d\sigma^2V_\tau(t)-d\sigma^2\inner{\bar\theta_t-\theta^\star}{m_t-\theta^\star}
    +\frac{d\sigma^2}{2}M_\tau(t)^2.
    \label{eq:diffusion-expand}
\end{equation}
Combining~\eqref{eq:Vdot-start}, \eqref{eq:drift-expand}, and~\eqref{eq:diffusion-expand},
\begin{equation}
    \dot V_\tau(t)
    =-aV_\tau(t)+(\lambda+d\sigma^2)\inner{\bar\theta_t-\theta^\star}{m_t-\theta^\star}+c_dM_\tau(t)^2.
    \label{eq:Vdot-exact}
\end{equation}
Finally,
\begin{equation*}
    \norm{\bar\theta_t-\theta^\star}_2
    \le\br{\int\norm{\theta-\theta^\star}_2^2\rho_{t,\tau}(\dd\theta)}^{1/2}
    =\sqrt{2V_\tau(t)}.
\end{equation*}
Using this bound in~\eqref{eq:Vdot-exact} gives the upper inequality in~\eqref{eq:Lyap-DI}. For the lower inequality, use
\begin{equation}
    (\lambda+d\sigma^2)\inner{\bar\theta_t-\theta^\star}{m_t-\theta^\star}
    \ge -b\sqrt{V_\tau(t)}M_\tau(t)
\end{equation}
and then keep the nonnegative term $c_dM_\tau(t)^2\ge0$ or drop it from below. This proves~\eqref{eq:Lyap-DI}.
\end{proof}

\textbf{Theorem \ref{thm:direct-mf-convergence}.}
%\begin{theorem}
\textit{Under Assumptions~\ref{ass:primitive} and \ref{ass:selector},  assume that
\(\theta^\star\in\operatorname{supp}(\rho_0)\). 
% Assume in addition that the
% soft selector has a two-sided exponential left tail in the following sense:
% there exist constants \(c_->0\) and \(z_->0\) such that
% \begin{equation}
%     \psi(-z)\ge c_-e^{-z},
%     \qquad
%     z\ge z_-.
%     \label{eq:selector-left-tail-lower}
% \end{equation}
Fix
\begin{equation}
    \varepsilon
    \in
    \left(
        0,
        \frac12 W_2^2(\rho_0,\delta_{\theta^\star})
    \right),
    \qquad
    \vartheta\in(0,1).
    \label{eq:direct-mf-eps-vartheta}
\end{equation}
Choose \(\lambda,\sigma>0\) such that
\begin{equation}
    2\lambda>d\sigma^2.
    \label{eq:direct-mf-dissipativity-condition}
\end{equation}
Define
\begin{equation}
    T^\star
    :=
    \frac{1}{(1-\vartheta)(2\lambda-d\sigma^2)}
    \log\left(
        \frac{W_2^2(\rho_0,\delta_{\theta^\star})/2}
        {\varepsilon}
    \right).
    \label{eq:direct-mf-Tstar}
\end{equation}
Then there exist \(\tau_0>0\) and functions
\(
    \tau\mapsto \beta_\tau\in(0,1),
\) and \(
    \tau\mapsto \alpha_\tau>0,
\)
defined for \(\tau\in(0,\tau_0]\), such that the following holds. For every \(\tau\in(0,\tau_0]\), set
$
    \alpha:=\alpha_\tau,
$ and $
    \beta:=\beta_\tau.
 %   \label{eq:direct-mf-alpha-beta-choice}
$
then there exists
\begin{equation}
    T
    \in
    \left[
        \frac{1-\vartheta}{1+\vartheta/2}T^\star,
        T^\star
    \right]
    \label{eq:direct-mf-hitting-time-range}
\end{equation}
such that
$
    \frac12
    W_2^2(\rho_{T,\tau},\delta_{\theta^\star})
    =
    \varepsilon.
    %\label{eq:direct-mf-hitting-identity}
$
Furthermore, on the time interval \([0,T]\),
\begin{equation}
    W_2^2(\rho_{t,\tau},\delta_{\theta^\star})
    \le
    W_2^2(\rho_0,\delta_{\theta^\star})
    \exp\left(
        -(1-\vartheta)(2\lambda-d\sigma^2)t
    \right),
    \qquad
    t\in[0,T].
    \label{eq:direct-mf-W2-exponential-decay}
\end{equation}}
%\end{theorem}

\begin{proof}
Set
$
    V_0
    :=
    \frac12 W_2^2(\rho_0,\delta_{\theta^\star}),
$ $
    a:=2\lambda-d\sigma^2,
$ $
    b:=\sqrt2(\lambda+d\sigma^2),
$ and $
    c_d:=\frac{d\sigma^2}{2}.
 %   \label{eq:proof-abc-def}
$
By~\eqref{eq:direct-mf-eps-vartheta}, one has
$
    V_0>\varepsilon.
   % \label{eq:proof-initial-above-eps}
$
By~\eqref{eq:direct-mf-dissipativity-condition}, one has \(a>0\). Define
\begin{equation*}
    c_\vartheta
    :=
    \min\left\{
        1,
        \frac{\vartheta a}{4b},
        \sqrt{\frac{\vartheta a}{4c_d}}
    \right\}.
    %\label{eq:proof-cvartheta-def}
\end{equation*}
The constants and parameters will be chosen in the following order. First we
choose the geometric constants \(r_G,u,\delta_{\rm lev},\bar r\). Then we
construct \(\tau_0>0\), and for each \(\tau\in(0,\tau_0]\) we define
\(\beta_\tau\) and \(\alpha_\tau\). Only after this construction is complete
will we fix an arbitrary \(\tau\in(0,\tau_0]\) and prove the convergence
statement for the corresponding mean-field solution.

We first choose the auxiliary geometric parameters. Choose
\(r_G\in(0,R_G]\) and \(u>0\) so small that
\begin{equation}
    u+2L_G r_G\le G_\infty,
    \label{eq:proof-upper-range}
\end{equation}
and
\begin{equation}
    r_G+
    \frac{(u+2L_G r_G)^{\nu_G}}{\eta_G}
    \le
    \frac{c_\vartheta\sqrt{\varepsilon}}{4}.
    \label{eq:proof-geometric-smallness}
\end{equation}
This is possible because the left-hand side tends to \(0\) as
\(r_G\downarrow0\) and \(u\downarrow0\). Define
$
    h_L(r_G)
    :=
    \min\{L_\infty,(\eta_Lr_G)^{1/\nu_L}\}.
   % \label{eq:proof-hL-def}
$
Choose
$
    \delta_{\rm lev}\in(0,h_L(r_G)).
    %\label{eq:proof-deltalev-choice}
$
Let
$
    \Delta_{\max}:=G_{\max}-G_{\min},
$ $
    \gamma:=1+\frac{\Delta_{\max}}{u}.
    %\label{eq:proof-Delta-gamma-def}
$
Choose a fixed radius \(\bar r>0\), independent of \(\tau,\beta,\alpha\), such that
\begin{equation*}
    \bar r
    <
    \min\left\{
        r_G,
        \frac{h_L(r_G)-\delta_{\rm lev}}{2L_L},
        \frac{\delta_{\rm lev}}{2\gamma L_L}
    \right\}.
    %\label{eq:proof-fixed-r-choice}
\end{equation*}
Then
$
    h_L(r_G)-\delta_{\rm lev}-L_L\bar r>0,
    %\label{eq:proof-positive-level-margin}
$
and
\begin{equation}
    \delta_{\rm lev}-\gamma L_L\bar r>0.
    \label{eq:proof-leakage-margin}
\end{equation}

We now construct a lower mass bound on \(B_{\bar r}(\theta^\star)\) using
the same stopped-bootstrap idea as in the proof of Theorem~2.7 of CB\(^2\)O. Set
$
    \bar B:=c_\vartheta\sqrt{V_0}.
    %\label{eq:proof-Bbar-def}
$
Fix any \(c\in(1/2,1)\) satisfying
$
    d(1-c)^2\le c(2c-1).
    %\label{eq:proof-c-choice}
$
Let
\begin{equation*}
    \bar K_1
    :=
    \frac{2\lambda(c\bar r+\bar B\sqrt c)}
    {(1-c)^2\bar r}
    +
    \frac{2\sigma^2(c\bar r^2+\bar B^2)(2c+d)}
    {(1-c)^4\bar r^2},
    \qquad
    \bar K_2
    :=
    \frac{\lambda^2}{\sigma^2c(2c-1)},
 %   \label{eq:proof-Kbar-def}
\end{equation*}
and
$
    \bar p:=2\max\{\bar K_1,\bar K_2\}.
    %\label{eq:proof-pbar-def}
$
Since \(\theta^\star\in\operatorname{supp}(\rho_0)\), the mollifier
\(\phi_{\bar r}\) from Proposition~\ref{prop:mass-lower-bound} satisfies
\begin{equation*}
    \int_{\mathbb R^d}\phi_{\bar r}(\theta)\rho_0(d\theta)>0.
  %  \label{eq:proof-positive-mollified-mass}
\end{equation*}
Let
$
    \bar m
    :=
    \left(
        \int_{\mathbb R^d}
        \phi_{\bar r}(\theta)\rho_0(d\theta)
    \right)
    e^{-\bar pT^\star}.
    %\label{eq:proof-mbar-def}
$
Then \(\bar m>0\), and \(\bar m\le 1\).

We next choose \(\tau_0>0\) small enough so that for all
\(\tau\in(0,\tau_0]\),
\begin{equation*}
    S_\tau
    :=
    \psi\left(
        \frac{
            h_L(r_G)-\delta_{\rm lev}-L_L\bar r
        }{\tau}
    \right)
    \ge \frac12.
    %\label{eq:proof-Stau-half}
\end{equation*}
For such \(\tau\), define
$
    \beta_\tau
    :=
    \frac12\bar m S_\tau.
    %\label{eq:proof-beta-tau-def}
$
Then
$
    0<\beta_\tau<1,
$ 
$    \beta_\tau\le \bar m S_\tau.
    %\label{eq:proof-beta-tau-basic}
$

Set
\begin{equation}
    z_{\beta_\tau}:=\psi^{-1}(\beta_\tau),
    \qquad
    \underline c_{\rm in}(\tau)
    :=
    \psi\left(
        z_{\beta_\tau}
        -
        \frac{L_L\bar r}{\tau}
    \right),
    %\label{eq:proof-cinbar-def}
\end{equation}
and
$
    \Gamma_\tau
    :=
    \bar m\,\underline c_{\rm in}(\tau).
    %\label{eq:proof-Gamma-tau-def}
$
Since \(S_\tau\to1\) as \(\tau\downarrow0\), we have
$
    \beta_\tau\to \frac12\bar m.
    %\label{eq:proof-beta-limit}
$
Therefore
\(
    z_{\beta_\tau}\to z_\star:=\psi^{-1}(\bar m/2).
\)
By reducing \(\tau_0\) if necessary, we may assume that
\begin{equation*}
    z_{\beta_\tau}\ge z_\star-1,
    \qquad
    \frac{L_L\bar r}{\tau}-z_{\beta_\tau}\ge z_-,
    \qquad
    \tau\in(0,\tau_0].
    %\label{eq:proof-zbeta-lower}
\end{equation*}
Using the lower tail bound, i.e., Assumption \ref{ass:selector}, we obtain
\begin{align}
    \underline c_{\rm in}(\tau)
    &=
    \psi\left(
        -\left(
            \frac{L_L\bar r}{\tau}
            -
            z_{\beta_\tau}
        \right)
    \right)\ge
    c_-
    \exp\left(
        -\frac{L_L\bar r}{\tau}
        +
        z_{\beta_\tau}
    \right)\ge
    c_-e^{z_\star-1}
    \exp\left(
        -\frac{L_L\bar r}{\tau}
    \right).
    \label{eq:proof-cinbar-lower}
\end{align}
Consequently,
\begin{equation}
    \Gamma_\tau
    \ge
    C_\Gamma
    \exp\left(
        -\frac{L_L\bar r}{\tau}
    \right),
    \qquad
    C_\Gamma:=\bar m c_-e^{z_\star-1}>0.
    \label{eq:proof-Gamma-lower}
\end{equation}

Moreover, define
$
    \Delta_{\bar r}
    :=
    \sup_{\theta\in B_{\bar r}(\theta^\star)}
    G(\theta)-G_{\min},
 $ $
    c_{\rm out}(\tau)
    :=
    \psi\left(-\frac{\delta_{\rm lev}}{\tau}\right).
    %\label{eq:proof-Delta-cout-def}
$
Clearly,
\begin{equation}
    0\le \Delta_{\bar r}\le \Delta_{\max}.
    \label{eq:proof-Delta-bar-bound}
\end{equation}

By Assumption~\ref{ass:selector},
\begin{equation}
    c_{\rm out}(\tau)
    \le
    C_-
    \exp\left(
        -\frac{\delta_{\rm lev}}{\tau}
    \right).
    \label{eq:proof-cout-upper}
\end{equation}
Choose a constant \(K>1\) large enough such that
\begin{equation}
    \frac{\sqrt2}{K}\le \frac{c_\vartheta}{4},
    \qquad
    \frac{r_G}{K}\le \frac{c_\vartheta\sqrt\varepsilon}{8}.
    \label{eq:proof-K-choice}
\end{equation}
For each \(\tau\in(0,\tau_0]\), define
\begin{equation}
    \alpha_\tau
    :=
    \frac1u
    \log\left(
        \frac{K}{\Gamma_\tau}
    \right).
    \label{eq:proof-alpha-tau-def}
\end{equation}
By further reducing \(\tau_0\) if necessary, \(\Gamma_\tau<K\), and hence
\(\alpha_\tau>0\).

We now verify that the leakage term is small for sufficiently small
\(\tau\). By \eqref{eq:proof-alpha-tau-def},
\begin{equation}
    e^{-\alpha_\tau u}
    =
    \frac{\Gamma_\tau}{K}.
    \label{eq:proof-alpha-first-identity}
\end{equation}
Furthermore, by \eqref{eq:proof-Delta-bar-bound},
\eqref{eq:proof-cout-upper}, and \eqref{eq:proof-Gamma-lower},
\begin{align}
    \frac{
        e^{\alpha_\tau\Delta_{\bar r}}c_{\rm out}(\tau)
    }{
        \Gamma_\tau
    }
    &=
    K^{\Delta_{\bar r}/u}
    c_{\rm out}(\tau)
    \Gamma_\tau^{-1-\Delta_{\bar r}/u}
    \nonumber\\
    &\le
    C_-K^{\Delta_{\max}/u}
    C_\Gamma^{-1-\Delta_{\max}/u}
    \exp\left(
        -\frac{
            \delta_{\rm lev}
            -
            (1+\Delta_{\max}/u)L_L\bar r
        }{\tau}
    \right)
    \nonumber\\
    &=
    C_-K^{\Delta_{\max}/u}
    C_\Gamma^{-1-\Delta_{\max}/u}
    \exp\left(
        -\frac{
            \delta_{\rm lev}
            -
            \gamma L_L\bar r
        }{\tau}
    \right).
    \label{eq:proof-leakage-ratio-small}
\end{align}
By \eqref{eq:proof-leakage-margin}, the exponent in
\eqref{eq:proof-leakage-ratio-small} is strictly negative as
\(\tau\downarrow0\). Hence, after reducing \(\tau_0\) once more, we have
\begin{equation}
    \frac{
        e^{\alpha_\tau\Delta_{\bar r}}c_{\rm out}(\tau)
    }{
        \Gamma_\tau
    }
    \le
    \min\left\{
        \frac{c_\vartheta}{4\sqrt2},
        \frac{c_\vartheta\sqrt\varepsilon}{8r_G}
    \right\},
    \qquad
    \tau\in(0,\tau_0].
    \label{eq:proof-leakage-small-final}
\end{equation}

Combining \eqref{eq:proof-alpha-first-identity},
\eqref{eq:proof-K-choice}, and \eqref{eq:proof-leakage-small-final}, we get
\begin{equation}
    \frac{
        r_G\left(
            \beta_\tau e^{-\alpha_\tau u}
            +
            e^{\alpha_\tau\Delta_{\bar r}}c_{\rm out}(\tau)
        \right)
    }{
        \Gamma_\tau
    }
    \le
    \frac{c_\vartheta\sqrt\varepsilon}{4},
    \label{eq:proof-admissibility-A-verified}
\end{equation}
and
\begin{equation}
    \frac{
        \sqrt2\left(
            e^{-\alpha_\tau u}
            +
            e^{\alpha_\tau\Delta_{\bar r}}c_{\rm out}(\tau)
        \right)
    }{
        \Gamma_\tau
    }
    \le
    \frac{c_\vartheta}{2}.
    \label{eq:proof-admissibility-B-verified}
\end{equation}
Thus the functions
$
    \tau\mapsto \beta_\tau,
$ $
    \tau\mapsto \alpha_\tau
$
are well defined on \((0,\tau_0]\), satisfy
$
    \beta_\tau\in(0,1),
$ and $
    \alpha_\tau>0,
$
and obey the smallness estimates
\eqref{eq:proof-admissibility-A-verified} and
\eqref{eq:proof-admissibility-B-verified} for every
\(\tau\in(0,\tau_0]\).

We now fix an arbitrary \(\tau\in(0,\tau_0]\). From this point on, set
$
    \alpha:=\alpha_\tau,
$ and $
    \beta:=\beta_\tau.
   % \label{eq:proof-fixed-alpha-beta-after-construction}
$
All quantities
$
    q_{\beta,\tau}^L[\cdot],
$ $
    J_{\beta,\tau}^L[\cdot],
    $ $
    A(\cdot),
    $ $
    B(\cdot),
    $ and $
    m(\cdot)
$
below are computed with this fixed parameters, 
$
    (\alpha_\tau,\beta_\tau,\tau,\lambda,\sigma).
$
Let
$
    \rho:=(\rho_{t,\tau})_{t\in[0,T^\star]}
$
be the mean-field solution specified in the theorem statement for this fixed
choice of parameters. Define
\begin{equation}
    V_\tau(t)
    :=
    \frac12
    \int_{\mathbb R^d}
    \|\theta-\theta^\star\|_2^2\,
    \rho_{t,\tau}(d\theta)
    =
    \frac12
    W_2^2(\rho_{t,\tau},\delta_{\theta^\star}).
    \label{eq:proof-Vtau-def}
\end{equation}
Then
$
    V_\tau(0)=V_0>\varepsilon.
    %\label{eq:proof-Vtau-initial}
$
Let 
$
    M_\tau(t):=\|m(\rho_{t,\tau})-\theta^\star\|_2.
   % \label{eq:proof-Mtau-def}
$
 and
\begin{equation}
    T_{\alpha_\tau,\beta_\tau,\tau}
    :=
    \sup
    \left\{
        t\in[0,T^\star]:
        V_\tau(s)>\varepsilon
        \ \text{and}\
        M_\tau(s)<c_\vartheta\sqrt{V_\tau(s)}
        \ \text{for all }s\in[0,t]
    \right\}.
    \label{eq:proof-stopping-time-def}
\end{equation}

We first show that \(T_{\alpha_\tau,\beta_\tau,\tau}>0\). To this end, we
establish the consensus estimate at \(t=0\). Since \(M_\tau(0)\) is finite,
the argument below applies on the degenerate interval \([0,0]\), and the
mass estimate \(\rho_0(B_{\bar r}(\theta^\star))\ge \bar m\) is immediate
from the definition of \(\bar m\). The same computation as below yields
\[
    M_\tau(0)
    \le
    \frac{c_\vartheta\sqrt\varepsilon}{2}
    +
    \frac{c_\vartheta}{2}\sqrt{V_0}
    <
    c_\vartheta\sqrt{V_0},
\]
because \(V_0>\varepsilon\). By continuity of \(V_\tau\) and \(M_\tau\),
this proves
$
    T_{\alpha_\tau,\beta_\tau,\tau}>0.
    %\label{eq:proof-stopping-time-positive}
$

Let \(t<T_{\alpha_\tau,\beta_\tau,\tau}\). On \([0,t]\), by their definitions,
\(V_\tau(s)>\varepsilon\) and
\(
    M_\tau(s)<c_\vartheta\sqrt{V_\tau(s)}.
\)
Using Lemma~\ref{lem:Lyapunov-DI}, for a.e. \(s\in[0,t]\),
\begin{align}
    \dot V_\tau(s)
    &\le
    -aV_\tau(s)
    +
    b\sqrt{V_\tau(s)}M_\tau(s)
    +
    c_dM_\tau(s)^2
    \nonumber\\
    &\le
    -
    \left(
        a-bc_\vartheta-c_dc_\vartheta^2
    \right)V_\tau(s).
    \label{eq:proof-bootstrap-upper-DI}
\end{align}
By the definition of \(c_\vartheta\), we have 
$
    bc_\vartheta\le \frac{\vartheta a}{4},
 $ and $
    c_dc_\vartheta^2\le \frac{\vartheta a}{4}.
$
Hence
\begin{equation*}
    \dot V_\tau(s)\le -(1-\vartheta)aV_\tau(s),
    \qquad
    \text{for a.e. }s\in[0,t].
    %\label{eq:proof-bootstrap-decay}
\end{equation*}
In particular,
\begin{equation*}
    V_\tau(s)\le V_0,
    \qquad
    M_\tau(s)\le \bar B,
    \qquad
    s\in[0,t].
   % \label{eq:proof-bootstrap-Bbar-bound}
\end{equation*}

Repeating the proof of Proposition~\ref{prop:mass-lower-bound} on the
interval \([0,t]\), with the a priori bound
\(
    \sup_{s\in[0,t]}M_\tau(s)\le \bar B
\)
in place of the moment-based bound \(B_0\), gives
\begin{equation}
    \rho_{s,\tau}(B_{\bar r}(\theta^\star))
    \ge
    \bar m,
    \qquad
    s\in[0,t].
    \label{eq:proof-local-mass-bootstrap}
\end{equation}
Indeed, the constants \(\bar K_1,\bar K_2,\bar p\) were chosen exactly as
in Proposition~\ref{prop:mass-lower-bound}, with \(r=\bar r\) and
\(B_0=\bar B\).

We now verify the quantile localization condition required by
Proposition~\ref{prop:soft-laplace}. Let
$
    q_0:=L_{\min}+h_L(r_G)-\delta_{\rm lev}.
   % \label{eq:proof-q0-def}
$
For every \(\theta\in B_{\bar r}(\theta^\star)\), the Lipschitz continuity
of \(L\) and \(L(\theta^\star)=L_{\min}\) imply
$
    L(\theta)\le L_{\min}+L_L\bar r.
    %\label{eq:proof-L-local-upper}
$
Therefore, by monotonicity of \(\psi\), by
\eqref{eq:proof-local-mass-bootstrap}, and by the definition of
\(\beta_\tau\),
\begin{align}
    \int_{\mathbb R^d}
    \psi\left(
        \frac{q_0-L(\theta)}{\tau}
    \right)
    \rho_{s,\tau}(d\theta)
    &\ge
    \rho_{s,\tau}(B_{\bar r}(\theta^\star))
    \psi\left(
        \frac{
            h_L(r_G)-\delta_{\rm lev}-L_L\bar r
        }{\tau}
    \right)
    \nonumber\\
    &\ge
    \bar m S_\tau
    =
    2\beta_\tau
    \ge
    \beta_\tau.
    \label{eq:proof-soft-quantile-comparison}
\end{align}
Since the soft-quantile defining map is strictly increasing in the quantile
variable,
$
    q_{\rho_{s,\tau}}\le q_0,
$ for $
    s\in[0,t].
   % \label{eq:proof-q-upper}
$
Consequently,
\begin{equation*}
    q_{\rho_{s,\tau}}+\delta_{\rm lev}
    \le
    L_{\min}+h_L(r_G),
    \qquad
    s\in[0,t].
   % \label{eq:proof-quantile-localization}
\end{equation*}

Next, by Lemma~\ref{lem:quantile-localization}, we obtain
\begin{equation*}
    q_{\rho_{s,\tau}}
    \ge
    L_{\min}+\tau z_{\beta_\tau}.
    %\label{eq:proof-q-lower}
\end{equation*}
This results in 
\begin{equation}
    \psi\left(
        \frac{
            q_{\rho_{s,\tau}}-L_{\min}-L_L\bar r
        }{\tau}
    \right)
    \ge
    \psi\left(
        z_{\beta_\tau}
        -
        \frac{L_L\bar r}{\tau}
    \right)
    =
    \underline c_{\rm in}(\tau).
    \label{eq:proof-cin-lower}
\end{equation}
Combining \eqref{eq:proof-local-mass-bootstrap} and
\eqref{eq:proof-cin-lower} gives
\begin{equation}
    \psi\left(
        \frac{
            q_{\rho_{s,\tau}}-L_{\min}-L_L\bar r
        }{\tau}
    \right)
    \rho_{s,\tau}(B_{\bar r}(\theta^\star))
    \ge
    \Gamma_\tau,
    \qquad
    s\in[0,t].
    \label{eq:proof-cin-mass-lower}
\end{equation}

Let
$
    \widetilde\theta_G:=\widetilde\theta_{r_G},
 $ and $
    G_{\rm er}
    :=
    \sup_{\theta\in B_{\bar r}(\theta^\star)}
    \left(
        G(\theta)-G(\widetilde\theta_G)
    \right).
    %\label{eq:proof-Ger-def}
$
Since \(\bar r\le r_G\), \(\widetilde\theta_G\in B_{r_G}(\theta^\star)\),
and \(G\) is \(L_G\)-Lipschitz, then
\begin{equation}
    G_{\rm er}\le L_G(\bar r+r_G)\le 2L_G r_G.
    \label{eq:proof-Ger-bound}
\end{equation}
Together with~\eqref{eq:proof-upper-range}, this implies
$
    u+G_{\rm er}\le G_\infty.
   % \label{eq:proof-u-Ger-range}
$

We may now apply Proposition~\ref{prop:soft-laplace} with
\(\varrho=\rho_{s,\tau}\) and \(r=\bar r\). For every \(s\in[0,t]\), it
yields
\begin{align}
    M_\tau(s)
    &\le
    r_G
    +
    \frac{(u+G_{\rm er})^{\nu_G}}{\eta_G}
    \nonumber\\
    &\quad
    +
    \frac{e^{-\alpha_\tau u}}
    {
        c_{\rm in}(q_{\rho_{s,\tau}},\bar r)
        \rho_{s,\tau}(B_{\bar r}(\theta^\star))
    }
    \left(
        \sqrt{2V_\tau(s)}+\beta_\tau r_G
    \right)
    \nonumber\\
    &\quad
    +
    \frac{
        e^{\alpha_\tau\Delta_{\bar r}}c_{\rm out}(\tau)
    }
    {
        c_{\rm in}(q_{\rho_{s,\tau}},\bar r)
        \rho_{s,\tau}(B_{\bar r}(\theta^\star))
    }
    \left(
        \sqrt{2V_\tau(s)}+r_G
    \right).
    \label{eq:proof-soft-laplace-applied}
\end{align}
Using \eqref{eq:proof-cin-mass-lower} and
\eqref{eq:proof-Ger-bound}, we get
$
    M_\tau(s)
    \le
    A_\tau+B_\tau\sqrt{V_\tau(s)},
    $ for $
    s\in[0,t],
   % \label{eq:proof-consensus-affine-bound}
$
where
\begin{equation*}
    A_\tau
    :=
    r_G
    +
    \frac{(u+2L_G r_G)^{\nu_G}}{\eta_G}
    +
    \frac{
        r_G\left(
            \beta_\tau e^{-\alpha_\tau u}
            +
            e^{\alpha_\tau\Delta_{\bar r}}c_{\rm out}(\tau)
        \right)
    }{
        \Gamma_\tau
    },
  %  \label{eq:proof-A-def}
\end{equation*}
and
\begin{equation*}
    B_\tau
    :=
    \frac{
        \sqrt2\left(
            e^{-\alpha_\tau u}
            +
            e^{\alpha_\tau\Delta_{\bar r}}c_{\rm out}(\tau)
        \right)
    }{
        \Gamma_\tau
    }.
    %\label{eq:proof-B-def}
\end{equation*}
By \eqref{eq:proof-geometric-smallness},
\eqref{eq:proof-admissibility-A-verified}, and
\eqref{eq:proof-admissibility-B-verified}, we get
$
    A_\tau\le \frac{c_\vartheta\sqrt\varepsilon}{2},
 $ and $
    B_\tau\le \frac{c_\vartheta}{2}.
    %\label{eq:proof-A-B-small}
$
Therefore, whenever \(V_\tau(s)>\varepsilon\),
\begin{equation}
    M_\tau(s)
    \le
    \frac{c_\vartheta\sqrt\varepsilon}{2}
    +
    \frac{c_\vartheta}{2}\sqrt{V_\tau(s)}
    <
    c_\vartheta\sqrt{V_\tau(s)}.
    \label{eq:proof-consensus-control-when-above-eps}
\end{equation}
This is a strict improvement of the bootstrap condition.

We now show that the process cannot stop before \(T^\star\) because the
consensus-control condition fails. Suppose, to the contrary, that
$
    T_{\alpha_\tau,\beta_\tau,\tau}<T^\star
$ and $
    V_\tau(T_{\alpha_\tau,\beta_\tau,\tau})>\varepsilon.
   % \label{eq:proof-contradiction-assumption}
$
Applying the previous argument with \(t<T_{\alpha_\tau,\beta_\tau,\tau}\)
and passing to the limit \(t\uparrow T_{\alpha_\tau,\beta_\tau,\tau}\) by
continuity, we get
\begin{equation*}
    M_\tau(T_{\alpha_\tau,\beta_\tau,\tau})
    <
    c_\vartheta
    \sqrt{V_\tau(T_{\alpha_\tau,\beta_\tau,\tau})}.
  %  \label{eq:proof-consensus-still-strict}
\end{equation*}
By continuity, the defining conditions in
\eqref{eq:proof-stopping-time-def} remain true on a slightly larger
interval, contradicting the definition of
\(T_{\alpha_\tau,\beta_\tau,\tau}\). Therefore, either
$
    T_{\alpha_\tau,\beta_\tau,\tau}=T^\star,
  %  \label{eq:proof-stop-at-Tstar}
$
or
\begin{equation}
    V_\tau(T_{\alpha_\tau,\beta_\tau,\tau})\le\varepsilon.
    \label{eq:proof-stop-by-hitting}
\end{equation}

On every interval on which \(V_\tau>\varepsilon\), the estimate
\eqref{eq:proof-consensus-control-when-above-eps} holds. Therefore, by
Lemma~\ref{lem:Lyapunov-DI}, for a.e. such \(t\), we can write
\begin{equation}
    \dot V_\tau(t)
    \le
    -(1-\vartheta)aV_\tau(t),
    \label{eq:proof-upper-differential-final}
\end{equation}
and also
\begin{equation}
    \dot V_\tau(t)
    \ge
    -(1+\vartheta/2)aV_\tau(t).
    \label{eq:proof-lower-differential-final}
\end{equation}
Indeed, the first inequality follows from
\[
    \dot V_\tau(t)
    \le
    -aV_\tau(t)
    +
    b\sqrt{V_\tau(t)}M_\tau(t)
    +
    c_dM_\tau(t)^2
\]
and the bounds
$
    bc_\vartheta\le\frac{\vartheta a}{4},
 $ and $
    c_dc_\vartheta^2\le\frac{\vartheta a}{4}.
$
The second inequality follows from
\[
    \dot V_\tau(t)
    \ge
    -aV_\tau(t)
    -
    b\sqrt{V_\tau(t)}M_\tau(t)
    \ge
    -(a+bc_\vartheta)V_\tau(t)
    \ge
    -(1+\vartheta/2)aV_\tau(t).
\]

Define the first hitting time
$
    T
    :=
    \inf\{t\in[0,T^\star]:V_\tau(t)\le\varepsilon\}.
   % \label{eq:proof-hitting-time-def}
$
The hitting set is nonempty. Indeed, if the stopping time is strictly smaller than \(T^\star\), then \eqref{eq:proof-stop-by-hitting} gives nonemptiness. If the stopping time equals \(T^\star\), then Gronwall's inequality applied to \eqref{eq:proof-upper-differential-final} gives
\begin{equation*}
    V_\tau(T^\star)
    \le
    V_0 e^{-(1-\vartheta)aT^\star}
    =
    \varepsilon.
    %\label{eq:proof-hitting-at-Tstar-case}
\end{equation*}
Thus
\begin{equation}
    T\le T^\star.
    \label{eq:proof-T-upper-bound}
\end{equation}
By continuity of \(V_\tau\) and \(V_0>\varepsilon\),
$
    V_\tau(T)=\varepsilon.
   % \label{eq:proof-hitting-equality}
$

For every \(t<T\), one has \(V_\tau(t)>\varepsilon\), and hence
\eqref{eq:proof-lower-differential-final} applies. Letting \(t\uparrow T\)
after applying Gronwall's inequality gives
$
    \varepsilon
    =
    V_\tau(T)
    \ge
    V_0e^{-(1+\vartheta/2)aT}.
   % \label{eq:proof-lower-comparison-at-hit}
$
Consequently,
\begin{equation}
    T
    \ge
    \frac{1}{(1+\vartheta/2)a}
    \log\left(\frac{V_0}{\varepsilon}\right)
    =
    \frac{1-\vartheta}{1+\vartheta/2}T^\star.
    \label{eq:proof-T-lower-bound}
\end{equation}
Combining \eqref{eq:proof-T-upper-bound} and
\eqref{eq:proof-T-lower-bound} proves
\eqref{eq:direct-mf-hitting-time-range}.

Finally, by \eqref{eq:proof-Vtau-def},
$
    W_2^2(\rho_{t,\tau},\delta_{\theta^\star})
    =
    2V_\tau(t).
   % \label{eq:proof-W2-V-identity}
$
Using \eqref{eq:proof-upper-differential-final} on \([0,T]\), we obtain
\begin{align*}
    W_2^2(\rho_{t,\tau},\delta_{\theta^\star})
    &\le
    2V_0e^{-(1-\vartheta)at}
    \nonumber\\
    &=
    W_2^2(\rho_0,\delta_{\theta^\star})
    \exp\left(
        -(1-\vartheta)(2\lambda-d\sigma^2)t
    \right),
    \qquad
    t\in[0,T].
    %\label{eq:proof-final-W2-decay}
\end{align*}
This proves \eqref{eq:direct-mf-W2-exponential-decay} and completes the
proof.
\end{proof}

\begin{remark}
\label{rem:auxiliary-mass-alpha-beta}
The constants \(\beta_\tau\) and \(\alpha_\tau\) constructed in the proof should
be understood as one coupled admissible choice. In fact, the construction has
an additional degree of freedom. Suppose that the proof provides a positive
mass lower bound
\[
    \rho_{t,\tau}(B_{\bar r}(\theta^\star))\ge \bar m>0,
    \qquad t\in[0,T].
\]
Then the same lower bound remains valid if \(\bar m\) is replaced by any
auxiliary constant \(m\in(0,\bar m]\).

For such an \(m\), define
\[
    S_\tau
    :=
    \psi\!\left(
    \frac{h_L(r_G)-\delta_{\rm lev}-L_L\bar r}{\tau}
    \right),
    \qquad
    \beta_{\tau,m}:=\frac12 mS_\tau .
\]
Since \(0<S_\tau<1\), we have
\(
    0<\beta_{\tau,m}<\frac12 m.
\)
Hence \(\beta_{\tau,m}\) can be made arbitrarily small by choosing the auxiliary mass parameter \(m\) sufficiently small. 
Next, set
\(
    z_{\tau,m}:=\psi^{-1}(\beta_{\tau,m}),
\)
and 
\[
    c_{\rm in}^{m}(\tau)
    :=
    \psi\!\left(
    z_{\tau,m}-\frac{L_L\bar r}{\tau}
    \right),
    \qquad
    \Gamma_{\tau,m}:=m\,c_{\rm in}^{m}(\tau),
\]
and choose
\[
    \alpha_{\tau,m}
    :=
    \frac1u
    \log\!\left(
    \frac{K}{\Gamma_{\tau,m}}
    \right).
\]
Since \(0<c_{\rm in}^{m}(\tau)<1\), one has
\(
    0<\Gamma_{\tau,m}<m,
\)
and therefore
\[
    \alpha_{\tau,m}
    >
    \frac1u
    \log\!\left(
    \frac{K}{m}
    \right).
\]
Consequently, by taking \(m\downarrow0\), the admissible coupled construction
yields \(\beta_{\tau,m}\downarrow0\) and
\(\alpha_{\tau,m}\to+\infty\).
Thus the proof allows admissible choices with arbitrarily small \(\beta\) and
arbitrarily large \(\alpha\), provided these parameters are chosen through the
same coupled construction.
\end{remark}

\begin{remark}\label{rem:selection-rule-alpha-beta} 
The construction in Remark \ref{rem:auxiliary-mass-alpha-beta} can be made explicit by choosing the auxiliary
mass parameter as a function of \(\tau\). For instance, fix \(p>0\) and set
\(
    m_\tau := \min\{\bar m/2,\tau^p\}.
\)
Then define
\[
    S_\tau
    :=
    \psi\!\left(
    \frac{h_L(r_G)-\delta_{\rm lev}-L_L\bar r}{\tau}
    \right),
    \qquad
    \beta_\tau
    :=
    \frac12 m_\tau S_\tau,
\]
and
\[
    z_\tau:=\psi^{-1}(\beta_\tau),\qquad
    c_{\rm in}(\tau)
    :=
    \psi\!\left(
    z_\tau-\frac{L_L\bar r}{\tau}
    \right),
    \qquad
    \Gamma_\tau:=m_\tau c_{\rm in}(\tau),
\]
\[
    \alpha_\tau
    :=
    \frac1u\log\!\left(\frac{K}{\Gamma_\tau}\right).
\]
This gives one coupled admissible choice of
\((\alpha_\tau,\beta_\tau)\). In particular, if
\(h_L(r_G)-\delta_{\rm lev}-L_L\bar r>0\), then \(S_\tau\to1\) as
\(\tau\downarrow0\), and therefore
\(
    \beta_\tau \asymp \tau^p.
\)
For sigmoid-type selectors, one further obtains
\[
    \alpha_\tau
    =
    O(\tau^{-1}+|\log\tau|).
\]
Thus the construction makes the lower-level selection sharper as
\(\tau\downarrow0\), while increasing the Gibbs parameter only at an explicit
controlled rate.
\end{remark}

\section{Proof of Proposition \ref{prop:mfa-cutoff}}\label{pp:prop:mfa-cutoff}

Throughout this section, fix $T,M>0$. The cutoff time $\tau_M$ and the indicator $I_M(t)=\one_{\{t<\tau_M\}}$ are defined in~\eqref{eq:tauM-def} and~\eqref{eq:IM-def}. Define
$
    C_q:=\frac{L_\psi L_L}{\tau\kappa}
   % \label{eq:Cq-def}
$, 
$a_M:=a(M),$ 
$
    a_{\mathrm{mf}}:=a(C_{\mathrm{bd}}(T)),
$ and $
    a_{M,T}:=\max\set{a_M,a_{\mathrm{mf}}}.
$
Furthermore, let
\begin{align*}
    &C_{\phi,x}:=e^{-\alpha G_{\min}}\br{1+\alpha L_GM^{1/4}+\frac{L_\psi L_L}{\tau}M^{1/4}},
    \qquad
    C_{\phi,q}:=e^{-\alpha G_{\min}}\frac{L_\psi}{\tau}M^{1/4},
    %\label{eq:Cphi-def}
    \\
    &\widetilde C_\phi:=C_{\phi,x}+C_{\phi,q}C_q,\nonumber\\
    &C_{\psi,x}:=e^{-\alpha G_{\min}}\br{\alpha L_G+\frac{L_\psi L_L}{\tau}},
    \qquad
    C_{\psi,q}:=e^{-\alpha G_{\min}}\frac{L_\psi}{\tau},
    %\label{eq:Cpsi-def}
    \\
    &\widetilde C_\psi:=C_{\psi,x}+C_{\psi,q}C_q.\nonumber
\end{align*}
And define
\begin{align*}
    &C_{\mathrm{coup}}
    :=2b_G^{-2}\widetilde C_\phi^2+2a_{M,T}^2b_G^{-4}\widetilde C_\psi^2,\\
%    \label{eq:Ccoup-def}
&C_{q,\mathrm{emp}}:=\frac1{4\kappa^2},\\
&    C_{A,\mathrm{emp}}
    :=2C_{\phi,q}^2C_{q,\mathrm{emp}}+2e^{-2\alpha G_{\min}}C_{\mathrm{bd}}(T)^{1/2},
    \\
    &C_{B,\mathrm{emp}}
    :=2C_{\psi,q}^2C_{q,\mathrm{emp}}+2e^{-2\alpha G_{\min}},\\
    &C_{\mathrm{emp}}    :=2b_G^{-2}C_{A,\mathrm{emp}}+2a_{M,T}^2b_G^{-4}C_{B,\mathrm{emp}}.
\end{align*}
With
$
    C_0:=-\lambda+2\sigma^2d,
 $ $
    C_1:=\lambda+2\sigma^2d,
  $ $
    K_{\mathrm{MFA}}:=C_0+2C_1C_{\mathrm{coup}},
$
and let
\begin{equation*}
    C_{\mathrm{MFA}}
    :=4C_1C_{\mathrm{emp}}T\exp\br{\max\set{K_{\mathrm{MFA}},0}T}.
    %\label{eq:CMFA-def}
\end{equation*}

\begin{lemma}
\label{lem:empirical-second-moment}
For
$
    \phi_q(x):=xe^{-\alpha G(x)}\psi\br{\frac{q-L(x)}{\tau}},
$ and $
    \chi_q(x):=e^{-\alpha G(x)}\psi\br{\frac{q-L(x)}{\tau}},
    %\label{eq:phi-psi-def}
$ 
one has
\begin{equation*}
    \sup_{t\in[0,T]}
    \E\sqbr{\norm{\phi_{q_{\rho_{t,\tau}}}(\bar\Theta_{t,\tau}^1)}_2^2+
    \abs{\chi_{q_{\rho_{t,\tau}}}(\bar\Theta_{t,\tau}^1)}^2}
    \le C_{\mathrm{mom}}(T),
\end{equation*}
where
$
    C_{\mathrm{mom}}(T):=e^{-2\alpha G_{\min}}\br{C_{\mathrm{bd}}(T)^{1/2}+1}.
$
\end{lemma}
\begin{proof}
Since $0<\psi<1$ and $G\ge G_{\min}$,
\begin{equation*}
    \norm{\phi_{q_{\rho_{t,\tau}}}(x)}_2^2\le e^{-2\alpha G_{\min}}\norm{x}_2^2,
    \qquad
    \abs{\chi_{q_{\rho_{t,\tau}}}(x)}^2\le e^{-2\alpha G_{\min}}.
\end{equation*}
By Lemma~\ref{lem:fourth-moment-bound},
\begin{equation*}
    \E\norm{\bar\Theta_{t,\tau}^1}_2^2
    \le\br{\E\norm{\bar\Theta_{t,\tau}^1}_2^4}^{1/2}
    \le C_{\mathrm{bd}}(T)^{1/2}.
\end{equation*}
Combining the two estimates proves the claim.
\end{proof}

%\begin{proposition}
\textbf{Proposition \ref{prop:mfa-cutoff}.}
%\label{prop:mfa-cutoff}
\textit{Under Assumption~\ref{ass:primitive}, if
\begin{equation}
    \Prob(\Omega_M)\ge\frac12,
    \label{eq:Omega-half}
\end{equation}
then
\begin{equation}
    \max_{1\le i\le N}\sup_{t\in[0,T]}
    \E\sqbr{\norm{\Theta_{t,\tau}^{i,N}-\bar\Theta_{t,\tau}^i}_2^2\mid\Omega_M}
    \le\frac{C_{\mathrm{MFA}}}{N}.
    \label{eq:MFA-result}
\end{equation}}

\begin{proof}
Set
$
    e_t^i:=\Theta_{t,\tau}^{i,N}-\bar\Theta_{t,\tau}^i,
$ $
    q_t^{N,\mathrm{int}}:=q_{\rho_{t,\tau}^{N,\mathrm{int}}},
$ $
    \bar q_t^N:=q_{\bar\rho_{t,\tau}^{N}},
$ and $
    q_t:=q_{\rho_{t,\tau}}.
$
Apply Ito's formula to the stopped process $\norm{e_{t\wedge\tau_M}^i}_2^2$, localize the stochastic integral, and then use Lemma~\ref{lem:stopped-integral} with $Y_t=\norm{e_t^i}_2^2$. The stochastic integral has zero expectation after localization, and localization is removed by monotone convergence because the terms are nonnegative. Thus, in the distributional sense,
\begin{align}
    \frac{\dd}{\dd t}\E\sqbr{\norm{e_t^i}_2^2I_M(t)}
    &\le 2\E\sqbr{\inner{e_t^i}{-\lambda\br{e_t^i-\br{m(\rho_{t,\tau}^{N,\mathrm{int}})-m(\rho_{t,\tau})}}}I_M(t)}\nonumber\\
    &\quad+\sigma^2\E\sqbr{\norm{D(\Theta_{t,\tau}^{i,N}-m(\rho_{t,\tau}^{N,\mathrm{int}}))-D(\bar\Theta_{t,\tau}^i-m(\rho_{t,\tau}))}_F^2I_M(t)}.
    \label{eq:error-ito-stopped}
\end{align}
Using $2\inner{x}{y}\le\norm{x}_2^2+\norm{y}_2^2$ and~\eqref{eq:D-basic}, we obtain
\begin{equation}
    \frac{\dd}{\dd t}\E\sqbr{\norm{e_t^i}_2^2I_M(t)}
    \le C_0\E\sqbr{\norm{e_t^i}_2^2I_M(t)}
    +C_1\E\sqbr{\norm{m(\rho_{t,\tau}^{N,\mathrm{int}})-m(\rho_{t,\tau})}_2^2I_M(t)}.
    \label{eq:error-base}
\end{equation}

\emph{Step 1: quantile difference between the two empirical measures.}
On $\{I_M(t)=1\}$, Lemma~\ref{lem:quantile-inverse-stability} gives
\begin{equation*}
    \abs{q_t^{N,\mathrm{int}}-\bar q_t^N}I_M(t)
    \le \kappa^{-1}\abs{F_{\rho_{t,\tau}^{N,\mathrm{int}}}(\bar q_t^N)-F_{\bar\rho_{t,\tau}^{N}}(\bar q_t^N)}I_M(t).
\end{equation*}
Since $\psi$ and $L$ are Lipschitz,
\begin{align*}
    \abs{F_{\rho_{t,\tau}^{N,\mathrm{int}}}(\bar q_t^N)-F_{\bar\rho_{t,\tau}^{N}}(\bar q_t^N)}
    &\le\frac{L_\psi}{\tau N}\sum_{j=1}^N\abs{L(\Theta_{t,\tau}^{j,N})-L(\bar\Theta_{t,\tau}^{j})} \le\frac{L_\psi L_L}{\tau N}\sum_{j=1}^N\norm{e_t^j}_2.
\end{align*}
Combining the preceding two displays and using Cauchy's inequality for the empirical average,
\begin{equation}
    \abs{q_t^{N,\mathrm{int}}-\bar q_t^N}^2I_M(t)
    \le C_q^2\frac1N\sum_{j=1}^N\norm{e_t^j}_2^2I_M(t).
    \label{eq:q-N-bar-bound}
\end{equation}

\emph{Step 2: consensus difference between the two empirical measures.}
For arbitrary $x,y\in\R^d$ and $q_1,q_2\in\R$,
\begin{align}\notag
    \norm{\phi_{q_1}(x)-\phi_{q_2}(y)}_2
    &\le e^{-\alpha G_{\min}}\norm{x-y}_2
    +\alpha e^{-\alpha G_{\min}}L_G\norm{y}_2\norm{x-y}_2\nonumber\\
    &\quad+e^{-\alpha G_{\min}}\norm{x}_2\frac{L_\psi}{\tau}\br{\abs{q_1-q_2}+L_L\norm{x-y}_2}.
    \label{eq:phi-lipschitz}
\end{align}
Indeed,
\begin{align*}
    \norm{xe^{-\alpha G(x)}-ye^{-\alpha G(y)}}_2
    &\le e^{-\alpha G(x)}\norm{x-y}_2+\norm{y}_2\abs{e^{-\alpha G(x)}-e^{-\alpha G(y)}}              \\
    &\le e^{-\alpha G_{\min}}\norm{x-y}_2
    +\alpha e^{-\alpha G_{\min}}L_G\norm{y}_2\norm{x-y}_2,
\end{align*}
and the remaining term follows from the Lipschitz property of $\psi$ and $L$.

On $\{I_M(t)=1\}$, the cutoff bound~\eqref{eq:cutoff-pointwise-bound} gives
\begin{equation*}
    \frac1N\sum_{j=1}^N\norm{\Theta_{t,\tau}^{j,N}}_2\le M^{1/4},
    \qquad
    \frac1N\sum_{j=1}^N\norm{\bar\Theta_{t,\tau}^{j}}_2\norm{e_t^j}_2
    \le M^{1/4}\br{\frac1N\sum_{j=1}^N\norm{e_t^j}_2^2}^{1/2}.
    %\label{eq:cutoff-Cauchy}
\end{equation*}
The second inequality follows from Cauchy--Schwarz and
$$
N^{-1}\sum_j\norm{\bar\Theta_{t,\tau}^{j}}_2^2\le(N^{-1}\sum_j\norm{\bar\Theta_{t,\tau}^{j}}_2^4)^{1/2}\le M^{1/2}.
$$
Averaging~\eqref{eq:phi-lipschitz}, using~\eqref{eq:q-N-bar-bound}, and then squaring yields
\begin{equation}
    \norm{A(\rho_{t,\tau}^{N,\mathrm{int}})-A(\bar\rho_{t,\tau}^{N})}_2^2I_M(t)
    \le \widetilde C_\phi^2\frac1N\sum_{j=1}^N\norm{e_t^j}_2^2I_M(t).
    \label{eq:A-N-bar-bound}
\end{equation}
Similarly,
\begin{equation*}
    \abs{\psi_{q_1}(x)-\psi_{q_2}(y)}
    \le e^{-\alpha G_{\min}}\br{\alpha L_G+\frac{L_\psi L_L}{\tau}}\norm{x-y}_2
    +e^{-\alpha G_{\min}}\frac{L_\psi}{\tau}\abs{q_1-q_2},
\end{equation*}
which gives
\begin{equation}
    \abs{B(\rho_{t,\tau}^{N,\mathrm{int}})-B(\bar\rho_{t,\tau}^{N})}^2I_M(t)
    \le \widetilde C_\psi^2\frac1N\sum_{j=1}^N\norm{e_t^j}_2^2I_M(t).
    \label{eq:B-N-bar-bound}
\end{equation}
On $\{I_M(t)=1\}$, Lemma~\ref{lem:deterministic-AB-bounds} gives
$B(\rho_{t,\tau}^{N,\mathrm{int}})\ge b_G$ and $B(\bar\rho_{t,\tau}^{N})\ge b_G$. Moreover,
\begin{equation}
    \norm{A(\bar\rho_{t,\tau}^{N})}_2
    \le e^{-\alpha G_{\min}}\br{\int\norm{x}_2^2\softJ[\bar\rho_{t,\tau}^{N}](\dd x)}^{1/2}\beta^{1/2}
    \le a_M.
    \label{eq:Abar-aM}
\end{equation}
The quotient identity
\begin{equation}
    m(\rho)-m(\nu)
    =\frac{A(\rho)-A(\nu)}{B(\rho)}
    +A(\nu)\frac{B(\nu)-B(\rho)}{B(\rho)B(\nu)}
    \label{eq:quotient-identity}
\end{equation}
combined with~\eqref{eq:A-N-bar-bound}, \eqref{eq:B-N-bar-bound}, \eqref{eq:Abar-aM}, and $(a+b)^2\le2(a^2+b^2)$ gives
\begin{equation}
    \norm{m(\rho_{t,\tau}^{N,\mathrm{int}})-m(\bar\rho_{t,\tau}^{N})}_2^2I_M(t)
    \le C_{\mathrm{coup}}\frac1N\sum_{j=1}^N\norm{e_t^j}_2^2I_M(t).
    \label{eq:m-N-bar-bound}
\end{equation}

\emph{Step 3: empirical fluctuation of the mean-field sample.}
On $\{I_M(t)=1\}$, Lemma~\ref{lem:quantile-inverse-stability} applied with $\mu=\bar\rho_{t,\tau}^{N}$ and $\nu=\rho_{t,\tau}$ gives
\begin{equation*}
    \abs{\bar q_t^N-q_t}I_M(t)
    \le\kappa^{-1}\abs{F_t^N(q_t)-F_t(q_t)}I_M(t),
    %\label{eq:qbar-qt-inv}
\end{equation*}
where
\begin{equation*}
    F_t^N(q):=\frac1N\sum_{j=1}^N\psi\br{\frac{q-L(\bar\Theta_{t,\tau}^j)}{\tau}}-\beta,
    \qquad
    F_t(q):=\int \psi\br{\frac{q-L(x)}{\tau}}\rho_{t,\tau}(\dd x)-\beta.
\end{equation*}
The random variables
\begin{equation*}
    Z_j:=\psi\br{\frac{q_t-L(\bar\Theta_{t,\tau}^j)}{\tau}}
\end{equation*}
are i.i.d., lie in $[0,1]$, and have mean $\beta$. Hence $\mathrm{Var}(Z_j)\le1/4$ and
\begin{align}\notag
    \E\sqbr{\abs{\bar q_t^N-q_t}^2I_M(t)}
    &\le\kappa^{-2}\E\sqbr{\abs{F_t^N(q_t)-F_t(q_t)}^2I_M(t)}                 \\ \notag
    &\le\kappa^{-2}\E\abs{F_t^N(q_t)-F_t(q_t)}^2                              \\
    &\le\frac1{4\kappa^2N}
    =\frac{C_{q,\mathrm{emp}}}{N}.
    \label{eq:qbar-emp-bound}
\end{align}
Next,
\begin{align}
    A(\bar\rho_{t,\tau}^{N})\!-\!A(\rho_{t,\tau})
    \!=\!\frac1N\sum_{j=1}^N\br{\phi_{\bar q_t^N}(\bar\Theta_{t,\tau}^j)\!-\phi_{q_t}(\bar\Theta_{t,\tau}^j)}       \!+\!\Big(\frac1N\sum_{j=1}^N\phi_{q_t}(\bar\Theta_{t,\tau}^j)\!-\!\E\phi_{q_t}(\bar\Theta_{t,\tau}^1)\Big).
    \label{eq:A-emp-decomp}
\end{align}
The first term in~\eqref{eq:A-emp-decomp} has norm at most $C_{\phi,q}\abs{\bar q_t^N-q_t}$ on $\{I_M(t)=1\}$. The second term has second moment bounded by
\begin{equation*}
    \frac1N\E\norm{\phi_{q_t}(\bar\Theta_{t,\tau}^1)}_2^2
    \le\frac{e^{-2\alpha G_{\min}}C_{\mathrm{bd}}(T)^{1/2}}{N}.
    %\label{eq:A-variance-bound}
\end{equation*}
Using $(a+b)^2\le2(a^2+b^2)$ and~\eqref{eq:qbar-emp-bound},
\begin{equation}
    \E\sqbr{\norm{A(\bar\rho_{t,\tau}^{N})-A(\rho_{t,\tau})}_2^2I_M(t)}
    \le\frac{C_{A,\mathrm{emp}}}{N}.
    \label{eq:Abar-rho-bound}
\end{equation}
The same argument for $B$ gives
\begin{equation}
    \E\sqbr{\abs{B(\bar\rho_{t,\tau}^{N})-B(\rho_{t,\tau})}^2I_M(t)}
    \le\frac{C_{B,\mathrm{emp}}}{N}.
    \label{eq:Bbar-rho-bound}
\end{equation}
Furthermore, Lemmas~\ref{lem:deterministic-AB-bounds} and~\ref{lem:fourth-moment-bound} imply
\begin{equation}
    \norm{A(\rho_{t,\tau})}_2
    \le\beta^{1/2}e^{-\alpha G_{\min}}C_{\mathrm{bd}}(T)^{1/4}=a_{\mathrm{mf}}.
    \label{eq:Arho-amf}
\end{equation}
Using~\eqref{eq:quotient-identity}, the denominator bound $B\ge b_G$, equations~\eqref{eq:Abar-rho-bound}, \eqref{eq:Bbar-rho-bound}, \eqref{eq:Arho-amf}, and $(a+b)^2\le2(a^2+b^2)$, we obtain
\begin{equation}
    \E\sqbr{\norm{m(\bar\rho_{t,\tau}^{N})-m(\rho_{t,\tau})}_2^2I_M(t)}
    \le\frac{C_{\mathrm{emp}}}{N}.
    \label{eq:mbar-rho-bound}
\end{equation}

\emph{Step 4: Gronwall closure.}
By
\begin{equation*}
    m(\rho_{t,\tau}^{N,\mathrm{int}})-m(\rho_{t,\tau})
    =\br{m(\rho_{t,\tau}^{N,\mathrm{int}})-m(\bar\rho_{t,\tau}^{N})}
    +\br{m(\bar\rho_{t,\tau}^{N})-m(\rho_{t,\tau})},
\end{equation*}
then~\eqref{eq:m-N-bar-bound}, \eqref{eq:mbar-rho-bound}, and $(a+b)^2\le2(a^2+b^2)$ give
\begin{equation*}
    \E\sqbr{\norm{m(\rho_{t,\tau}^{N,\mathrm{int}})-m(\rho_{t,\tau})}_2^2I_M(t)}
    \le2C_{\mathrm{coup}}\frac1N\sum_{j=1}^N\E\sqbr{\norm{e_t^j}_2^2I_M(t)}+\frac{2C_{\mathrm{emp}}}{N}.
  %  \label{eq:total-m-bound}
\end{equation*}
Averaging~\eqref{eq:error-base} over $i$ and setting
$
    E_t:=\frac1N\sum_{i=1}^N\E\sqbr{\norm{e_t^i}_2^2I_M(t)},
$
we obtain
\begin{equation*}
    \dot E_t\le K_{\mathrm{MFA}}E_t+\frac{2C_1C_{\mathrm{emp}}}{N},
    \qquad E_0=0.
\end{equation*}
Therefore, for $t\in[0,T]$,
\begin{equation*}
E_t\le\frac{2C_1C_{\mathrm{emp}}T\exp\br{\max\set{K_{\mathrm{MFA}},0}T}}{N}.
    %\label{eq:Et-bound}
\end{equation*}
By the i.i.d. initialization and synchronous coupling in Assumption~\ref{ass:primitive}(A5), the pairs $(\Theta^{i,N},\bar\Theta^i)$ are exchangeable. Hence the same upper bound holds for each $\E\sqbr{\norm{e_t^i}_2^2I_M(t)}$. Since $\one_{\Omega_M}\le I_M(t)$ by~\eqref{eq:Omega-below-I} and~\eqref{eq:Omega-half} holds,
\begin{align*}
    \E\sqbr{\norm{e_t^i}_2^2\mid\Omega_M}
    &=\frac{\E\sqbr{\norm{e_t^i}_2^2\one_{\Omega_M}}}{\Prob(\Omega_M)}\le2\E\sqbr{\norm{e_t^i}_2^2I_M(t)}                                                         \\
    &\le\frac{4C_1C_{\mathrm{emp}}T\exp\br{\max\set{K_{\mathrm{MFA}},0}T}}{N}
    =\frac{C_{\mathrm{MFA}}}{N}.
\end{align*}
Taking the supremum over $t\in[0,T]$ and maximum over $i$ proves~\eqref{eq:MFA-result}.
\end{proof}

\section{Proof of Theorem~\ref{thm:numerical-convergence}}\label{pp:thm:numerical-convergence}

Define the empirical averages, as follows,
\begin{equation*}
    \widehat\theta_{K,\tau}^N:=\frac1N\sum_{i=1}^N\theta_{K,\tau}^i,
    \qquad
    \widehat\Theta_{T,\tau}^{N,\mathrm{int}}:=\frac1N\sum_{i=1}^N\Theta_{T,\tau}^{i,N},
    \qquad
    \widehat{\bar\Theta}_{T,\tau}^{N}:=\frac1N\sum_{i=1}^N\bar\Theta_{T,\tau}^{i}.
\end{equation*}
Young's inequality gives
\begin{align*}
    \norm{\widehat\theta_{K,\tau}^N-\theta^\star}_2^2
    &\le3\norm{\widehat\theta_{K,\tau}^N-\widehat\Theta_{T,\tau}^{N,\mathrm{int}}}_2^2
    +3\norm{\widehat\Theta_{T,\tau}^{N,\mathrm{int}}-\widehat{\bar\Theta}_{T,\tau}^{N}}_2^2+3\norm{\widehat{\bar\Theta}_{T,\tau}^{N}-\theta^\star}_2^2.
    %\label{eq:numerical-young}
\end{align*}

\subsection{Localized Euler--Maruyama approximation}

\begin{lemma}
\label{lem:Euler-fourth-moment-bound}
Under Assumption~\ref{ass:primitive}, fix \(T>0\), \(N\in\mathbb N\), and
let \(K\Delta t=T\) with \(0<\Delta t\le 1\). Suppose that the Euler--Maruyama
scheme~\eqref{eq:em_scheme} is driven by the same Brownian motions as the interacting
particle system, namely
\begin{equation}
    B_k^i
    :=
    \frac{W_{(k+1)\Delta t}^i-W_{k\Delta t}^i}{\sqrt{\Delta t}},
    \qquad
   k=0,\ldots,K-1,
    \quad i\in[N],
    \label{eq:Euler-Brownian-increments}
\end{equation}
and starts from the same initial data
\begin{equation}
    \theta_{0,\tau}^i=\Theta_{0,\tau}^{i,N},
    \qquad i\in[N].
    \label{eq:Euler-same-initial-data}
\end{equation}
Then there exists a finite constant
$
    C_{\mathrm{bd}}^\theta(T)
    =
    C_{\mathrm{bd}}^\theta
    \br{T,d,\lambda,\sigma,C_m(T),\mu_4}
    >0
    %\label{eq:Euler-Cbdtheta-dependence}
$
such that
\begin{equation}
    \E\sqbr{
    \max_{0\le k\le K}
    \frac1N\sum_{i=1}^N
    \norm{\theta_{k,\tau}^{i}}_2^4
    }
    \le
    C_{\mathrm{bd}}^\theta(T).
    \label{eq:Euler-fourth-moment-bound}
\end{equation}
\end{lemma}

\begin{proof}
For \(k=0,\ldots,K\), define
$
    Y_k^\theta
    :=
    \frac1N\sum_{i=1}^N
    \norm{\theta_{k,\tau}^{i}}_2^4.
    %\label{eq:Euler-Yk-def}
$
Also define
$
    m_k^N
    :=
    m(\rho_{k,\tau}^N),
 $ and $
    u_k^i
    :=
    \theta_{k,\tau}^i-m_k^N.
   % \label{eq:Euler-mk-uk-def}
$
By Corollary~\ref{cor:consensus-moment-stability},
\begin{equation*}
    \norm{m_k^N}_2^4
    \le
    C_m(T)^4
    \frac1N\sum_{j=1}^N
    \norm{\theta_{k,\tau}^{j}}_2^4.
    %\label{eq:Euler-mk-fourth-bound}
\end{equation*}
Therefore, using \(\norm{x-y}_2^4\le 8\norm{x}_2^4+8\norm{y}_2^4\), we get
\begin{align}
    \frac1N\sum_{i=1}^N
    \norm{u_k^i}_2^4
    &=
    \frac1N\sum_{i=1}^N
    \norm{\theta_{k,\tau}^i-m_k^N}_2^4
    \nonumber\\
    &\le
    8\frac1N\sum_{i=1}^N
    \norm{\theta_{k,\tau}^i}_2^4
    +
    8\norm{m_k^N}_2^4
    \nonumber\\
    &\le
    8\br{1+C_m(T)^4}Y_k^\theta.
    \label{eq:Euler-uk-fourth-bound}
\end{align}

From the Euler--Maruyama update~\eqref{eq:em_scheme},
\begin{equation*}
    \theta_{k+1,\tau}^i
    =
    \theta_{0,\tau}^i
    -
    \lambda\Delta t
    \sum_{\ell=0}^{k}u_\ell^i
    +
    \sigma\sqrt{\Delta t}
    \sum_{\ell=0}^{k}
    D(u_\ell^i)B_\ell^i.
    %\label{eq:Euler-expanded-update}
\end{equation*}
Hence, for every \(0\le r\le K\),
\begin{align}
    \norm{\theta_{r,\tau}^i}_2^4
    &\le
    27\norm{\theta_{0,\tau}^i}_2^4
    +
    27\lambda^4
    \norm{
    \Delta t\sum_{\ell=0}^{r-1}u_\ell^i
    }_2^4+
    27\sigma^4
    \norm{
    \sqrt{\Delta t}
    \sum_{\ell=0}^{r-1}
    D(u_\ell^i)B_\ell^i
    }_2^4.
    \label{eq:Euler-pathwise-fourth-split}
\end{align}
Taking the maximum over \(r=0,\ldots,k\), then expectation, gives
\begin{align}
    \E\sqbr{
    \max_{0\le r\le k}
    \norm{\theta_{r,\tau}^i}_2^4
    }
    &\le
    27\E\norm{\theta_{0,\tau}^i}_2^4
    +
    27\lambda^4
    \E\sqbr{
    \max_{0\le r\le k}
    \norm{
    \Delta t\sum_{\ell=0}^{r-1}u_\ell^i
    }_2^4
    }
    \nonumber\\
    &\quad+
    27\sigma^4
    \E\sqbr{
    \max_{0\le r\le k}
    \norm{
    \sqrt{\Delta t}
    \sum_{\ell=0}^{r-1}
    D(u_\ell^i)B_\ell^i
    }_2^4
    }.
    \label{eq:Euler-fourth-max-start}
\end{align}

We estimate the deterministic sum first. By the discrete Jensen inequality,
\begin{align}
    \max_{0\le r\le k}
    \norm{
    \Delta t\sum_{\ell=0}^{r-1}u_\ell^i
    }_2^4
    &\le
    T^3
    \Delta t
    \sum_{\ell=0}^{k-1}
    \norm{u_\ell^i}_2^4.
    \label{eq:Euler-drift-sum-bound}
\end{align}
For the stochastic sum, observe that
\[
    \sqrt{\Delta t}\sum_{\ell=0}^{r-1}D(u_\ell^i)B_\ell^i
    =
    \int_0^{r\Delta t}
    D\br{u_{\lfloor s/\Delta t\rfloor}^i}
    \dd W_s^i.
\]
Therefore, by the BDG inequality~\eqref{eq:BDG-constant},
\begin{align}
    &\E\sqbr{
    \max_{0\le r\le k}
    \norm{
    \sqrt{\Delta t}
    \sum_{\ell=0}^{r-1}
    D(u_\ell^i)B_\ell^i
    }_2^4
    }
    \nonumber\\
    &\qquad\le
    B_{4,d}
    \E\sqbr{
    \br{
    \int_0^{k\Delta t}
    \norm{
    D\br{u_{\lfloor s/\Delta t\rfloor}^i}
    }_F^2
    \dd s
    }^2
    }
    \nonumber\\
    &\qquad=
    B_{4,d}
    \E\sqbr{
    \br{
    \Delta t
    \sum_{\ell=0}^{k-1}
    \norm{D(u_\ell^i)}_F^2
    }^2
    }.
    \label{eq:Euler-martingale-BDG-start}
\end{align}
Since \(D(z)=\norm{z}_2\Id\), we have
\begin{equation}
    \norm{D(u_\ell^i)}_F^2
    =
    d\norm{u_\ell^i}_2^2.
    \label{eq:Euler-D-Frobenius}
\end{equation}
Thus, by Cauchy's inequality in discrete time,
\begin{align}
    \br{
    \Delta t
    \sum_{\ell=0}^{k-1}
    \norm{D(u_\ell^i)}_F^2
    }^2
    &=
    d^2
    \br{
    \Delta t
    \sum_{\ell=0}^{k-1}
    \norm{u_\ell^i}_2^2
    }^2
    \le
    d^2 T
    \Delta t
    \sum_{\ell=0}^{k-1}
    \norm{u_\ell^i}_2^4.
    \label{eq:Euler-discrete-Cauchy}
\end{align}
Combining~\eqref{eq:Euler-martingale-BDG-start},
\eqref{eq:Euler-D-Frobenius}, and~\eqref{eq:Euler-discrete-Cauchy}, we get
\begin{equation}
    \E\sqbr{
    \max_{0\le r\le k}
    \norm{
    \sqrt{\Delta t}
    \sum_{\ell=0}^{r-1}
    D(u_\ell^i)B_\ell^i
    }_2^4
    }
    \le
    B_{4,d}d^2T
    \Delta t
    \sum_{\ell=0}^{k-1}
    \E\norm{u_\ell^i}_2^4.
    \label{eq:Euler-martingale-bound}
\end{equation}

Substituting~\eqref{eq:Euler-drift-sum-bound} and
\eqref{eq:Euler-martingale-bound} into~\eqref{eq:Euler-fourth-max-start}
yields
\begin{align}
    \E\sqbr{
    \max_{0\le r\le k}
    \norm{\theta_{r,\tau}^i}_2^4
    }
    &\le
    27\E\norm{\theta_{0,\tau}^i}_2^4
    \nonumber\\
    &\quad+
    27\br{\lambda^4T^3+\sigma^4B_{4,d}d^2T}
    \Delta t
    \sum_{\ell=0}^{k-1}
    \E\norm{u_\ell^i}_2^4.
    \label{eq:Euler-single-particle-max-bound}
\end{align}
Averaging over \(i\in[N]\) and using~\eqref{eq:Euler-uk-fourth-bound},
we obtain
\begin{align*}
    \E\sqbr{
    \max_{0\le r\le k}
    Y_r^\theta
    }
    &\le
    27\mu_4+
    K_{\mathrm{bd}}^\theta(T)
    \Delta t
    \sum_{\ell=0}^{k-1}
    \E\sqbr{Y_\ell^\theta},
    %\label{eq:Euler-averaged-max-before-Gronwall}
\end{align*}
where
$
    K_{\mathrm{bd}}^\theta(T)
    :=
    216\br{1+C_m(T)^4}
    \br{\lambda^4T^3+\sigma^4B_{4,d}d^2T}.
   % \label{eq:Euler-Kbdtheta-def}
$
Since \(Y_\ell^\theta\le\max_{0\le r\le \ell}Y_r^\theta\), we have
\begin{equation*}
    \E\sqbr{Y_\ell^\theta}
    \le
    \E\sqbr{
    \max_{0\le r\le \ell}Y_r^\theta
    }.
    %\label{eq:Euler-Yl-max-comparison}
\end{equation*}
Therefore,
\begin{equation*}
    \E\sqbr{
    \max_{0\le r\le k}
    Y_r^\theta
    }
    \le
    27\mu_4
    +
    K_{\mathrm{bd}}^\theta(T)
    \Delta t
    \sum_{\ell=0}^{k-1}
    \E\sqbr{
    \max_{0\le r\le \ell}
    Y_r^\theta
    }.
   % \label{eq:Euler-discrete-Gronwall-input}
\end{equation*}
The discrete Gronwall inequality gives
\begin{equation*}
    \E\sqbr{
    \max_{0\le k\le K}
    Y_k^\theta
    }
    \le
    27\mu_4
    \exp\br{K_{\mathrm{bd}}^\theta(T)T}.
   % \label{eq:Euler-discrete-Gronwall-output}
\end{equation*}
Thus~\eqref{eq:Euler-fourth-moment-bound} holds with
$
    C_{\mathrm{bd}}^\theta(T)
    :=
    27\mu_4
    \exp\br{K_{\mathrm{bd}}^\theta(T)T}.
    %\label{eq:Euler-Cbdtheta-def}
$
This completes the proof.
\end{proof}

\begin{lemma}
\label{lem:localized-EM-approximation}
Under Assumption~\ref{ass:primitive}, fix \(T>0\), \(N\in\mathbb N\), and 
let \(K\Delta t=T\) with \(0<\Delta t\le 1\). Define
\begin{equation}
    \Omega_M^\theta
    :=
    \left\{
    \max_{0\le k\le K}
    \frac1N\sum_{i=1}^N
    \norm{\theta_{k,\tau}^{i}}_2^4
    \le M
    \right\},
    \label{eq:Omega-theta-def}
\end{equation}
and
\begin{equation}
    \Omega_M^{\mathrm{num}}
    :=
    \Omega_M\cap\Omega_M^\theta.
    \label{eq:Omega-num-def}
\end{equation}
If
\begin{equation}
    \mathbb P\br{\Omega_M^{\mathrm{num}}}\ge \frac12,
    \label{eq:Omega-num-half-assumption}
\end{equation}
then there exists a finite constant
$
    C_{\mathrm{NA}}
    =
    C_{\mathrm{NA}}
    \br{
    M,N,d,\alpha,\beta,\tau,\lambda,\sigma,T,\rho_0
    }
    >0,
 %   \label{eq:localized-CNA-dependence}
$
not asserted to be uniform in \(N\), such that
\begin{equation}
    \E\sqbr{
    \frac1N\sum_{i=1}^N
    \norm{\theta_{K,\tau}^{i}-\Theta_{T,\tau}^{i,N}}_2^2
    \,\middle|\,
    \Omega_M^{\mathrm{num}}
    }
    \le
    C_{\mathrm{NA}}\Delta t.
    \label{eq:localized-EM-conditional-error}
\end{equation}
Equivalently, this is the numerical strong error estimate with order
\(m=1/2\):
\begin{equation}
    \E\sqbr{
    \frac1N\sum_{i=1}^N
    \norm{\theta_{K,\tau}^{i}-\Theta_{T,\tau}^{i,N}}_2^2
    \,\middle|\,
    \Omega_M^{\mathrm{num}}
    }
    \le
    C_{\mathrm{NA}}(\Delta t)^{2m},
    \qquad m=\frac12.
    \label{eq:localized-EM-conditional-error-m}
\end{equation}
\end{lemma}

\begin{proof}
We prove the estimate in several steps.

\emph{Step 1: Product notation and local Lipschitz regularity of the consensus map.}
For
\(
    \mathbf x=(x^1,\ldots,x^N)\in(\mathbb R^d)^N,
\)
define
\begin{equation*}
    \norm{\mathbf x}_{N,2}
    :=
    \br{
    \sum_{i=1}^N\norm{x^i}_2^2
    }^{1/2},
    \qquad
    \mu_{\mathbf x}^N
    :=
    \frac1N\sum_{i=1}^N\delta_{x^i}.
    %\label{eq:product-norm-and-empirical-measure}
\end{equation*}
Set
$
    Q_N(\mathbf x)
    :=
    \qsoft[\mu_{\mathbf x}^N],
$ and $
    \mathfrak m_N(\mathbf x)
    :=
    m(\mu_{\mathbf x}^N).
   % \label{eq:QN-mN-def}
$
For \(R>0\), define the product ball
\begin{equation*}
    B_R^N
    :=
    \set{
    \mathbf x\in(\mathbb R^d)^N:
    \norm{\mathbf x}_{N,2}\le R
    }.
   % \label{eq:product-ball-def}
\end{equation*}

We first show that \(Q_N\) is globally Lipschitz. By
Lemma~\ref{lem:quantile-inverse-stability}, for any
\(\mathbf x,\mathbf y\in(\mathbb R^d)^N\),
\begin{align}
    \abs{Q_N(\mathbf x)-Q_N(\mathbf y)}
    &\le
    \kappa^{-1}
    \abs{
    F_{\mu_{\mathbf x}^N}(Q_N(\mathbf y))
    -
    F_{\mu_{\mathbf y}^N}(Q_N(\mathbf y))
    }
    \nonumber\\
    &=
    \kappa^{-1}
    \abs{
    \frac1N
    \sum_{i=1}^N
    \left[
    \psi\br{\frac{Q_N(\mathbf y)-L(x^i)}{\tau}}
    -
    \psi\br{\frac{Q_N(\mathbf y)-L(y^i)}{\tau}}
    \right]
    }
    \nonumber\\
    &\le
    \frac{L_\psi}{\kappa\tau N}
    \sum_{i=1}^N
    \abs{L(x^i)-L(y^i)}
    \nonumber\\
    &\le
    \frac{L_\psi L_L}{\kappa\tau N}
    \sum_{i=1}^N
    \norm{x^i-y^i}_2
    \le
    \frac{L_\psi L_L}{\kappa\tau\sqrt N}
    \norm{\mathbf x-\mathbf y}_{N,2}.
    \label{eq:QN-global-Lipschitz}
\end{align}
Define
$
    C_{Q,N}
    :=
    \frac{L_\psi L_L}{\kappa\tau\sqrt N}.
   % \label{eq:CQN-def}
$
Then~\eqref{eq:QN-global-Lipschitz} becomes
\begin{equation}
    \abs{Q_N(\mathbf x)-Q_N(\mathbf y)}
    \le
    C_{Q,N}\norm{\mathbf x-\mathbf y}_{N,2}.
    \label{eq:QN-global-Lipschitz-short}
\end{equation}

For each \(i\in [N]\), define the empirical weight
$
    a_i(\mathbf x)
    :=
    e^{-\alpha G(x^i)}
    \psi\big(\frac{Q_N(\mathbf x)-L(x^i)}{\tau}\big).
   % \label{eq:empirical-weight-ai-def}
$
Since \(G\ge G_{\min}\), \(G\) is globally Lipschitz, and
\(r\mapsto e^{-\alpha r}\) has derivative bounded by
\(\alpha e^{-\alpha G_{\min}}\) on \([G_{\min},\infty)\), we have
\begin{equation*}
    \abs{
    e^{-\alpha G(x^i)}
    -
    e^{-\alpha G(y^i)}
    }
    \le
    \alpha e^{-\alpha G_{\min}}L_G
    \norm{x^i-y^i}_2.
    %\label{eq:expG-Lipschitz}
\end{equation*}
Using \(0<\psi<1\), the Lipschitz continuity of \(\psi\), and
\eqref{eq:QN-global-Lipschitz-short}, we obtain
\begin{align*}
    \abs{a_i(\mathbf x)-a_i(\mathbf y)}
    &\le
    \alpha e^{-\alpha G_{\min}}L_G\norm{x^i-y^i}_2
    \nonumber\\
    &\quad+
    e^{-\alpha G_{\min}}
    \frac{L_\psi}{\tau}
    \br{
    \abs{Q_N(\mathbf x)-Q_N(\mathbf y)}
    +
    \abs{L(x^i)-L(y^i)}
    }
    \nonumber\\
    &\le
    C_{a,1}\norm{x^i-y^i}_2
    +
    C_{a,2}\norm{\mathbf x-\mathbf y}_{N,2},
    %\label{eq:ai-Lipschitz}
\end{align*}
where
$
    C_{a,1}
    :=
    \alpha e^{-\alpha G_{\min}}L_G
    +
    e^{-\alpha G_{\min}}\frac{L_\psi L_L}{\tau},
 $ and $
    C_{a,2}
    :=
    e^{-\alpha G_{\min}}\frac{L_\psi C_{Q,N}}{\tau}.
    %\label{eq:Ca1-Ca2-def}
$
Consequently,
\begin{align}
    \frac1N\sum_{i=1}^N
    \abs{a_i(\mathbf x)-a_i(\mathbf y)}
    &\le
    \br{\frac{C_{a,1}}{\sqrt N}+C_{a,2}}
    \norm{\mathbf x-\mathbf y}_{N,2}.
    \label{eq:average-ai-Lipschitz}
\end{align}
Define
$
    L_{B,N}
    :=
    \frac{C_{a,1}}{\sqrt N}+C_{a,2},
    %\label{eq:LBN-def}
$
$
    A_N(\mathbf x)
    :=
    A(\mu_{\mathbf x}^N)
    =
    \frac1N\sum_{i=1}^N x^i a_i(\mathbf x),
 $ and $
    B_N(\mathbf x)
    :=
    B(\mu_{\mathbf x}^N)
    =
    \frac1N\sum_{i=1}^N a_i(\mathbf x).
    %\label{eq:AN-BN-def}
$
By~\eqref{eq:average-ai-Lipschitz}, we get
\begin{equation*}
    \abs{B_N(\mathbf x)-B_N(\mathbf y)}
    \le
    L_{B,N}\norm{\mathbf x-\mathbf y}_{N,2}.
    %\label{eq:BN-Lipschitz}
\end{equation*}
Moreover, by Lemma~\ref{lem:deterministic-AB-bounds},
\begin{equation}
    B_N(\mathbf x)
    \ge
    b_G
    :=
    \beta e^{-\alpha G_{\max}}
    >0.
    \label{eq:BN-lower-bound}
\end{equation}

Let \(\mathbf x,\mathbf y\in B_R^N\). Since \(a_i(\mathbf x)\le e^{-\alpha G_{\min}}\),
\begin{align}
    \norm{A_N(\mathbf x)-A_N(\mathbf y)}_2
    &\le
    \frac1N\sum_{i=1}^N
    \norm{x^i-y^i}_2 a_i(\mathbf x)
    +
    \frac1N\sum_{i=1}^N
    \norm{y^i}_2\abs{a_i(\mathbf x)-a_i(\mathbf y)}
    \nonumber\\
    &\le
    \frac{e^{-\alpha G_{\min}}}{\sqrt N}
    \norm{\mathbf x-\mathbf y}_{N,2}
    +
    R L_{B,N}
    \norm{\mathbf x-\mathbf y}_{N,2}.
    \label{eq:AN-Lipschitz}
\end{align}
Define
$
    L_{A,R,N}
    :=
    \frac{e^{-\alpha G_{\min}}}{\sqrt N}
    +
    R L_{B,N}.
    %\label{eq:LARN-def}
$
Then
\begin{equation*}
    \norm{A_N(\mathbf x)-A_N(\mathbf y)}_2
    \le
    L_{A,R,N}\norm{\mathbf x-\mathbf y}_{N,2}.
 %   \label{eq:AN-Lipschitz-short}
\end{equation*}
Also, for every \(\mathbf y\in B_R^N\),
\begin{equation}
    \norm{A_N(\mathbf y)}_2
    \le
    \frac{e^{-\alpha G_{\min}}}{N}
    \sum_{i=1}^N\norm{y^i}_2
    \le
    \frac{e^{-\alpha G_{\min}}R}{\sqrt N}.
    \label{eq:AN-bound-on-ball}
\end{equation}

Since
$
    \mathfrak m_N(\mathbf x)
    =
    \frac{A_N(\mathbf x)}{B_N(\mathbf x)},
    %\label{eq:mN-quotient}
$
we have, for \(\mathbf x,\mathbf y\in B_R^N\),
\begin{align*}
    \norm{
    \mathfrak m_N(\mathbf x)-\mathfrak m_N(\mathbf y)
    }_2
    &\le
    \frac{
    \norm{A_N(\mathbf x)-A_N(\mathbf y)}_2
    }{
    B_N(\mathbf x)
    }
    +
    \norm{A_N(\mathbf y)}_2
    \abs{
    \frac1{B_N(\mathbf x)}
    -
    \frac1{B_N(\mathbf y)}
    }
    \nonumber\\
    &\le
    b_G^{-1}L_{A,R,N}\norm{\mathbf x-\mathbf y}_{N,2}
    +
    \frac{e^{-\alpha G_{\min}}R}{\sqrt N}
    b_G^{-2}
    L_{B,N}
    \norm{\mathbf x-\mathbf y}_{N,2}.
   % \label{eq:mN-Lipschitz-proof}
\end{align*}
Define
\begin{equation*}
    L_{\mathfrak m,R,N}
    :=
    b_G^{-1}L_{A,R,N}
    +
    \frac{e^{-\alpha G_{\min}}R}{\sqrt N}
    b_G^{-2}L_{B,N}.
   % \label{eq:LmRN-def}
\end{equation*}
Then
\begin{equation}
    \norm{
    \mathfrak m_N(\mathbf x)-\mathfrak m_N(\mathbf y)
    }_2
    \le
    L_{\mathfrak m,R,N}
    \norm{\mathbf x-\mathbf y}_{N,2},
    \qquad
    \mathbf x,\mathbf y\in B_R^N.
    \label{eq:mN-local-Lipschitz}
\end{equation}

\emph{Step 2: Local Lipschitz and boundedness of the fixed-\(N\) coefficients.}
Define the fixed-\(N\) drift and diffusion coefficients by
\begin{equation*}
    b_N^i(\mathbf x)
    :=
    -\lambda\br{x^i-\mathfrak m_N(\mathbf x)},
    \qquad
    i\in[N],
    %\label{eq:fixedN-drift-block-def}
\end{equation*}
and
\begin{equation*}
    \Sigma_N^i(\mathbf x)
    :=
    \sigma D\br{x^i-\mathfrak m_N(\mathbf x)},
    \qquad
    i\in[N].
    %\label{eq:fixedN-diffusion-block-def}
\end{equation*}
Let
\begin{equation*}
    b_N(\mathbf x)
    :=
    \br{b_N^1(\mathbf x),\ldots,b_N^N(\mathbf x)}
    \in(\mathbb R^d)^N,
    %\label{eq:fixedN-drift-vector-def}
\end{equation*}
and let \(\Sigma_N(\mathbf x)\) be the block-diagonal matrix with blocks
\(\Sigma_N^i(\mathbf x)\).

For \(\mathbf x,\mathbf y\in B_R^N\), using~\eqref{eq:mN-local-Lipschitz},
\begin{align*}
    \norm{b_N(\mathbf x)-b_N(\mathbf y)}_{N,2}^2
    &=
    \lambda^2
    \sum_{i=1}^N
    \norm{
    x^i-y^i
    -
    \br{
    \mathfrak m_N(\mathbf x)-\mathfrak m_N(\mathbf y)
    }
    }_2^2
    \nonumber\\
    &\le
    2\lambda^2
    \norm{\mathbf x-\mathbf y}_{N,2}^2
    +
    2\lambda^2N
    \norm{
    \mathfrak m_N(\mathbf x)-\mathfrak m_N(\mathbf y)
    }_2^2
    \nonumber\\
    &\le
    2\lambda^2
    \br{
    1+N L_{\mathfrak m,R,N}^2
    }
    \norm{\mathbf x-\mathbf y}_{N,2}^2.
   % \label{eq:fixedN-drift-local-Lipschitz}
\end{align*}
Moreover, by~\eqref{eq:D-basic},
\begin{align}
    \norm{\Sigma_N(\mathbf x)-\Sigma_N(\mathbf y)}_F^2
    &=
    \sum_{i=1}^N
    \norm{
    \Sigma_N^i(\mathbf x)-\Sigma_N^i(\mathbf y)
    }_F^2
    \nonumber\\
    &\le
    \sigma^2 d
    \sum_{i=1}^N
    \norm{
    x^i-y^i
    -
    \br{
    \mathfrak m_N(\mathbf x)-\mathfrak m_N(\mathbf y)
    }
    }_2^2
    \nonumber\\
    &\le
    2\sigma^2d
    \br{
    1+N L_{\mathfrak m,R,N}^2
    }
    \norm{\mathbf x-\mathbf y}_{N,2}^2.
    \label{eq:fixedN-diffusion-local-Lipschitz}
\end{align}
Thus there exists a finite constant \(L_{R,N}>0\) such that
\begin{equation}
    \norm{b_N(\mathbf x)-b_N(\mathbf y)}_{N,2}
    +
    \norm{\Sigma_N(\mathbf x)-\Sigma_N(\mathbf y)}_F
    \le
    L_{R,N}\norm{\mathbf x-\mathbf y}_{N,2},
    \qquad
    \mathbf x,\mathbf y\in B_R^N.
    \label{eq:fixedN-coefficients-local-Lipschitz}
\end{equation}

Next, from~\eqref{eq:AN-bound-on-ball} and~\eqref{eq:BN-lower-bound},
\begin{equation*}
    \norm{\mathfrak m_N(\mathbf x)}_2
    \le
    \frac{e^{-\alpha G_{\min}}R}{b_G\sqrt N},
    \qquad
    \mathbf x\in B_R^N.
    %\label{eq:mN-bound-on-ball}
\end{equation*}
Therefore, for \(\mathbf x\in B_R^N\),
\begin{align*}
    \sum_{i=1}^N
    \norm{x^i-\mathfrak m_N(\mathbf x)}_2^2
    &\le
    2\sum_{i=1}^N\norm{x^i}_2^2
    +
    2N\norm{\mathfrak m_N(\mathbf x)}_2^2
    \le
    2R^2
    +
    2N
    \br{
    \frac{e^{-\alpha G_{\min}}R}{b_G\sqrt N}
    }^2.
    %\label{eq:x-minus-m-bound-on-ball}
\end{align*}
Define
$
    G_{R,N}
    :=
    2\br{\lambda^2+\sigma^2d}
    R^2
    \br{
    1+
    b_G^{-2}e^{-2\alpha G_{\min}}
    }.
    %\label{eq:GRN-def}
$
Then, using \(D(z)=\norm{z}_2\Id\),
\begin{equation}
    \norm{b_N(\mathbf x)}_{N,2}^2
    +
    \norm{\Sigma_N(\mathbf x)}_F^2
    \le
    G_{R,N},
    \qquad
    \mathbf x\in B_R^N.
    \label{eq:fixedN-coefficients-bound-on-ball}
\end{equation}

\emph{Step 3: The strengthened stopping time.}
Define
\begin{equation*}
    \mathbf\Theta_{t,\tau}^N
    :=
    \br{
    \Theta_{t,\tau}^{1,N},\ldots,\Theta_{t,\tau}^{N,N}
    },
    \qquad
    \mathbf\theta_{k,\tau}^N
    :=
    \br{
    \theta_{k,\tau}^{1},\ldots,\theta_{k,\tau}^{N}
    }.
    %\label{eq:vector-processes-EM-proof}
\end{equation*}
Define
$
    \eta(t):=\lfloor t/\Delta t\rfloor\Delta t.
    %\label{eq:eta-time-projection}
$
and the stopping time
\begin{equation*}
    \tau_M^{\mathrm{num}}
    :=
    \inf\left\{
    t\in[0,T]:
    \frac1N\sum_{i=1}^N
    \max\left\{
    \norm{\Theta_{t,\tau}^{i,N}}_2^4,
    \norm{\bar\Theta_{t,\tau}^{i}}_2^4
    \right\}
    >M
    \ \text{or}\
    \frac1N\sum_{i=1}^N
    \norm{\theta_{\lfloor t/\Delta t\rfloor,\tau}^{i}}_2^4
    >M
    \right\},
   % \label{eq:tauM-num-def}
\end{equation*}
with the convention \(\tau_M^{\mathrm{num}}=+\infty\) if the set is empty.
Then
$
    \Omega_M^{\mathrm{num}}
    =
    \set{
    \tau_M^{\mathrm{num}}>T
    }.
    %\label{eq:Omega-num-stopping-equivalence}
$
Let
$
    I_M^{\mathrm{num}}(t)
    :=
    \one_{\{t<\tau_M^{\mathrm{num}}\}}.
    %\label{eq:IM-num-def}
$
On the event \(\{I_M^{\mathrm{num}}(t)=1\}\), both
\(\mathbf\Theta_{t,\tau}^N\) and
\(\mathbf\theta_{\lfloor t/\Delta t\rfloor,\tau}^N\) are bounded in the
product norm. Indeed, if
\[
    \frac1N\sum_{i=1}^N\norm{x^i}_2^4\le M,
\]
then
\begin{align*}
    \norm{\mathbf x}_{N,2}^2
    =
    \sum_{i=1}^N\norm{x^i}_2^2
    &\le
    \br{\sum_{i=1}^N1^2}^{1/2}
    \br{\sum_{i=1}^N\norm{x^i}_2^4}^{1/2}
    \le
    N^{1/2}(NM)^{1/2}
    =
    N M^{1/2}.
    %\label{eq:fourth-to-second-product-bound}
\end{align*}
Thus we may take
$
    R_{M,N}:=N^{1/2}M^{1/4},
    %\label{eq:RMN-def}
$
so that, on \(\{I_M^{\mathrm{num}}(t)=1\}\),
\begin{equation*}
    \mathbf\Theta_{t,\tau}^N\in B_{R_{M,N}}^N,
    \qquad
    \mathbf\theta_{\lfloor t/\Delta t\rfloor,\tau}^N
    \in B_{R_{M,N}}^N.
    %\label{eq:processes-in-product-ball}
\end{equation*}

\emph{Step 4: Continuous Euler interpolation.}
Define the continuous-time Euler interpolation
\begin{equation}
    \mathbf\theta_{t,\tau}^{N,\Delta}
    :=
    \mathbf\theta_{0,\tau}^N
    +
    \int_0^t
    b_N\br{\mathbf\theta_{\eta(s)/\Delta t,\tau}^N}
    \dd s
    +
    \int_0^t
    \Sigma_N\br{\mathbf\theta_{\eta(s)/\Delta t,\tau}^N}
    \dd\mathbf W_s,
    \label{eq:continuous-Euler-interpolation}
\end{equation}
where
$
    \mathbf W_t
    :=
    \br{W_t^1,\ldots,W_t^N}.
    %\label{eq:vector-Brownian-motion}
$
Since \(\eta(s)/\Delta t=\lfloor s/\Delta t\rfloor\), the notation in
\eqref{eq:continuous-Euler-interpolation} is consistent with the grid value
\(\mathbf\theta_{\lfloor s/\Delta t\rfloor,\tau}^N\). In particular,
$
    \mathbf\theta_{K\Delta t,\tau}^{N,\Delta}
    =
    \mathbf\theta_{K,\tau}^N.
    %\label{eq:Euler-interpolation-grid-identity}
$

Define the error process
$
    e_t
    :=
    \mathbf\Theta_{t,\tau}^N
    -
    \mathbf\theta_{t,\tau}^{N,\Delta}.
    %\label{eq:EM-error-process-def}
$
Because of~\eqref{eq:Euler-same-initial-data}, we have
$
    e_0=0.
    %\label{eq:EM-error-initial-zero}
$

\emph{Step 5: Stopped error estimate.}
For \(t\in[0,T]\), the stopped error satisfies
\begin{align*}
    e_{t\wedge\tau_M^{\mathrm{num}}}
    &=
    \int_0^t
    I_M^{\mathrm{num}}(s)
    \left[
    b_N\br{\mathbf\Theta_{s,\tau}^N}
    -
    b_N\br{\mathbf\theta_{\eta(s)/\Delta t,\tau}^N}
    \right]\dd s
    \nonumber\\
    &\quad+
    \int_0^t
    I_M^{\mathrm{num}}(s)
    \left[
    \Sigma_N\br{\mathbf\Theta_{s,\tau}^N}
    -
    \Sigma_N\br{\mathbf\theta_{\eta(s)/\Delta t,\tau}^N}
    \right]\dd\mathbf W_s.
    %\label{eq:stopped-EM-error-equation}
\end{align*}
Using Cauchy's inequality for the drift integral and Itô's isometry for the
martingale integral, there exists a numerical constant \(C_T>0\), depending
only on \(T\), such that
\begin{align*}
    \E\norm{e_{t\wedge\tau_M^{\mathrm{num}}}}_{N,2}^2
    &\le
    C_T
    \int_0^t
    \E\sqbr{
    I_M^{\mathrm{num}}(s)
    \norm{
    b_N\br{\mathbf\Theta_{s,\tau}^N}
    -
    b_N\br{\mathbf\theta_{\eta(s)/\Delta t,\tau}^N}
    }_{N,2}^2
    }
    \dd s
    \nonumber\\
    &\quad+
    C_T
    \int_0^t
    \E\sqbr{
    I_M^{\mathrm{num}}(s)
    \norm{
    \Sigma_N\br{\mathbf\Theta_{s,\tau}^N}
    -
    \Sigma_N\br{\mathbf\theta_{\eta(s)/\Delta t,\tau}^N}
    }_F^2
    }
    \dd s.
    %\label{eq:stopped-EM-error-before-Lipschitz}
\end{align*}
On \(\{I_M^{\mathrm{num}}(s)=1\}\), both arguments of \(b_N\) and
\(\Sigma_N\) lie in \(B_{R_{M,N}}^N\). Hence, by
\eqref{eq:fixedN-coefficients-local-Lipschitz},
\begin{equation}
    \E\norm{e_{t\wedge\tau_M^{\mathrm{num}}}}_{N,2}^2
    \le
    C_{M,N}
    \int_0^t
    \E\sqbr{
    I_M^{\mathrm{num}}(s)
    \norm{
    \mathbf\Theta_{s,\tau}^N
    -
    \mathbf\theta_{\eta(s)/\Delta t,\tau}^N
    }_{N,2}^2
    }
    \dd s,
    \label{eq:stopped-EM-error-after-Lipschitz}
\end{equation}
where \(C_{M,N}>0\) is finite and depends on
\(
    M,N,d,\alpha,\beta,\tau,\lambda,\sigma,T.
\) 
We now split
\begin{align}
    \mathbf\Theta_{s,\tau}^N
    -
    \mathbf\theta_{\eta(s)/\Delta t,\tau}^N
    &=
    \br{
    \mathbf\Theta_{s,\tau}^N
    -
    \mathbf\Theta_{\eta(s),\tau}^N
    }
    +
    \br{
    \mathbf\Theta_{\eta(s),\tau}^N
    -
    \mathbf\theta_{\eta(s)/\Delta t,\tau}^N
    }.
    \label{eq:EM-error-splitting}
\end{align}
Therefore,
\begin{align}
    I_M^{\mathrm{num}}(s)
    \norm{
    \mathbf\Theta_{s,\tau}^N
    -
    \mathbf\theta_{\eta(s)/\Delta t,\tau}^N
    }_{N,2}^2
    &\le
    2I_M^{\mathrm{num}}(s)
    \norm{
    \mathbf\Theta_{s,\tau}^N
    -
    \mathbf\Theta_{\eta(s),\tau}^N
    }_{N,2}^2
    \nonumber\\
    &\quad+
    2I_M^{\mathrm{num}}(s)
    \norm{
    \mathbf\Theta_{\eta(s),\tau}^N
    -
    \mathbf\theta_{\eta(s)/\Delta t,\tau}^N
    }_{N,2}^2.
    \label{eq:EM-error-splitting-square}
\end{align}

For the first term, since
\[
    I_M^{\mathrm{num}}(s)=1
    \quad\Longrightarrow\quad
    \mathbf\Theta_{r,\tau}^N\in B_{R_{M,N}}^N
    \quad\text{for all }r\in[\eta(s),s],
\]
the coefficient bound~\eqref{eq:fixedN-coefficients-bound-on-ball} implies
\begin{equation}
    \E\sqbr{
    I_M^{\mathrm{num}}(s)
    \norm{
    \mathbf\Theta_{s,\tau}^N
    -
    \mathbf\Theta_{\eta(s),\tau}^N
    }_{N,2}^2
    }
    \le
    C_{M,N}\Delta t.
    \label{eq:continuous-particle-increment-bound}
\end{equation}
Indeed, this follows by writing the increment of
\(\mathbf\Theta_{t,\tau}^N\) over \([\eta(s),s]\), applying
Cauchy's inequality to the drift part, Itô's isometry to the stochastic part,
and using \(s-\eta(s)\le\Delta t\).

For the second term, since \(s\ge\eta(s)\), we have
\(
    I_M^{\mathrm{num}}(s)
    \le
    I_M^{\mathrm{num}}(\eta(s)).
\)
Moreover,
\(
    \mathbf\Theta_{\eta(s),\tau}^N
    -
    \mathbf\theta_{\eta(s)/\Delta t,\tau}^N
    =
    e_{\eta(s)}.
\)
Therefore
\begin{equation}
    \E\sqbr{
    I_M^{\mathrm{num}}(s)
    \norm{
    \mathbf\Theta_{\eta(s),\tau}^N
    -
    \mathbf\theta_{\eta(s)/\Delta t,\tau}^N
    }_{N,2}^2
    }
    \le
    \E\norm{
    e_{\eta(s)\wedge\tau_M^{\mathrm{num}}}
    }_{N,2}^2.
    \label{eq:grid-error-stopped-bound}
\end{equation}

Combining~\eqref{eq:stopped-EM-error-after-Lipschitz},
\eqref{eq:EM-error-splitting-square},
\eqref{eq:continuous-particle-increment-bound}, and
\eqref{eq:grid-error-stopped-bound}, we obtain
\begin{equation}
    \E\norm{e_{t\wedge\tau_M^{\mathrm{num}}}}_{N,2}^2
    \le
    C_{M,N}
    \int_0^t
    \E\norm{
    e_{\eta(s)\wedge\tau_M^{\mathrm{num}}}
    }_{N,2}^2
    \dd s
    +
    C_{M,N}\Delta t.
    \label{eq:EM-stopped-Gronwall-input}
\end{equation}
Define
$
    Z(t)
    :=
    \sup_{0\le r\le t}
    \E\norm{e_{r\wedge\tau_M^{\mathrm{num}}}}_{N,2}^2.
   % \label{eq:Zt-def}
$
Since \(\eta(s)\le s\), equation~\eqref{eq:EM-stopped-Gronwall-input} gives
$$
Z(t)
    \le
    C_{M,N}
    \int_0^t Z(s)\dd s
    +
    C_{M,N}\Delta t.
    %\label{eq:Zt-Gronwall-input}
$$
By Gronwall's inequality,
$
    Z(T)
    \le
    C_{M,N}\Delta t.
    %\label{eq:Zt-Gronwall-output}
$
In particular,
$
    \E\norm{
    e_{T\wedge\tau_M^{\mathrm{num}}}
    }_{N,2}^2
    \le
    C_{M,N}\Delta t.
    %\label{eq:stopped-terminal-error-bound}
$

On \(\Omega_M^{\mathrm{num}}=\{\tau_M^{\mathrm{num}}>T\}\), we have
\(T\wedge\tau_M^{\mathrm{num}}=T\). Therefore,
\begin{equation*}
    \E\sqbr{
    \norm{
    \mathbf\theta_{K,\tau}^N
    -
    \mathbf\Theta_{T,\tau}^N
    }_{N,2}^2
    \one_{\Omega_M^{\mathrm{num}}}
    }
    \le
    C_{M,N}\Delta t.
    %\label{eq:localized-vector-EM-error-indicator}
\end{equation*}
Dividing by \(N\), and absorbing the factor \(N^{-1}\) into the constant, gives
\begin{equation*}
    \E\sqbr{
    \frac1N
    \sum_{i=1}^N
    \norm{
    \theta_{K,\tau}^i
    -
    \Theta_{T,\tau}^{i,N}
    }_2^2
    \one_{\Omega_M^{\mathrm{num}}}
    }
    \le
    C_{M,N}\Delta t.
    %\label{eq:localized-averaged-EM-error-indicator}
\end{equation*}
Using~\eqref{eq:Omega-num-half-assumption},
\begin{align*}
    \E\sqbr{
    \frac1N
    \sum_{i=1}^N
    \norm{
    \theta_{K,\tau}^i
    -
    \Theta_{T,\tau}^{i,N}
    }_2^2
    \,\middle|\,
    \Omega_M^{\mathrm{num}}
    }
    =
    \frac{
    \E\sqbr{
    \frac1N
    \sum_{i=1}^N
    \norm{
    \theta_{K,\tau}^i
    -
    \Theta_{T,\tau}^{i,N}
    }_2^2
    \one_{\Omega_M^{\mathrm{num}}}
    }
    }{
    \mathbb P(\Omega_M^{\mathrm{num}})
    }\le
    2C_{M,N}\Delta t.
   % \label{eq:localized-conditional-EM-error-final}
\end{align*}
Renaming \(2C_{M,N}\) as \(C_{\mathrm{NA}}\) proves
\eqref{eq:localized-EM-conditional-error}. Since Euler--Maruyama has strong
order \(m=1/2\), this is exactly~\eqref{eq:localized-EM-conditional-error-m}.
\end{proof}

\textbf{Theorem \ref{thm:numerical-convergence}.}
% \label{thm:numerical-convergence}
\textit{Under Assumption~\ref{ass:primitive}, let \(T=T_\varepsilon\) be the hitting time supplied by Theorem~\ref{thm:direct-mf-convergence}, and let
$
    K\Delta t=T,
 $ for $
    0<\Delta t\le1.
   % \label{eq:numerical-time-grid}
$
Define \(\Omega_M^\theta\) and \(\Omega_M^{\mathrm{num}}\) as in
\eqref{eq:Omega-theta-def} and~\eqref{eq:Omega-num-def}. Let
\begin{equation*}
    C_{\mathrm{bd}}^{\mathrm{num}}(T)
    :=
    C_{\mathrm{bd}}(T)+C_{\mathrm{bd}}^\theta(T),
    %\label{eq:Cbd-num-def}
\end{equation*}
and let \(\delta_M\in(0,1/2]\). Choose
$
    M
    :=
    \frac{
    C_{\mathrm{bd}}^{\mathrm{num}}(T)
    }{
    \delta_M
    }.
    %\label{eq:numerical-M-choice}
$
If the hypotheses of Theorem~\ref{thm:direct-mf-convergence} hold and
$
    V_\tau(T)=\varepsilon,
    %\label{eq:numerical-mf-hitting-condition}
$
then there exists a finite constant
\begin{equation*}
    C_{\mathrm{NA}}
    =
    C_{\mathrm{NA}}
    \br{
    M,N,d,\alpha,\beta,\tau,\lambda,\sigma,T,\rho_0
    }
    >0,
    %\label{eq:numerical-CNA-dependence}
\end{equation*}
not asserted to be uniform in \(N\), such that, for every
\(\varepsilon_{\mathrm{tot}}>0\),
\begin{equation}
    \mathbb P\br{
    \norm{\widehat\theta_{K,\tau}^{N}-\theta^\star}_2^2
    \le
    \varepsilon_{\mathrm{tot}}
    }
    \ge
    1-\delta_M
    -
    \frac{
    3C_{\mathrm{NA}}\Delta t
    +
    6C_{\mathrm{MFA}}N^{-1}
    +
    12\varepsilon
    }{
    \varepsilon_{\mathrm{tot}}
    }.
    \label{eq:numerical-convergence-probability}
\end{equation}}

\begin{proof}
We first prove that the strengthened cutoff event has high probability. By
Lemma~\ref{lem:fourth-moment-bound},
\begin{equation*}
    \mathbb P(\Omega_M^c)
    \le
    \frac{C_{\mathrm{bd}}(T)}{M}.
    %\label{eq:numerical-OmegaM-probability}
\end{equation*}
By Lemma~\ref{lem:Euler-fourth-moment-bound} and Markov's inequality,
\begin{equation*}
    \mathbb P\br{(\Omega_M^\theta)^c}
    \le
    \frac{C_{\mathrm{bd}}^\theta(T)}{M}.
    %\label{eq:numerical-Omega-theta-probability}
\end{equation*}
Therefore, by the union bound,
\begin{align*}
    \mathbb P\br{(\Omega_M^{\mathrm{num}})^c}
    &=
    \mathbb P\br{
    \Omega_M^c\cup(\Omega_M^\theta)^c
    }
    \nonumber\\
    &\le
    \frac{C_{\mathrm{bd}}(T)}{M}
    +
    \frac{C_{\mathrm{bd}}^\theta(T)}{M}
    =
    \frac{C_{\mathrm{bd}}^{\mathrm{num}}(T)}{M}
    =
    \delta_M.
   % \label{eq:numerical-Omega-num-probability}
\end{align*}
Hence
\begin{equation}
    \mathbb P\br{\Omega_M^{\mathrm{num}}}
    \ge
    1-\delta_M
    \ge
    \frac12.
    \label{eq:numerical-Omega-num-lower-bound}
\end{equation}

By Lemma~\ref{lem:localized-EM-approximation}, we have
\begin{equation}
    \E\sqbr{
    \frac1N\sum_{i=1}^N
    \norm{\theta_{K,\tau}^{i}-\Theta_{T,\tau}^{i,N}}_2^2
    \,\middle|\,
    \Omega_M^{\mathrm{num}}
    }
    \le
    C_{\mathrm{NA}}\Delta t.
    \label{eq:numerical-EM-input-from-lemma}
\end{equation}

Recall the empirical averages
\begin{equation*}
    \widehat\theta_{K,\tau}^{N}
    :=
    \frac1N\sum_{i=1}^N \theta_{K,\tau}^i,
    \qquad
    \widehat\Theta_{T,\tau}^{N,\mathrm{int}}
    :=
    \frac1N\sum_{i=1}^N \Theta_{T,\tau}^{i,N},
    \qquad
    \widehat{\overline\Theta}_{T,\tau}^{N}
    :=
    \frac1N\sum_{i=1}^N \overline\Theta_{T,\tau}^{i}.
   % \label{eq:numerical-empirical-averages-proof}
\end{equation*}
By Young's inequality,
\begin{align}
    \norm{\widehat\theta_{K,\tau}^{N}-\theta^\star}_2^2
    &\le
    3\norm{
        \widehat\theta_{K,\tau}^{N}
        -
        \widehat\Theta_{T,\tau}^{N,\mathrm{int}}
    }_2^2+
    3\norm{
        \widehat\Theta_{T,\tau}^{N,\mathrm{int}}
        -
        \widehat{\overline\Theta}_{T,\tau}^{N}
    }_2^2+
    3\norm{
        \widehat{\overline\Theta}_{T,\tau}^{N}
        -
        \theta^\star
    }_2^2.
    \label{eq:numerical-three-term-decomposition-proof}
\end{align}

We estimate the three terms under the conditional probability
\(\mathbb P(\,\cdot\,|\,\Omega_M^{\mathrm{num}})\).

First, by Jensen's inequality and~\eqref{eq:numerical-EM-input-from-lemma},
\begin{align}
    \E\sqbr{
    \norm{
        \widehat\theta_{K,\tau}^{N}
        -
        \widehat\Theta_{T,\tau}^{N,\mathrm{int}}
    }_2^2
    \,\middle|\,
    \Omega_M^{\mathrm{num}}
    }
    &\le
    \E\sqbr{
    \frac1N\sum_{i=1}^N
    \norm{
        \theta_{K,\tau}^i
        -
        \Theta_{T,\tau}^{i,N}
    }_2^2
    \,\middle|\,
    \Omega_M^{\mathrm{num}}
    }\le
    C_{\mathrm{NA}}\Delta t.
    \label{eq:numerical-first-term-bound}
\end{align}

Second, define
\begin{equation*}
    Y_{\mathrm{MFA}}
    :=
    \frac1N\sum_{i=1}^N
    \norm{
        \Theta_{T,\tau}^{i,N}
        -
        \overline\Theta_{T,\tau}^{i}
    }_2^2.
    %\label{eq:YMFA-def}
\end{equation*}
By Jensen's inequality,
\begin{equation}
    \norm{
        \widehat\Theta_{T,\tau}^{N,\mathrm{int}}
        -
        \widehat{\overline\Theta}_{T,\tau}^{N}
    }_2^2
    \le
    Y_{\mathrm{MFA}}.
    \label{eq:MFA-average-Jensen}
\end{equation}
Let
$
    A:=\Omega_M^{\mathrm{num}},
 $ and $
    B:=\Omega_M.
    %\label{eq:A-B-events-def}
$
Then \(A\subseteq B\). Proposition~\ref{prop:mfa-cutoff} gives
\begin{equation*}
    \E\sqbr{
    Y_{\mathrm{MFA}}
    \,\middle|\,
    B
    }
    \le
    \frac{C_{\mathrm{MFA}}}{N}.
    %\label{eq:YMFA-conditional-on-B}
\end{equation*}
Equivalently,
$
    \E\sqbr{
    Y_{\mathrm{MFA}}\one_B
    }
    =
    \mathbb P(B)
    \E\sqbr{
    Y_{\mathrm{MFA}}
    \,\middle|\,
    B
    }
    \le
    \frac{C_{\mathrm{MFA}}}{N}.
    %\label{eq:YMFA-indicator-B}
$
Since \(A\subseteq B\), we have \(\one_A\le\one_B\). Therefore
\begin{equation}
    \E\sqbr{
    Y_{\mathrm{MFA}}\one_A
    }
    \le
    \E\sqbr{
    Y_{\mathrm{MFA}}\one_B
    }
    \le
    \frac{C_{\mathrm{MFA}}}{N}.
    \label{eq:YMFA-indicator-A}
\end{equation}
Using~\eqref{eq:numerical-Omega-num-lower-bound}, we get
\begin{align}
    \E\sqbr{
    Y_{\mathrm{MFA}}
    \,\middle|\,
    \Omega_M^{\mathrm{num}}
    }
    &=
    \frac{
    \E\sqbr{
    Y_{\mathrm{MFA}}\one_{\Omega_M^{\mathrm{num}}}
    }
    }{
    \mathbb P(\Omega_M^{\mathrm{num}})
    }\le
    \frac{2C_{\mathrm{MFA}}}{N}.
    \label{eq:YMFA-conditional-on-A}
\end{align}
Combining~\eqref{eq:MFA-average-Jensen} and~\eqref{eq:YMFA-conditional-on-A},
we obtain
\begin{equation}
    \E\sqbr{
    \norm{
        \widehat\Theta_{T,\tau}^{N,\mathrm{int}}
        -
        \widehat{\overline\Theta}_{T,\tau}^{N}
    }_2^2
    \,\middle|\,
    \Omega_M^{\mathrm{num}}
    }
    \le
    \frac{2C_{\mathrm{MFA}}}{N}.
    \label{eq:numerical-second-term-bound}
\end{equation}

Third, by Jensen's inequality and~\eqref{eq:numerical-Omega-num-lower-bound},
\begin{align}
    \E\sqbr{
    \norm{
        \widehat{\overline\Theta}_{T,\tau}^{N}
        -
        \theta^\star
    }_2^2
    \,\middle|\,
    \Omega_M^{\mathrm{num}}
    }
    &=
    \frac{
    \E\sqbr{
    \norm{
        \widehat{\overline\Theta}_{T,\tau}^{N}
        -
        \theta^\star
    }_2^2
    \one_{\Omega_M^{\mathrm{num}}}
    }
    }{
    \mathbb P(\Omega_M^{\mathrm{num}})
    }
    \nonumber\\
    &\le
    2
    \E\sqbr{
    \norm{
        \widehat{\overline\Theta}_{T,\tau}^{N}
        -
        \theta^\star
    }_2^2
    }
    \nonumber\\
    &\le
    \frac{2}{N}
    \sum_{i=1}^N
    \E\sqbr{
    \norm{
        \overline\Theta_{T,\tau}^{i}
        -
        \theta^\star
    }_2^2
    }
    \nonumber\\
    &=
    2
    \int_{\mathbb R^d}
    \norm{x-\theta^\star}_2^2
    \rho_{T,\tau}(\dd x)
    \nonumber\\
    &=
    4V_\tau(T)
    =
    4\varepsilon.
    \label{eq:numerical-third-term-bound}
\end{align}

Combining
\eqref{eq:numerical-three-term-decomposition-proof},
\eqref{eq:numerical-first-term-bound},
\eqref{eq:numerical-second-term-bound}, and
\eqref{eq:numerical-third-term-bound}, we obtain
\begin{equation*}
    \E\sqbr{
    \norm{\widehat\theta_{K,\tau}^{N}-\theta^\star}_2^2
    \,\middle|\,
    \Omega_M^{\mathrm{num}}
    }
    \le
    3C_{\mathrm{NA}}\Delta t
    +
    6C_{\mathrm{MFA}}N^{-1}
    +
    12\varepsilon.
    %\label{eq:numerical-conditional-mean-square-bound}
\end{equation*}

By Markov's inequality under the conditional probability
\(\mathbb P(\,\cdot\,|\,\Omega_M^{\mathrm{num}})\),
\begin{equation*}
    \mathbb P\br{
    \norm{\widehat\theta_{K,\tau}^{N}-\theta^\star}_2^2
    >
    \varepsilon_{\mathrm{tot}}
    \,\middle|\,
    \Omega_M^{\mathrm{num}}
    }
    \le
    \frac{
    3C_{\mathrm{NA}}\Delta t
    +
    6C_{\mathrm{MFA}}N^{-1}
    +
    12\varepsilon
    }{
    \varepsilon_{\mathrm{tot}}
    }.
    %\label{eq:numerical-conditional-markov}
\end{equation*}
Therefore,
\begin{align*}
    \mathbb P\br{
    \norm{\widehat\theta_{K,\tau}^{N}-\theta^\star}_2^2
    >
    \varepsilon_{\mathrm{tot}}
    }
    &\le
    \mathbb P\br{(\Omega_M^{\mathrm{num}})^c}
    +
    \mathbb P\br{
    \norm{\widehat\theta_{K,\tau}^{N}-\theta^\star}_2^2
    >
    \varepsilon_{\mathrm{tot}},
    \Omega_M^{\mathrm{num}}
    }
    \nonumber\\
    &\le
    \delta_M
    +
    \mathbb P(\Omega_M^{\mathrm{num}})
    \mathbb P\br{
    \norm{\widehat\theta_{K,\tau}^{N}-\theta^\star}_2^2
    >
    \varepsilon_{\mathrm{tot}}
    \,\middle|\,
    \Omega_M^{\mathrm{num}}
    }
    \nonumber\\
    &\le
    \delta_M
    +
    \frac{
    3C_{\mathrm{NA}}\Delta t
    +
    6C_{\mathrm{MFA}}N^{-1}
    +
    12\varepsilon
    }{
    \varepsilon_{\mathrm{tot}}
    }.
    %\label{eq:numerical-unconditional-probability}
\end{align*}
Taking complements proves~\eqref{eq:numerical-convergence-probability}.
\end{proof}

\section{Extension to parameterized lower-level solution sets}\label{app:extension}
The present analysis focuses on the case where the lower-level solution set
\(\Theta\) is fixed. A natural extension is the parameterized setting
\[
    (\phi^\star,\theta^\star)
    \in
    \arg\min_{\phi\in\mathbb R^p,\ \theta\in\Theta_\phi}
    G(\phi,\theta),
    \qquad
    \Theta_\phi
    :=
    \arg\min_{\theta\in\mathbb R^d} L(\phi,\theta).
\]
In this case, one may view \(\theta\) as a lower-level response variable for a
fixed value of \(\phi\). For each \(\phi\), define the best-response upper-level
value
\[
    G_{\mathrm{br}}(\phi)
    :=
    \min_{\theta\in\Theta_\phi} G(\phi,\theta),
\]
provided the minimum is attained. Then the original problem can be formally
rewritten as the single-level problem
\[
    \phi^\star \in \arg\min_{\phi\in\mathbb R^p} G_{\mathrm{br}}(\phi),
    \qquad
    \theta^\star \in
    \arg\min_{\theta\in\Theta_{\phi^\star}}
    G(\phi^\star,\theta).
\]
This suggests a nested consensus-based strategy. For each fixed \(\phi\), one
can apply the proposed quantile-selection and Laplace-weighting mechanism in
the \(\theta\)-variable to approximate a best response
\[
    \theta_{\mathrm{br}}(\phi)
    \in
    \arg\min_{\theta\in\Theta_\phi}G(\phi,\theta),
\]
and hence obtain an approximation of \(G_{\mathrm{br}}(\phi)\). A standard CBO
scheme can then be applied in the outer variable \(\phi\) to minimize the
resulting best-response objective \(G_{\mathrm{br}}\). In this way, the method
can be interpreted as an inner CB$^2$O procedure over \(\theta\), coupled with an
outer CBO procedure over \(\phi\).

A rigorous convergence theory for this nested extension would require additional
uniform assumptions in \(\phi\), such as uniform lower-level error bounds for
\(L(\phi,\cdot)\), uniform local growth conditions for
\(G(\phi,\cdot)\) on \(\Theta_\phi\), stability of the best-response value
\(G_{\mathrm{br}}\), and quantitative control of the error produced by the inner
solver. These issues are beyond the scope of the present work, but the
best-response reformulation provides a natural route for extending the proposed
framework to parameterized bi-level problems.

\section{Limitations}\label{app:limits}
This work focuses on establishing a first finite-particle approximation theory for soft consensus-based bilevel optimization. Several directions remain open.
Future work could strengthen the finite-particle analysis by deriving constants uniform in \(N\), relaxing the cutoff-event localization, and proving mean-field well-posedness directly from primitive assumptions. It would also be valuable to weaken the inverse-continuity condition and extend the theory to broader classes of overparameterized models. On the empirical side, larger-scale deep-learning and reinforcement-learning benchmarks could further clarify the practical advantages of soft selection.

\section{Extended Experimental Details}\label{app:experiment}

%\subsection{Comparison with other Benchmarks}\label{app:compare}

\subsection{Comparison against gradient based methods in nonconvex objective function}\label{app:comparison}

Here, we compare our derivative-free method (S)CB$^2$O with the state-of-the-art gradient-based method, called VPBGD of \cite{shen2023penalty} and a gradient-based heuristic called SBGD which is given in Algorithm \ref{alg:sbgd_single}.
To this end, we used the following setting in~\Cref{tab:gd3-hyperparameters}. 
Note that lower-level function is not convex and consequently $\Theta$ is non-convex.

\begin{algorithm}[t]
\caption{Sequential Bi-level Gradient Descent (SBGD)}
\label{alg:sbgd_single}
\begin{algorithmic}
\REQUIRE initial point $x_0 \in \mathbb{R}^d$,
         base step size $\alpha > 0$,
         lower-to-upper step ratio $\xi > 0$,
         number of iterations $K$
\STATE Set $\alpha_L \leftarrow \gamma \alpha$,\quad $\alpha_G \leftarrow \alpha$
\FOR{$k = 0, 1, \ldots, K-1$}
    \STATE $x_{k+\frac{1}{2}} \leftarrow \Pi_C\!\left(x_k - \alpha_L\,\nabla L(x_k)\right)$
           \hfill \COMMENT{lower-level gradient step}
    \STATE $x_{k+1} \leftarrow \Pi_C\!\left(x_{k+\frac{1}{2}} - \alpha_G\,\nabla G(x_{k+\frac{1}{2}})\right)$
           \hfill \COMMENT{upper-level gradient step at intermediate point}
\ENDFOR
\ENSURE $x_K$
\end{algorithmic}
\end{algorithm}

\textbf{lower-level objective.}
The lower-level objective is
\begin{equation}
    \bar{g}(x,y)
    =
    h(x) + h(y) + 0.1\sin^2(2x+3y),
    \qquad
    h(t)=\sin^2(t)+\sin^2(3t)+0.25\sin^2(5t).
    \label{eq:gd3-lower-nonpl}
\end{equation}
All terms are nonnegative, the global lower-level minimizers are
\(\pi\mathbb{Z}\times\pi\mathbb{Z}\), with minimum value \(0\).  

\textbf{Quadratic upper-level objective.}
The upper-level objective is the centered quadratic
\begin{equation}
    q(x,y) = \frac{1}{2}(x^2+y^2),
    \qquad
    \nabla q(x,y) =
    \begin{bmatrix}
        x\\y
    \end{bmatrix}.
    \label{eq:gd3-upper-quadratic}
\end{equation}
The bi-level optimum is attained at \((0,0)\), which lies in the global
lower-level minimizer grid, and the optimal upper-level value is \(q^\star=0\). Next table summarizes the setting of this experiment. 

\begin{small}
\begin{table}[h]
\centering
\caption{Experiment setting for the comparison between gradient descent-based methods and (S)CB\(^2\)O}
\label{tab:gd3-hyperparameters}
\begin{tabular}{lc}
\toprule
Hyperparameter & Value \\
\midrule
Seeds & \(\{0,1,2\}\) \\
Iterations & \(300\) \\
Number of particles for (S)CB\(^2\)O & \(1000\) \\
\(\alpha\) & \(40\) \\
Learning rate & \(10^{-2}\) \\
\(\beta\) & \(0.01\) \\
\(\lambda\) & \(1\) \\
Penalty for gradient-method \(\xi\) & \(50\) \\
SCB\(^2\)O \(\xi\) & \(10\) \\
Initialization mean & \(0\) \\
Initialization scale & \(50\) \\
\bottomrule
\end{tabular}
\end{table}
\end{small}

\textit{Results.} Figure \ref{fig:gd3-lower} shows the resulting curves. Both CB$^2$O and SCB$^2$O runs reach the best upper-level values and remain near the lower-level minimizer grid.  The gradient baselines more aggressively reduce the lower-level objective, but their final upper-level values remain far from the quadratic optimum which indicates that it gets trapped in some local optimum.

\begin{figure}[h]
  \centering
  \includegraphics[width=0.48\linewidth]{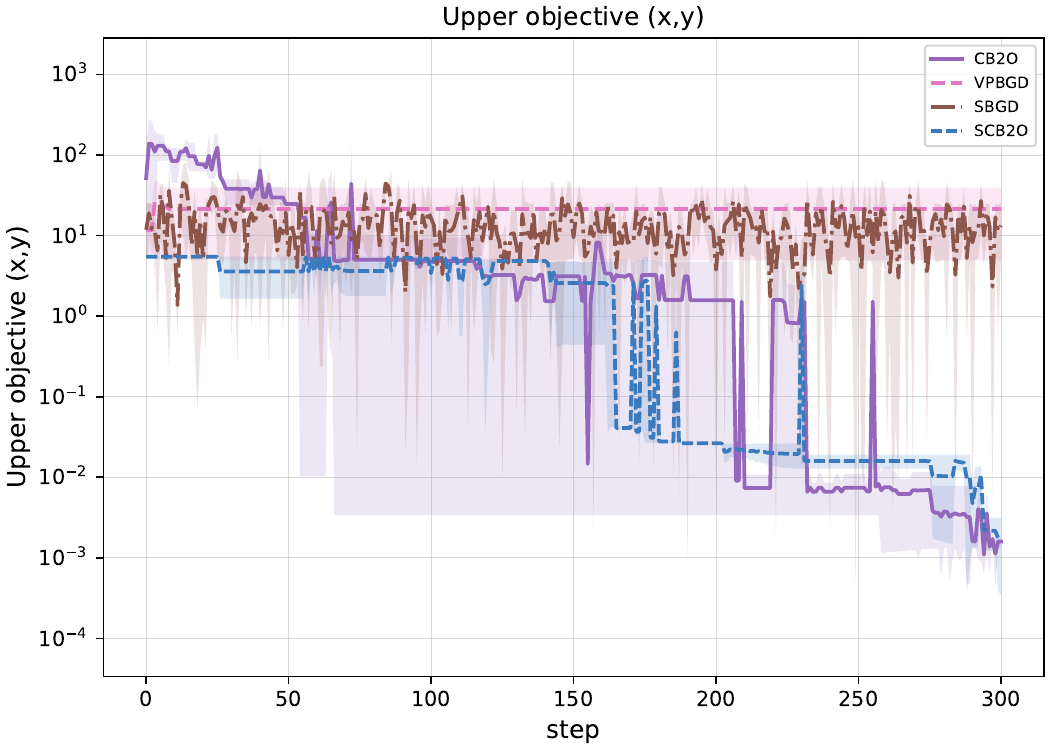}
  \includegraphics[width=0.48\linewidth]{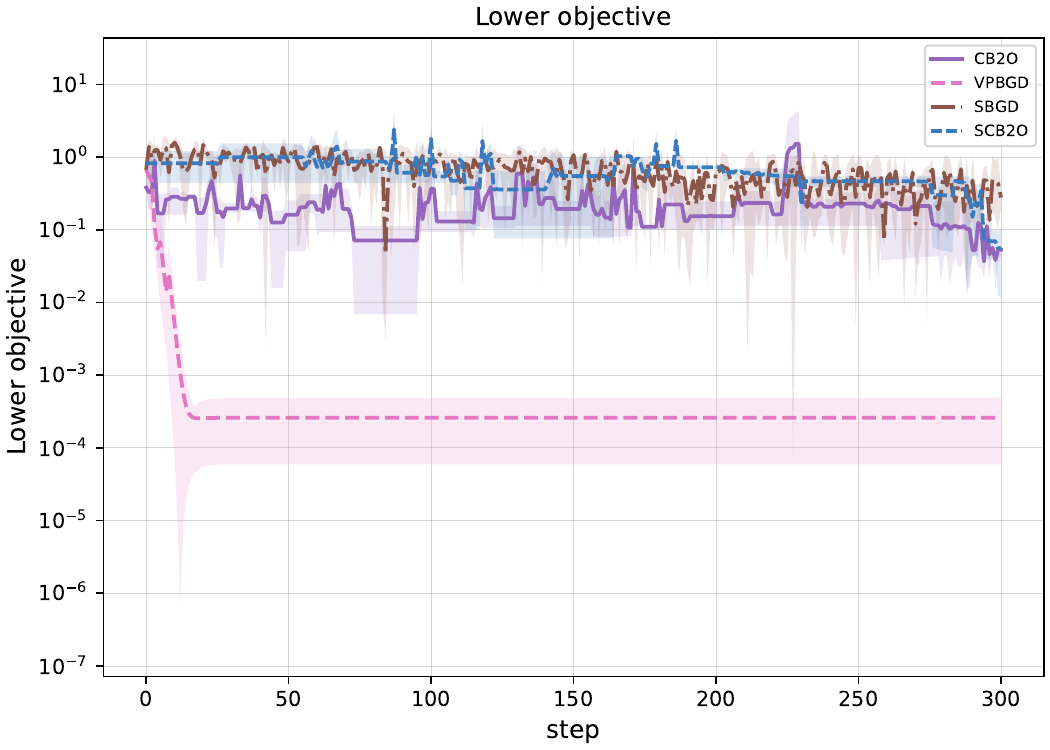}
  \caption{Left: Upper-level objective value. Right: Lower-level objective value. }\label{fig:gd3-lower}
\end{figure}

\begin{small}
\begin{table}[h!]
\centering
\caption{Final metrics after 300 iterations. Values are mean \(\pm\) standard deviation
over seeds \(\{0,1,2\}\); for SCB$^2$O, one non-finite seed is omitted.}
\label{tab:gd3-nonpl-quadratic-results}
\begin{tabular}{lcc}
\toprule
Method
  & upper
  & lower \\
\midrule
CB$^2$O  & \(0.0016\pm0.0002\) & \(0.0532\pm0.0081\) \\
VPBGD  & \(21.3581\pm14.1328\) & \(0.0003\pm0.0002\) \\
SBGD   & \(13.3005\pm8.6923\) & \(0.2783\pm0.0741\)\\
SCB$^2$O & \(0.0017\pm0.0014\) & \(0.0556\pm0.0442\) \\
\bottomrule
\end{tabular}
\end{table}
\end{small}

\subsection{2D Constrained Problem}
\begin{small}
\begin{table}[h]
\centering
\caption{2D constrained benchmark settings from \texttt{configs/benchmark\_2d.json}.}\label{tab:2d-settings}
\begin{tabular}{ll}
\toprule
Parameter & Value \\
\midrule
Problems & Circ, Star \\
Seeds & 0--4 (5 runs) \\
Steps & 600 \\
Particles $N$ & 100 \\
$\alpha$ & 30 \\
$\beta$ & 0.05 \\
$\lambda$ & 1.0 \\
$\sigma$ & 1.0 \\
SCB$^2$O $\xi$ sweep & 10, 100, 1000, 10000 \\
\bottomrule
\end{tabular}
\end{table}
\end{small}

\begin{figure}[h]
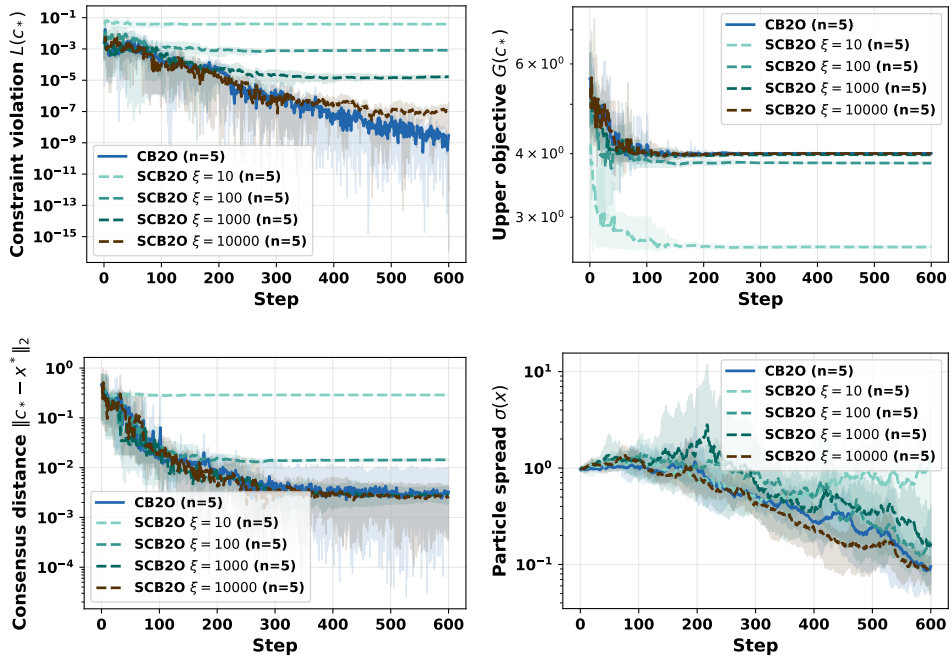

  \centering
  \includegraphics[width=0.45\textwidth, trim=0 0 0 40, clip]{figures/benchmark2d_Circ_lower_val.pdf}%
  \
  \includegraphics[width=0.45\linewidth, trim=0 0 0 40, clip]{figures/benchmark2d_Circ_upper_val.pdf}\\[4pt]
  \includegraphics[width=0.45\linewidth, trim=0 0 0 40, clip]{figures/benchmark2d_Circ_consensus_dist.pdf}\
  \includegraphics[width=0.45\linewidth, trim=0 0 0 40, clip]{figures/benchmark2d_Circ_particle_std.pdf}
  \caption{CB$^2$O vs SCB$^2$O on the circle constraint. All metrics use log scale. 
           }
\end{figure}

\begin{figure}[h]
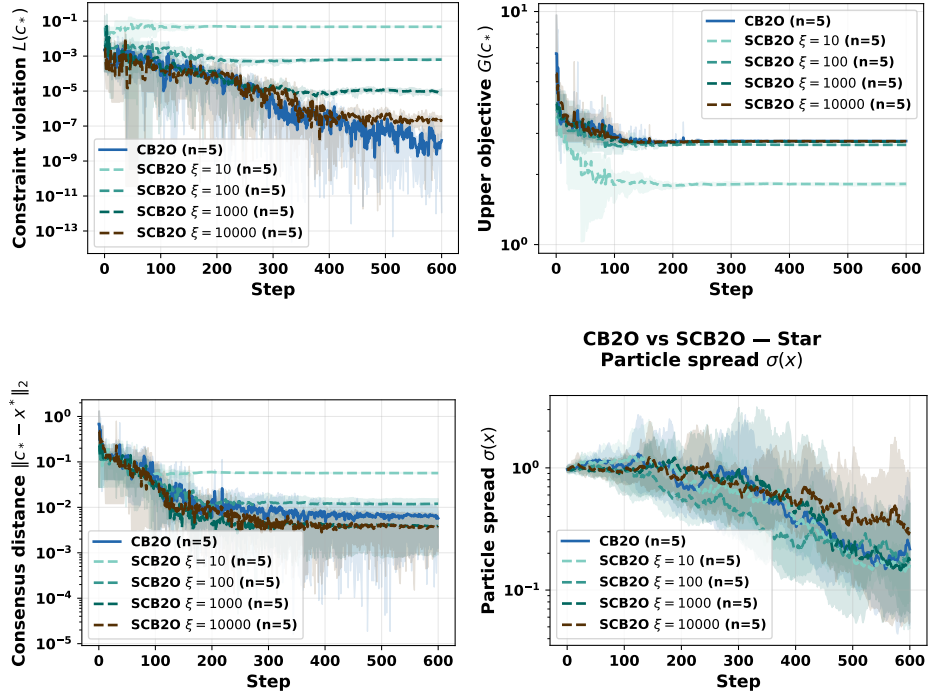

  \centering
  \includegraphics[width=0.44\linewidth, trim=0 0 0 40, clip]{figures/benchmark2d_Star_lower_val.pdf}%
  \includegraphics[width=0.44\linewidth, trim=0 0 0 40, clip]{figures/benchmark2d_Star_upper_val.pdf}\\[4pt]
  \includegraphics[width=0.44\linewidth, trim=0 0 0 40, clip]{figures/benchmark2d_Star_consensus_dist.pdf}\
  \includegraphics[width=0.44\linewidth]{figures/benchmark2d_Star_particle_std.pdf}
  \caption{CB$^2$O vs SCB$^2$O on the star-shaped constraint. All metrics use log scale. 
           }
\end{figure}

Next table shows the quantitative summary of the results after 600 iterations.
\begin{small}
\begin{table}[h]
\centering
\caption{Mean metrics at step 600 (averaged over 5 seeds).
         $\dagger$: upper gap clamped at $10^{-12}$ indicates consensus not yet on the constraint.}
\label{tab:2d-results}
\begin{tabular}{llcccc}
\toprule
Problem & Algorithm & $L(c_\star)$ & $G(c_\star)-G^\star$ & $\|c_\star-\theta^\star\|_2$ & $\sigma(x)$ \\
\midrule
\multirow{5}{*}{Circ}
 & CB$^2$O            & $2.9\times10^{-9}$  & $2.5\times10^{-3}$ & $3.2\times10^{-3}$ & $9.6\times10^{-2}$ \\
 & SCB$^2$O $\xi=10$    & $3.9\times10^{-2}$  & $\dagger$          & $2.9\times10^{-1}$ & $1.2\times10^{0}$  \\
 & SCB$^2$O $\xi=100$   & $8.3\times10^{-4}$  & $\dagger$          & $1.4\times10^{-2}$ & $1.7\times10^{-1}$ \\
 & SCB$^2$O $\xi=1000$  & $1.7\times10^{-5}$  & $\dagger$          & $2.8\times10^{-3}$ & $1.6\times10^{-1}$ \\
 & SCB$^2$O $\xi=10000$ & $1.3\times10^{-7}$  & $3.5\times10^{-5}$ & $2.6\times10^{-3}$ & $8.7\times10^{-2}$ \\
\midrule
\multirow{5}{*}{Star}
 & CB$^2$O            & $1.5\times10^{-8}$  & $2.9\times10^{-3}$ & $5.8\times10^{-3}$ & $2.2\times10^{-1}$ \\
 & SCB$^2$O $\xi=10$    & $4.7\times10^{-2}$  & $\dagger$          & $5.7\times10^{-2}$ & $1.6\times10^{-1}$ \\
 & SCB$^2$O $\xi=100$   & $6.3\times10^{-4}$  & $\dagger$          & $1.2\times10^{-2}$ & $1.8\times10^{-1}$ \\
 & SCB$^2$O $\xi=1000$  & $9.6\times10^{-6}$  & $\dagger$          & $3.9\times10^{-3}$ & $1.7\times10^{-1}$ \\
 & SCB$^2$O $\xi=10000$ & $2.4\times10^{-7}$  & $4.1\times10^{-4}$ & $3.6\times10^{-3}$ & $3.0\times10^{-1}$ \\
\bottomrule
\end{tabular}
\end{table}
\end{small}

\subsection{MNIST Experiment}

\begin{table}[h]
\centering
\caption{MNIST benchmark settings from \texttt{configs/benchmark\_mnist.json}.}
\label{tab:mnist-settings}
\begin{tabular}{ll}
\toprule
Parameter & Value \\
\midrule
Model & CNN with bias \\
Dataset & MNIST (10\,000 training samples) \\
Batch size & 60 \\
Particles $N$ & 50 \\
Mini-batch evaluations $M$ & 100 \\
Epochs $T$ & 100 \\
$\alpha$ & 50 \\
$\lambda_1$ & 1.0 \\
$\sigma_1$ & 0.632 \\
$\beta$ minimum & 0.02 \\
$\beta$ decay factor & 1.0 \\
$\gamma$ & 0.1 \\
Noise type & anisotropic \\
$p$-norm & 0.5 \\
Seeds & 0--4 (5 runs) \\
\midrule
$\beta$ sweep & 0.04, 0.06, 0.08, 0.10 \\
SCB$^2$O $\xi$ sweep & 1, 5, 10, 100, 1000 \\
\bottomrule
\end{tabular}
\end{table}

The rest of the resulting plots are presented here. 
\begin{figure}[h]
  \centering
  \includegraphics[width=\linewidth, trim=0 0 0 40, clip]{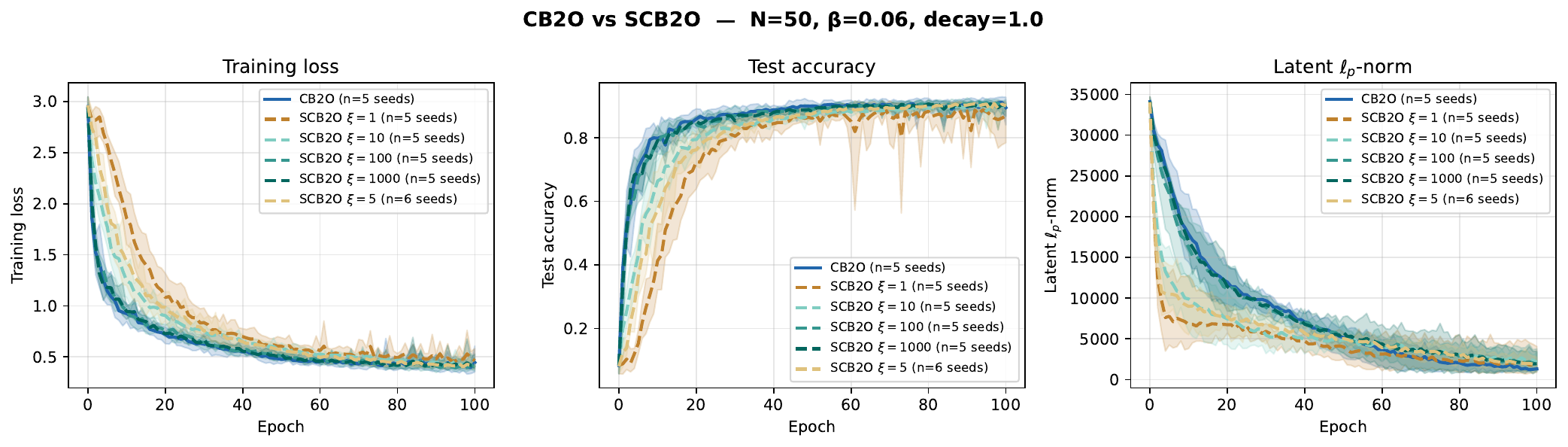}
  \caption{MNIST training curves: $N=50$, $\beta=0.06$.}
  \label{fig:mnist-b006}
\end{figure}

\begin{figure}[h]
  \centering
  \includegraphics[width=\linewidth, trim=0 0 0 40, clip]{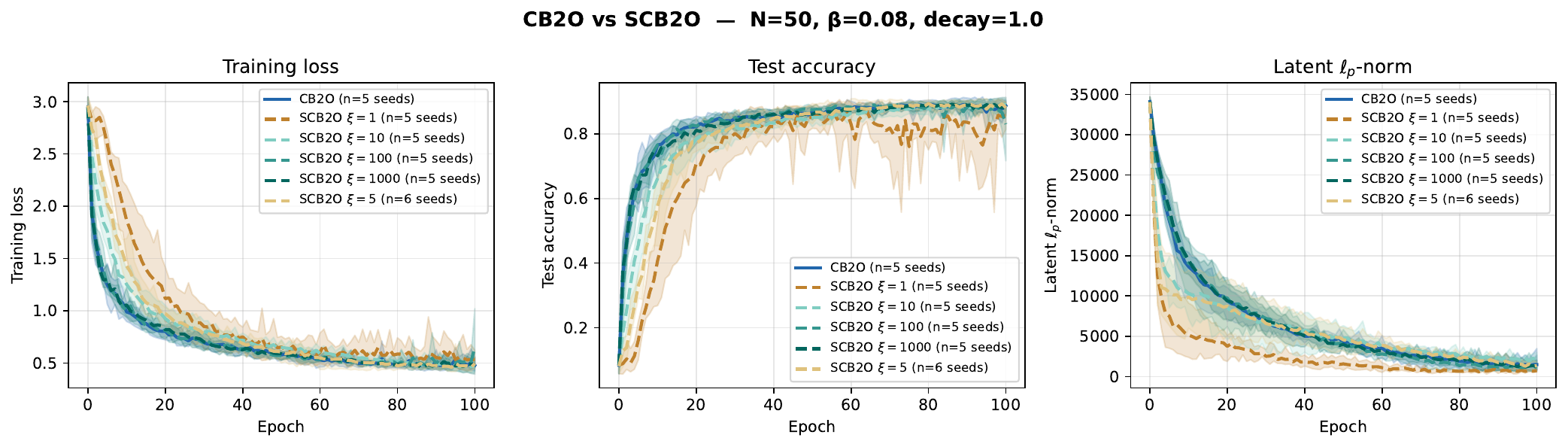}
  \caption{MNIST training curves: $N=50$, $\beta=0.08$.}
  \label{fig:mnist-b008}
\end{figure}

\begin{figure}[h]
  \centering
  \includegraphics[width=\linewidth, trim=0 0 0 40, clip]{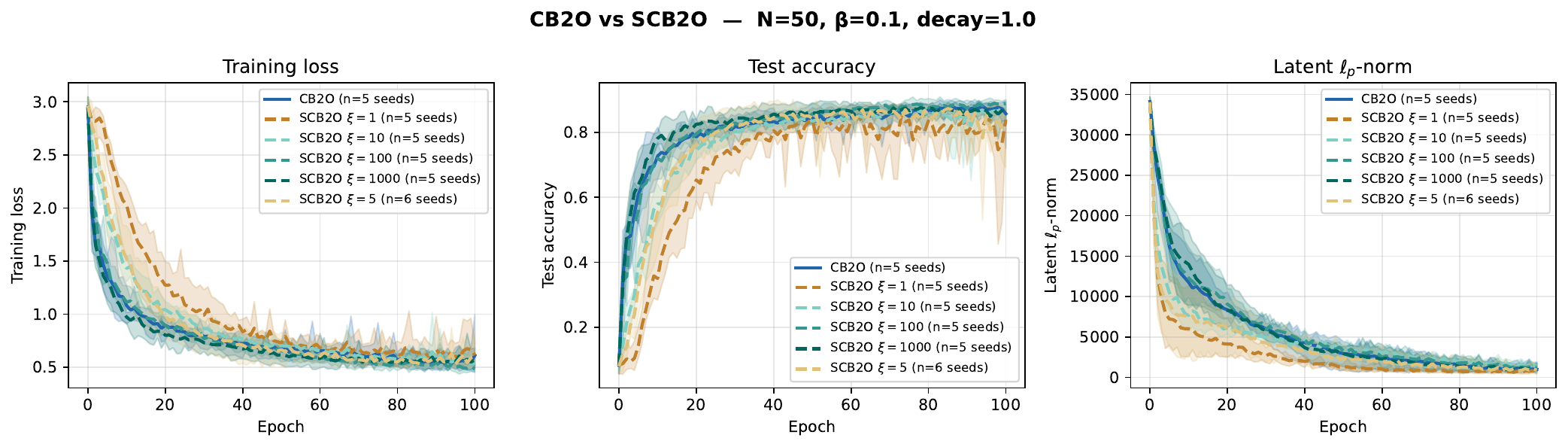}
  \caption{MNIST training curves: $N=50$, $\beta=0.10$.}
  \label{fig:mnist-b010}
\end{figure}

\textit{Varying-$\beta$ Comparison.} 
Figures~\ref{fig:vbeta-1}--\ref{fig:vbeta-1000} show how final-epoch
performance varies with $\beta$ for all five $\xi$ values.

\begin{figure}
  \centering
  \begin{subfigure}[b]{0.98\linewidth}
    \includegraphics[width=\linewidth, trim=0 0 0 30, clip]{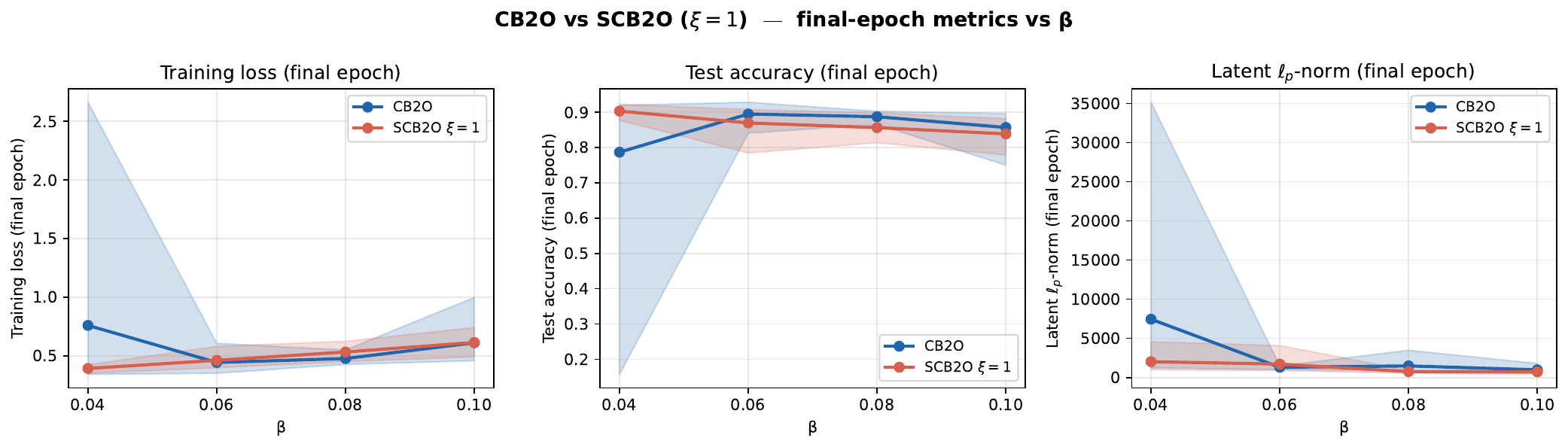}
    \caption{$\xi=1$}
    \label{fig:vbeta-1}
  \end{subfigure}\hfill
  \begin{subfigure}[b]{0.98\linewidth}
    \includegraphics[width=\linewidth, trim=0 0 0 30, clip]{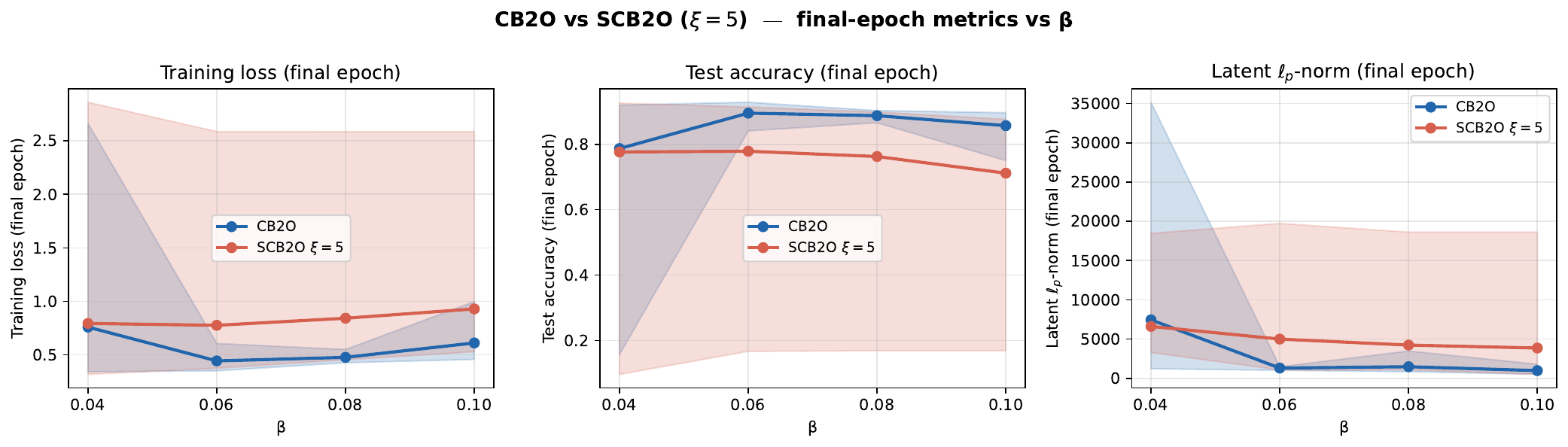}
    \caption{$\xi=5$}
    \label{fig:vbeta-5}
  \end{subfigure}

  \begin{subfigure}[b]{0.98\linewidth}
    \includegraphics[width=\linewidth, trim=0 0 0 30, clip]{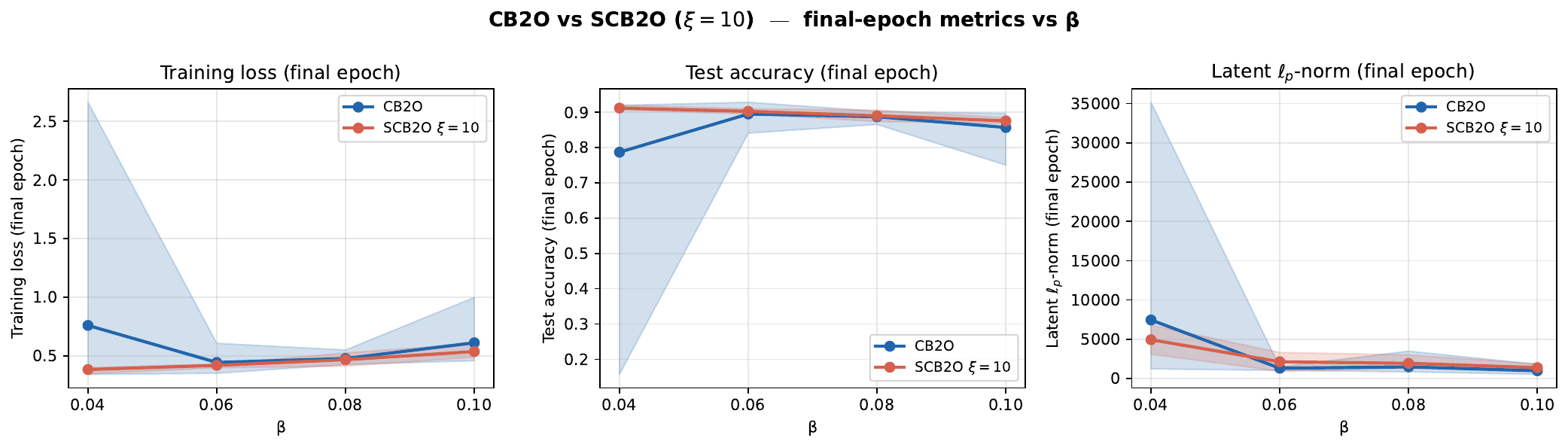}
    \caption{$\xi=10$}
    \label{fig:vbeta-10}
  \end{subfigure}\hfill
  \begin{subfigure}[b]{0.98\linewidth}
    \includegraphics[width=\linewidth, trim=0 0 0 30, clip]{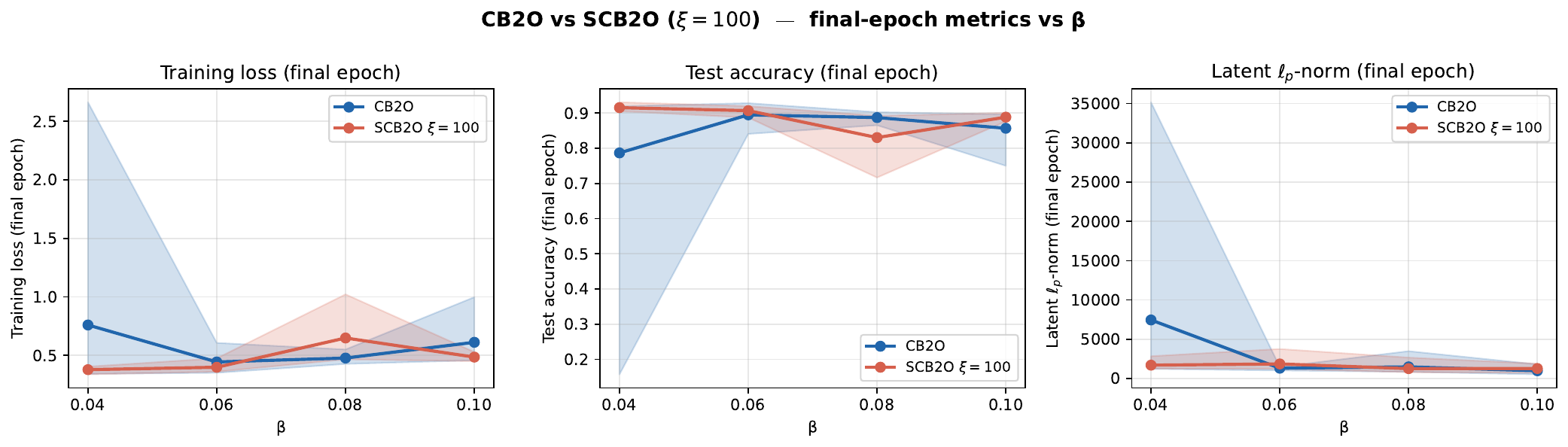}
    \caption{$\xi=100$}
    \label{fig:vbeta-100}
  \end{subfigure}

  \begin{subfigure}[b]{0.98\linewidth}
    \includegraphics[width=\linewidth, trim=0 0 0 30, clip]{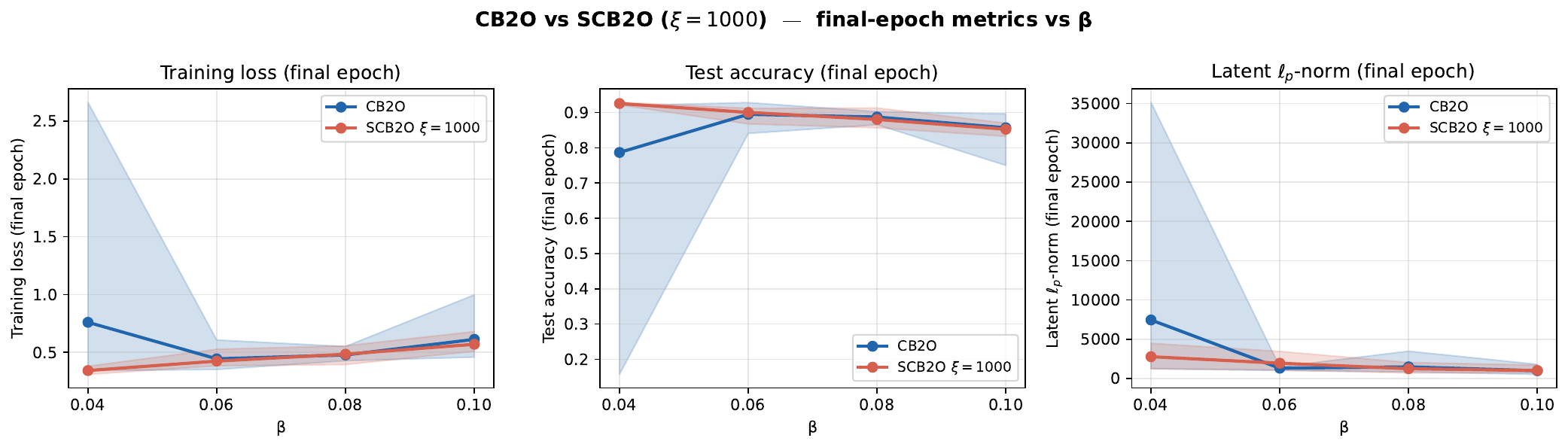}
    \caption{$\xi=1000$}
    \label{fig:vbeta-1000}
  \end{subfigure}
  \caption{Final-epoch metrics vs.\ $\beta$ for CB$^2$O and SCB$^2$O across all
           $\xi$ values.  Shaded band = min--max across seeds.}
  \label{fig:mnist-varying-beta}
\end{figure}

%%%%%%%%%%%%%%%%%%%%%%%%%%%%%%%%%%%%%%%%%%%%%%%%%%%%%%%%%%%%

\clearpage
%\subimport{./}{checklist}

\end{document}